\documentclass[a4paper,11pt,reqno]{article}
\usepackage{latexsym,amsmath,amssymb,amsfonts,amsthm}
\usepackage{mathrsfs}
\usepackage{a4wide}
\usepackage[usenames, dvipsnames]{xcolor}
\usepackage{changes}
\usepackage[colorlinks=true, linkcolor=blue, citecolor=ForestGreen]{hyperref}

\newcounter{margcount}
\newcommand{\xupref}[2]{\hspace{-0.3ex}\stackrel{\eqref{#1}}{#2}} 

\begingroup
\newtheorem{theorem}{Theorem}[section]
\newtheorem{lemma}[theorem]{Lemma}
\newtheorem{proposition}[theorem]{Proposition}

\endgroup
\begingroup
\newtheorem{definition}[theorem]{Definition}
\endgroup
\newtheorem{remark}[theorem]{Remark}

\numberwithin{equation}{section}

\newcommand{\e}{\varepsilon}
\newcommand{\vphi}{\varphi}
\newcommand{\N}{\mathbb N}
\newcommand{\Z}{\mathbb Z}
\newcommand{\R}{\mathbb R}

\newcommand{\de}{\,\mathrm{d}}

\newcommand{\supp}{{\rm supp\,}}

\newcommand{\measg}{\mathcal{M}_+(\R)}
\newcommand{\measgbeta}{\mathcal{M}_+(\R; 1 + 2^{(\beta+1) x})}

\newcommand{\cc}{C_{\mathrm c}}
\newcommand{\coag}{\mathscr{C}}
\newcommand{\frag}{\mathscr{F}}

\title{Solutions with peaks for a coagulation-fragmentation equation. Part II: aggregation in peaks}
\author{Marco Bonacini\thanks{\emailmarco} \and Barbara Niethammer\thanks{\emailbarbara} \and Juan J. L. Vel\'{a}zquez\thanks{\emailjuan}}
\date{${}^*$\addmarco \\ ${}^\dag$$^\ddag$\uniadd \\[3ex]\today}

\newcommand{\email}[1]{E-mail: \tt #1}
\newcommand{\emailmarco}{\email{marco.bonacini@unitn.it}}
\newcommand{\emailbarbara}{\email{niethammer@iam.uni-bonn.de}}
\newcommand{\emailjuan}{\email{velazquez@iam.uni-bonn.de}}
\newcommand{\addmarco}{\emph{\small Department of Mathematics, University of Trento\\ Via Sommarive 14, 38123 Povo (TN), Italy}}
\newcommand{\uniadd}{\emph{\small University of Bonn, Institute for Applied Mathematics\\ Endenicher Allee 60, 53115 Bonn, Germany}}

\begin{document}

\maketitle

\begin{abstract}
The aim of this two-part paper is to investigate the stability properties of a special class of solutions to a coagulation-fragmentation equation. We assume
that the coagulation kernel is close to the diagonal kernel, and that the fragmentation kernel is diagonal. In a companion paper we constructed a two-parameter family of stationary solutions concentrated in Dirac masses, and we carefully studied the asymptotic decay of the tails of these solutions, showing that this behaviour is stable. In this paper we prove that for initial data which are sufficiently concentrated, the corresponding solutions approach one of these stationary solutions for large times.
\end{abstract}



\section{Introduction} \label{sect:intro}

The aim of this two-part paper is to investigate the stability properties of a special class of solutions to a coagulation-fragmentation model.
We consider the evolution equation
\begin{equation} \label{eq:coagfrag}
\begin{split}
\partial_t f(\xi,t) = \coag[f](\xi,t) + \frag[f](\xi,t)\,,
\end{split}
\end{equation}
where the coagulation operator and the fragmentation operator are respectively defined as
\begin{equation} \label{coag}
\coag[f](\xi,t) := \frac12\int_0^{\xi} K(\xi-\eta,\eta)f(\xi-\eta,t)f(\eta,t)\de \eta - \int_0^\infty K(\xi,\eta)f(\xi,t)f(\eta,t)\de \eta\,,
\end{equation}
\begin{equation} \label{frag}
\frag[f](\xi,t) := \int_0^\infty \Gamma(\xi+\eta,\eta)f(\xi+\eta,t)\de\eta - \frac12\int_0^\xi \Gamma(\xi,\eta)f(\xi,t)\de\eta\,.
\end{equation}
We consider a coagulation kernel $K$ compactly supported around the diagonal $\{\xi=\eta\}$, and a diagonal fragmentation kernel $\Gamma$ (see Section~\ref{sect:setting} for the precise assumptions).
In a companion paper \cite{BNVd} we started the investigation of the stability properties of a family of stationary solutions to \eqref{eq:coagfrag} with peaks
concentrated in Dirac masses; the goal of this paper is to show for a class of initial data such that the corresponding solutions to the evolution equation \eqref{eq:coagfrag} 
converge for large times to one of these stationary solutions.

The coagulation-fragmentation with such kernels is thus an example of nonuniqueness of stationary solutions with given mass. In addition, even though a detailed balance 
condition is satisfied, one cannot exploit a corresponding entropy as in \cite{LM03,Can07}. Another motivation for our study is that for the pure coagulation equation 
with kernels that concentrate near the diagonal there is evidence of evolution into time-periodic peak solutions \cite{HNV16} and we expect that the techniques
developed here will also be useful for a corresponding study. 
We refer to the introduction of \cite{BNVd} for more detailed motivations, bibliographical references and related questions, and to \cite{BLL19a,BLL19b} for 
general background on coagulation-fragmentation equations. We pass now to describe the main result proved in this paper.

It is shown in \cite{BNVd} (see also Proposition~\ref{prop:stationary} below) that, given any value of the total mass $M>0$ and a shifting parameter $\rho\in[0,1)$, there exists a stationary (measure) solution to \eqref{eq:coagfrag} in the form
\begin{equation} \label{intro1}
f_p(\xi;M,\rho) = \sum_{n=-\infty}^\infty f_n(M,\rho)\delta(\xi-2^{n+\rho}),
\end{equation}
with total mass
\begin{equation} \label{intro2}
\int_0^\infty \xi f_p(\xi;M,\rho) \de\xi = \sum_{n=-\infty}^\infty 2^{n+\rho}f_n(M,\rho) = M.
\end{equation}
The measure $f_p(\cdot\,;M,\rho)$ is concentrated in the discrete set $\{2^{n+\rho}\}_{n\in\Z}$.
In the main result of this paper (Theorem~\ref{thm:stability}) we show that, for a class of initial data compactly supported around the points $\{2^{n}\}_{n\in\Z}$, a solution to \eqref{eq:coagfrag} converges as $t\to\infty$ to one of the discrete measures $f_p(\cdot\,;M,\rho)$, for some $\rho\in[0,1)$, with the same mass $M$ of the initial datum.

The class of initial data for which this stability result holds is determined in terms of the asymptotic behaviour as $\xi\to\infty$. We consider an initial datum $f_0\in\mathcal{M}^+(0,\infty)$ with total mass $M=\int_0^\infty \xi f_0(\xi)\de\xi$ such that
\begin{equation} \label{intro3}
\supp f_0 \subset \bigcup_{n\in\Z}(2^{n-\delta_0},2^{n+\delta_0})
\end{equation}
for some $\delta_0>0$ sufficiently small. Secondly, by introducing the quantities 
\begin{equation}  \label{intro5}
m_n(0) := \int_{(2^{n-\delta_0},2^{n+\delta_0})} f_0(\xi)\de\xi
\end{equation}
representing the number of particles located around the point $2^n$, we assume that the sequence $\{m_n(0)\}_{n\in\Z}$ is a small perturbation of the coefficients of one of the stationary states \eqref{intro1} (not necessarily the one with the same mass), in the sense that
\begin{equation} \label{intro4}
m_n(0) = (1+\e_n^0) f_n(M^0,\rho^0) , \qquad n\in\Z,
\end{equation}
for some $M^0>0$, $\rho^0>0$, and for a sequence $|\e_n^0|\leq\delta_0$. Then our main result can be stated in the following form: there exists $\delta_0>0$ small enough, depending ultimately only on the total mass $M>0$ of the initial datum, such that if $f_0$ satisfies \eqref{intro3}--\eqref{intro4} then there exists a (weak) solution to \eqref{eq:coagfrag} with initial datum $f_0$ which converges, as $t\to\infty$, to a stationary solution $f_p(\cdot\,;M,\rho)$, where $M$ is the mass of $f_0$ and $\rho\in[0,1)$.

This stability property results from the combination of two main effects. We introduce the first and second moments around each peak of the solution $f(\xi,t)$ at time $t$:
\begin{equation*}
\begin{split}
p_n(t) &:= \frac{1}{m_n(t)}\int_{2^{n-\delta_0}}^{2^{n+\delta_0}} \frac{\ln(\xi/2^n)}{\ln2}f(\xi,t)\de\xi,\\
q_n(t) &:= \frac{1}{m_n(t)}\int_{2^{n-\delta_0}}^{2^{n+\delta_0}} \Bigl( \frac{\ln(\xi/2^n)}{\ln 2}-p_n(t) \Bigr)^2 f(\xi,t)\de\xi
\end{split}
\end{equation*}
(where $m_n(t)$ is defined as in \eqref{intro5} with $f_0$ replaced by $f(\cdot,t)$). Notice that, for a solution concentrated in peaks in the form \eqref{intro1}, one has $p_n=\rho$, $q_n=0$ for all $n\in\Z$. We show that all the first moments $p_n(t)$ tend to align, as $t\to\infty$, to a common value $\rho\in[0,1)$, which describes the asymptotic position of the peaks. Furthermore, the second moments $q_n(t)$ converge exponentially to 0 as $t\to\infty$, for every $n\in\Z$: this yields concentration in peaks.

This asymptotic behaviour of the functions $p_n(t)$, $q_n(t)$ can be obtained by a careful analysis of the corresponding evolution equations. In turn, this relies on a representation of the functions $m_n(t)$, at each time, as a perturbation of a stationary state: indeed by a fixed point argument we show that an identity in the form \eqref{intro4} holds for every positive time $t$, with $M^0$ replaced by a value $M(t)$ which converges to the mass $M$ of the solution as $t\to\infty$. We refer to the beginning of Section~\ref{sect:strategy} for a more detailed discussion of the general strategy of the proof.

We conclude by noting that the result obtained in \cite{BNVd} can be seen as a particular case of the stability theorem proved in this paper: it corresponds to the case in which the initial datum is already supported at the points $\{2^{n+\rho}\}_{n\in\Z}$, and therefore all the shifting coefficients $p_n(t)$ are constant and equal to $\rho$, and all the variances $q_n(t)$ vanish identically. Hence the proof in \cite{BNVd} can be used also as a guide to get an insight of the main strategy, which is here technically more involved due to the presence of a dispersion around the peaks and of a not uniform shifting.

\medskip
\noindent\textbf{Structure of the paper.} The paper is organized as follows. In Section~\ref{sect:setting} we formulate the precise assumptions on the coagulation and fragmentation kernels, we recall from \cite{BNVd} the construction of the family of stationary solutions in the form \eqref{intro1}, and we state the main result of the paper. In Section~\ref{sect:strategy} we discuss the strategy of the proof and we introduce some auxiliary results. In Section~\ref{sect:linear} we state the regularity result on the linearized equation, whose proof is postponed to Appendix~\ref{sect:appendix}. Finally, in Section~\ref{sect:proof} we give the proof of the main result of this paper.


\section{Setting and main result} \label{sect:setting}

\subsection{Assumptions on the kernels} \label{subsect:kernel}
We assume the following: the coagulation kernel is supported near the diagonal and has the form
\begin{equation} \label{kernel1}
K(\xi,\eta) = \frac{1}{\xi+\eta}k\Bigl(\frac{\xi+\eta}{2}\Bigr) Q\Bigl(\frac{2\eta}{\xi+\eta}-1\Bigr)\,,
\end{equation}
where $k\in C^2((0,\infty))$, $k>0$, satisfies the growth conditions
\begin{align}
\qquad\qquad\qquad&k(\xi)\sim\xi^{\alpha+1} & &\text{as }\xi\to\infty, \quad\alpha\in(0,1), & \qquad\qquad\qquad \label{kernel2}\\
\qquad\qquad\qquad&k(\xi)= k_0 + O(\xi^{\bar{\alpha}}) & &\text{as }\xi\to0^+, \quad\bar{\alpha}>1, & \qquad\qquad\qquad \label{kernel2bis}
\end{align}
\begin{align}
|k'(\xi)|&\leq k_1\xi^\alpha \quad\text{for $\xi\geq1$,} \qquad |k'(\xi)|\leq k_1\xi^{\bar{\alpha}-1} \quad\text{for $\xi\leq1$,} \label{kernel2ter} \\
|k''(\xi)|&\leq k_2\xi^{\alpha-1} \quad\text{for $\xi\geq1$,} \qquad |k''(\xi)|\leq k_2\xi^{\bar{\alpha}-2} \quad\text{for $\xi\leq1$,} \label{kernel2quater}
\end{align}
for some $k_0,k_1,k_2>0$, and $Q$ is a cut-off function such that
\begin{equation} \label{kernel3}
Q\in C^2(\R), \quad Q\geq0, \quad Q(0)=1, \quad \supp Q\subset\bigl(-{\textstyle\frac13},{\textstyle\frac13}\bigr), \quad Q(\xi)=Q(-\xi).
\end{equation}
The kernel $K$ has been written in the form \eqref{kernel1} to emphasize that it is close to the diagonal kernel in the sense of measures. The condition on the support of $Q$ guarantees that, for solutions concentrated in Dirac masses at points $\{2^n\}_{n\in\Z}$, the different peaks do not interact with each other; in particular
\begin{equation} \label{suppK0}
\supp K(\xi,\eta) \subset \Bigl\{ \frac12\xi < \eta < 2\xi \Bigr\}.
\end{equation}

As for the fragmentation kernel $\Gamma(\xi,\eta)$, we assume that
\begin{equation} \label{kernel4}
\Gamma(\xi,\eta) = \gamma(\xi)\delta(\xi-2\eta)\,,
\end{equation}
where $\gamma\in C^2((0,\infty))$, $\gamma(\xi)>0$, satisfies the growth conditions
\begin{align}
\qquad\qquad\qquad&\gamma(\xi)=\xi^{\beta} + O(\xi^{\tilde{\beta}}) & &\text{as }\xi\to\infty, \quad\beta\in(1,2),\; \tilde{\beta}<\beta & \qquad\qquad\qquad \label{kernel5}\\
\qquad\qquad\qquad&\gamma(\xi)= \gamma_0 + O(\xi^{\bar{\beta}}) & &\text{as }\xi\to0^+, \quad\bar{\beta}>1, & \qquad\qquad\qquad \label{kernel5bis}
\end{align}
\begin{align}
|\gamma'(\xi)|&\leq \gamma_1\xi^{\beta-1} \quad\text{for $\xi\geq1$,} \qquad |\gamma'(\xi)|\leq \gamma_1\xi^{\bar{\beta}-1} \quad\text{for $\xi\leq1$,} \label{kernel5ter} \\
|\gamma''(\xi)|&\leq \gamma_2\xi^{\beta-2} \quad\text{for $\xi\geq1$,} \qquad |\gamma''(\xi)|\leq \gamma_2\xi^{\bar{\beta}-2} \quad\text{for $\xi\leq1$,} \label{kernel5quater}
\end{align}
for some $\gamma_0,\gamma_1,\gamma_2>0$.


\subsection{Logarithmic variables} \label{sect:variables}
It is convenient to go over to logarithmic variables, which will be used along the rest of the paper: we set
\begin{equation}\label{variables}
g(x,t) := \xi f(\xi,t), \qquad \xi=2^{x}.
\end{equation}
After an elementary change of variables, \eqref{eq:coagfrag} takes the form
\begin{equation} \label{eq:coagfrag2}
\partial_t g(x,t) = \coag[g](x,t) + \frag[g](x,t)\,,
\end{equation}
where the coagulation operator $\coag$ and the fragmentation operator $\frag$ are now given by
\begin{equation} \label{coag2}
\begin{split}
\coag[g](x,t) &:= \frac{\ln2}{2}\int_{-\infty}^{x} \frac{2^xK(2^x-2^y,2^y)}{2^x-2^y} g\Bigl(\frac{\ln(2^x-2^y)}{\ln2},t\Bigr)g(y,t)\de y \\
&\qquad - \ln2\int_{-\infty}^\infty K(2^x,2^y)g(x,t)g(y,t)\de y\,,
\end{split}
\end{equation}
\begin{equation} \label{frag2}
\begin{split}
\frag[g](x,t) &:= \ln2\int_{-\infty}^\infty \frac{2^{x+y}\Gamma(2^x+2^y,2^y)}{2^x+2^y} g\Bigl(\frac{\ln(2^x+2^y)}{\ln2},t\Bigr)\de y \\
&\qquad - \frac{\ln2}{2}\int_{-\infty}^x \Gamma(2^x,2^y)g(x,t)2^y\de y 
\end{split}
\end{equation}
and mass conservation is expressed by 
\begin{equation} \label{mass}
\int_{\R} 2^x g(x,t)\de x = \int_{\R} 2^xg(x,0)\de x \qquad\text{for all }t>0.
\end{equation}


\subsection{Weak formulation and well-posedness} \label{subsect:weaksol}
In \cite{BNVd} we introduced the following notion of weak solution in the space of positive Radon measures $g\in\measg$, that allows to consider solutions to \eqref{eq:coagfrag2} concentrated in Dirac masses. In the following, with abuse of notation, we denote by $\int_A\phi(x)g(x)\de x$ the integral of $\phi$ on $A\subset\R$ with respect to the measure $g$, also in the case that $g$ is not absolutely continuous with respect to the Lebesgue measure.

\begin{definition}[Weak solution] \label{def:weakg}
A map $g\in C([0,T];\measgbeta)$ is a \emph{weak solution} to \eqref{eq:coagfrag2} in $[0,T]$ with initial condition $g_0\in\measg$ if for every $t\in[0,T]$
\begin{equation} \label{weakg}
\begin{split}
\partial_t\biggl(\int_{\R} &g(x,t)\vphi(x)\de x \biggr) \\
& = \frac{\ln2}{2}\int_{\R}\int_{\R} K(2^y,2^z)g(y,t)g(z,t)\biggl[\vphi\Bigl(\frac{\ln(2^y+2^z)}{\ln2}\Bigr) - \vphi(y) - \vphi(z) \biggr]\de y \de z \\
& \qquad - \frac14\int_{\R} \gamma(2^{y+1}) g(y+1) \bigl[\vphi(y+1) - 2\vphi(y) \bigr]\de y
\end{split}
\end{equation}
for every test function $\vphi\in C(\R)$ such that $\lim_{x\to-\infty}\vphi(x)<\infty$ and $\vphi(x)\lesssim 2^x$ as $x\to\infty$, and $g(\cdot,0)=g_0$.
\end{definition}

Notice that, in order for the fragmentation term on the right-hand side of \eqref{weakg} to be well-defined for test functions with exponential growth at infinity, 
in view of the growth assumption \eqref{kernel5} we require in the definition the finiteness of the integral $\int_{\R}2^{(1+\beta)x}g(x)\de x$. In particular,
as the test function $\vphi(x)=2^x$ is admissible, this notion of weak solution guarantees the mass conservation property \eqref{mass}.

The existence of a (global in time) weak solution for a suitable class of initial data is proved in \cite{BNVd}. 
We report the statement here for the reader's convenience.

\begin{theorem}[Existence of weak solutions] \label{thm:wp}
Suppose that $g_0\in\measg$ satisfies
\begin{equation} \label{moment0}
\|g_0\| :=  \sup_{\substack{n\in\Z\\n< 0}} \, \frac{1}{2^n}\int_{[n,n+1)}g_0(x)\de x + \int_{[0,\infty)} g_0(x)\de x <\infty, \qquad \int_{\R} 2^{\theta x} g_0(x)\de x <\infty
\end{equation}
for some $\theta>\beta+1$. Then there exists a global weak solution $g$ to \eqref{eq:coagfrag2} with initial datum $g_0$, according to Definition~\ref{def:weakg}, which satisfies for all $T>0$
\begin{equation} \label{wpestg}
\sup_{0\leq t\leq T}\|g(\cdot,t)\| \leq C(T, g_0), \qquad
\sup_{0\leq t\leq T}\int_{\R} 2^{\theta x} g(x,t)\de x \leq C(T, g_0),
\end{equation}
where $C(T, g_0)$ denotes a constant depending on $T$, $g_0$, and on the properties of the kernels.
\end{theorem}

It is convenient to introduce a notation for the right-hand side of the weak equation \eqref{weakg}, evaluated on a given test function $\vphi$: therefore we define the operators
\begin{align} 
B_{\mathrm c}[g,g;\vphi]
& :=\frac{\ln2}{2}\int_{\R}\int_{\R} K(2^y,2^z)g(y)g(z)\biggl[\vphi\Bigl(\frac{\ln(2^y+2^z)}{\ln2}\Bigr) - \vphi(y) - \vphi(z) \biggr]\de y \de z \,, \label{rhsweakc} \\
B_{\mathrm f}[g;\vphi]
&:= \frac14\int_{\R} \gamma(2^{y+1}) g(y+1) \bigl[\vphi(y+1) - 2\vphi(y) \bigr]\de y \,.  \label{rhsweakf}
\end{align}

\begin{remark} \label{rm:supp}
Notice that, in view of \eqref{suppK0}, the coagulation kernel $K$ evaluated at the point $(2^y,2^z)$ is supported in the region $\{ |y-z|<1 \}$. In particular, there exists $\e_0\in(0,1)$, determined by the kernel $K$, such that
\begin{equation} \label{suppK}
\supp K(2^y,2^z) \subset \bigl\{ |y-z|<\e_0 \bigr\}.
\end{equation}
Furthermore, in view of the asymptotic properties \eqref{kernel2}--\eqref{kernel2bis} we have a uniform estimate
\begin{equation} \label{estK}
(2^y+2^z)K(2^y,2^z) \leq C_K (1+2^y2^z) \qquad\text{for every $y,z\in\R$, $|y-z|<1$.}
\end{equation}
\end{remark}


\subsection{Stationary solutions} \label{subsect:stationary}
In \cite{BNVd} we proved that the equation \eqref{eq:coagfrag2} admits a two-parameters family of stationary solutions supported in a set of Dirac masses at integer distance, of the form
\begin{equation}\label{peak1}
g_p(x;A,\rho) = \sum_{n=-\infty}^\infty a_n(A,\rho) \delta(x-n-\rho).
\end{equation}
The parameter $\rho\in[0,1)$ fixes the shifting of the peaks with respect to the integers, while $A>0$ characterizes the decay of the solution as $x\to\infty$ (see \eqref{peak5}). In particular, the parameter $A$ is in one-to-one correspondence with the total mass $M$ of the solution: given any $\rho\in[0,1)$ and any value of the total mass $M$, there exists a unique value $A_{M,\rho}>0$ such that the corresponding stationary solution $g_p(\cdot;A_{M,\rho},\rho)$ satisfies the mass constraint
\begin{equation} \label{peak2}
\int_{\R}2^{x}g_p(x;A_{M,\rho},\rho) \de x = \sum_{n=-\infty}^\infty 2^{n+\rho}a_n(A_{M,\rho},\rho) = M \,.
\end{equation}

By plugging the expression \eqref{peak1} into the weak formulation \eqref{weakg} of the equation, we see that $g_p$ is a stationary solution if the coefficients $a_n$ satisfy the recurrence equation
\begin{equation} \label{eq:stat}
a_{n+1} = \zeta_{n,\rho}a_n^2\,,
\qquad\text{where }
\zeta_{n,\rho} := \frac{\ln2}{2^{n+\rho}} \frac{k(2^{n+\rho})}{\gamma(2^{n+\rho+1})} \,.
\end{equation}
The existence of stationary solutions in the form \eqref{peak1} is guaranteed by the following proposition, proved in \cite{BNVd}.

\begin{proposition}[Stationary peaks solutions] \label{prop:stationary}
Let $\rho\in[0,1)$ and $A>0$ be given. There exists a unique family of coefficients $\{a_n(A,\rho)\}_{n\in\Z}$ solving \eqref{eq:stat} which are positive, bounded, and satisfy
\begin{equation} \label{peak5}
\begin{split}
a_n & = a_{-\infty}\bigl(2^n + A_02^{2n}\bigr) + o(2^{2n}) \qquad\qquad\text{as }n\to-\infty, \\
a_n & \sim a_\infty 2^{(\beta-\alpha)n}e^{-A2^n} \qquad\qquad\qquad\qquad\text{as } n\to\infty,
\end{split}
\end{equation}
where $a_{-\infty}:=\frac{\gamma_0 2^{\rho+1}}{k_0\ln2}$, $a_\infty:=(\ln2)^{-1}2^\beta2^{(\beta-\alpha)(\rho+1)}$, and $A_0$ is uniquely determined by $A$.

In particular, the measure $g_p(\cdot;A,\rho)$ defined by \eqref{peak1} is a stationary solution to \eqref{eq:coagfrag2}.
\end{proposition}

In the proof of the main result of this paper we will not directly work with the solutions to the stationary equation \eqref{eq:stat} constructed in Proposition~\ref{prop:stationary}, as we did in \cite{BNVd}, but we will need to consider the more general case in which the parameter $\rho$ in \eqref{eq:stat} actually depends on $n$. More precisely, we assume that $p=\{p_n\}_{n\in\Z}$ is a given sequence such that $|p_n|\leq\delta_0$ for some given $\delta_0\in(0,1)$, and we consider the recurrence equation
\begin{equation} \label{nearlystat1}
\bar{m}_{n+1} = \zeta_{n}(p)\bar{m}_n^2\,,
\qquad\text{where}\quad
\zeta_{n}(p) := \frac{\ln2}{2^{n+p_n}} \frac{k(2^{n+p_n})}{\gamma(2^{n+1+p_{n+1}})} \,.
\end{equation}
The values $\bar{m}_n$ represent the coefficients of a stationary solution with peaks at the points $n+p_n$, $n\in\Z$.
We generalize Proposition~\ref{prop:stationary} to the case of a nonconstant shifting $p$ in Lemma~\ref{lem:nearlystat} below.

It will be sometimes convenient to switch to new variables
\begin{equation} \label{nearlystat2}
\bar{\mu}_n := \frac12\zeta_{n}(p)\bar{m}_n\,,
\end{equation}
solving
\begin{equation} \label{nearlystat3}
\bar{\mu}_{n+1} = \theta_n(p) \bar{\mu}_n^2\,,
\qquad \text{with}\quad
\theta_n(p) := \frac{2\zeta_{n+1}(p)}{\zeta_{n}(p)}\,.
\end{equation}
We also have the relation
\begin{equation} \label{nearlystat3bis}
\frac{\bar{m}_{n+1}}{\bar{m}_n} = 2\bar{\mu}_n.
\end{equation}
In view of assumptions \eqref{kernel2}-\eqref{kernel2bis} and \eqref{kernel5}--\eqref{kernel5bis} on the kernels, one can show the following asymptotic behaviour of the coefficients:
\begin{equation} \label{nearlystat8}
\begin{split}
\zeta_{n}(p)&= \Bigl(\frac{k_0\ln2}{\gamma_0}\Bigr)2^{-(n+p_n)} + O(2^{(\bar{c}-1)n}) \qquad\text{as }n\to-\infty, \\
\zeta_{n}(p)&\sim (\ln2) 2^{\alpha(n+p_n)}2^{-\beta(n+1+p_{n+1})} \qquad\qquad\text{ as }n\to\infty,
\end{split}
\end{equation}
and
\begin{equation} \label{nearlystat9}
\begin{split}
\theta_{n}(p) &= 2^{p_n-p_{n+1}} + O(2^{\bar{c}n}) \qquad\qquad\qquad\qquad\text{as }n\to-\infty, \\
\theta_{n}(p) &\sim 2^{\alpha-\beta+1}2^{\alpha(p_{n+1}-p_n)}2^{-\beta(p_{n+2}-p_{n+1})} \qquad\text{as }n\to\infty,
\end{split}
\end{equation}
where $\bar{c}:=\min\{\bar{\alpha},\bar{\beta}\}>1$. 

\begin{lemma} \label{lem:nearlystat}
Let $A>0$ and $p=\{p_n\}_{n\in\Z}$, with $|p_n|\leq\delta_0$, be given.
Then there exists a family of positive and bounded coefficients $\{\bar{m}_n(A,p)\}_{n\in\Z}$ solving \eqref{nearlystat1} which satisfy
\begin{equation} \label{nearlystat4}
\bar{m}_n = O(2^n) \quad\text{as $n\to-\infty$,}
\qquad
\bar{m}_n = O\bigl( 2^{(\beta-\alpha)n}e^{-A2^n} \bigr) \quad\text{as $n\to\infty$.}
\end{equation}
Furthermore
\begin{equation} \label{nearlystat6}
\frac{\partial\bar{m}_n}{\partial A} = -2^n\bar{m}_n,
\qquad\qquad
\Big|\frac{\partial\bar{m}_n}{\partial p_k}\Big| \leq c2^{n-k}\bar{m}_n \quad\text{for }k\geq n,
\end{equation}
and $\frac{\partial\bar{m}_n}{\partial p_k}=0$ for $k<n$, for some constant $c>0$ depending only on the kernels.
\end{lemma}

\begin{proof}
For given $A$ and $p$, it is easily checked that the sequence $\{\bar{\mu}_n\}_{n\in\Z}$ defined by
\begin{equation} \label{nearlystat5}
\bar{\mu}_n = e^{-A2^n}\exp\Biggl( -2^n\sum_{j=n+1}^{\infty}2^{-j}\ln(\theta_{j-1}(p))\Biggr)
\end{equation}
is a solution to \eqref{nearlystat3}. In turn, defining $\bar{m}_n$ by the relation \eqref{nearlystat2}, and recalling the asymptotic behaviour \eqref{nearlystat8}--\eqref{nearlystat9}, we obtain a positive and bounded solution to \eqref{nearlystat1} with the decay \eqref{nearlystat4}. The first identity in \eqref{nearlystat6} follows directly from \eqref{nearlystat2} and \eqref{nearlystat5}. Finally by \eqref{nearlystat2} and the definitions \eqref{nearlystat1}, \eqref{nearlystat3} of $\zeta_n(p)$, $\theta_n(p)$ we have with straightforward computations
\begin{equation*}
\begin{split}
\frac{1}{\bar{m}_n}\frac{\partial\bar{m}_n}{\partial p_k}
& = \frac{1}{\bar{\mu}_n}\frac{\partial\bar{\mu}_n}{\partial p_k} - \frac{1}{\zeta_n}\frac{\partial\zeta_n}{\partial p_k}
= -2^n\sum_{j=n+1}^\infty \frac{2^{-j}}{\theta_{j-1}}\frac{\partial\theta_{j-1}}{\partial p_k} - \frac{1}{\zeta_n}\frac{\partial\zeta_n}{\partial p_k} \\
& = -2^n\sum_{j=n+1}^\infty 2^{-j} \Bigl( \frac{1}{\zeta_{j}}\frac{\partial\zeta_{j}}{\partial p_k} - \frac{1}{\zeta_{j-1}}\frac{\partial\zeta_{j-1}}{\partial p_k}\Bigr)
- \frac{1}{\zeta_n}\frac{\partial\zeta_n}{\partial p_k}
= - \sum_{j=n}^\infty 2^{n-j-1}\frac{1}{\zeta_j}\frac{\partial\zeta_j}{\partial p_k}.
\end{split}
\end{equation*}
Observing that
\begin{equation*}
\frac{1}{\zeta_j}\frac{\partial\zeta_j}{\partial p_k} =
\begin{cases}
\ln2\Bigl( \frac{k'(2^{k+p_k})2^{k+p_k}}{k(2^{k+p_k})}-1\Bigr) & \text{if }j=k,\\
-\frac{\gamma'(2^{k+p_k})2^{k+p_k}\ln2}{\gamma(2^{k+p_k})} & \text{if }j=k-1,\\
0 & \text{otherwise,}\\
\end{cases}
\end{equation*}
we finally obtain
\begin{equation} \label{nearlystat7}
\frac{1}{\bar{m}_n}\frac{\partial\bar{m}_n}{\partial p_k} =
\begin{cases}
-2^{n-k-1}\ln2\Bigl( \frac{k'(2^{k+p_k})2^{k+p_k}}{k(2^{k+p_k})}-1\Bigr) + 2^{n-k} \frac{\gamma'(2^{k+p_k})2^{k+p_k}\ln2}{\gamma(2^{k+p_k})} & \text{if } k>n,\\
-2^{n-k-1}\ln2\Bigl( \frac{k'(2^{k+p_k})2^{k+p_k}}{k(2^{k+p_k})}-1\Bigr)& \text{if } k=n,\\
0 & \text{if } k<n.\\
\end{cases}
\end{equation}
The second condition in \eqref{nearlystat6} follows thanks to the assumptions \eqref{kernel2ter}, \eqref{kernel5ter}.
\end{proof}


\subsection{Main result} \label{subsect:mainresult}
We are now in the position to state the main result of the paper. It is first convenient to introduce a notation for the following space of sequences: for $\theta\in\R$, we set
\begin{equation} \label{spacesequences}
\mathcal{Y}_\theta := \Bigl\{ y=\{y_n\}_{n\in\Z} \,:\, \|y\|_{\theta}<\infty \Bigr\},
\qquad\text{with}\quad
\|y\|_{\theta} := \sup_{n\leq0} \, 2^{n}|y_n| + \sup_{n>0} \, 2^{\theta n}|y_n|.
\end{equation}
Given an initial datum $g_0\in\measg$, we also introduce the following quantities:
\begin{align}
m_n^0 &:= \int_{[n-1/2,n+1/2)} g_0(x)\de x\,, \label{mn0}\\
p_n^0 &:= \frac{1}{m_n^0} \int_{[n-1/2,n+1/2)} g_0(x)(x-n)\de x\,, \label{pn0}\\
q_n^0 &:= \frac{1}{m_n^0} \int_{[n-1/2,n+1/2)} g_0(x)(x-n-p_n^0)^2\de x \,, \label{qn0}
\end{align}
and $m^0:=\{m_n^0\}_{n\in\Z}$, $p^0:=\{p_n^0\}_{n\in\Z}$,  $q^0:=\{q_n^0\}_{n\in\Z}$.
Our main result reads as follows.

\begin{theorem}[Stability of stationary peaks solutions] \label{thm:stability}
Given $M>0$, there exists $\delta_0>0$, depending only on $M$, with the following property.
Let $g_0\in\measg$ be an initial datum with total mass
\begin{equation} \label{mass0}
	\int_{\R} 2^x g_0(x)\de x = M,
\end{equation}
and support
\begin{equation} \label{supp0}
\supp g_0 \subset \bigcup_{n\in\Z} (n-\delta_0,n+\delta_0).
\end{equation}
Assume further that
\begin{equation} \label{m0}
m_n^0 = \bar{m}_n(A^0, p^0)(1+2^ny_n^0),
\end{equation}
for some $A^0>0$ and $y^0=\{y_n^0\}_{n\in\Z}\in\mathcal{Y}_1$ satisfying
\begin{equation} \label{A0y0}
|A^0-A_M|\leq\delta_0,\qquad \|y^0\|_1\leq\delta_0.
\end{equation}
Then there exists a weak solution $g$ to \eqref{eq:coagfrag2} with initial datum $g_0$, and $\rho\in[0,1)$, such that
\begin{equation*}
g(\cdot,t) \to g_p(\cdot\,;A_{M,\rho},\rho) \qquad\text{in the sense of measures as $t\to\infty$,}
\end{equation*}
where $g_p(\cdot\,;A_{M,\rho},\rho)$ is the stationary solution with total mass $M$ concentrated in peaks at the points $\{n+\rho\}_{n\in\Z}$, see \eqref{peak1}--\eqref{peak2}.
\end{theorem}

In the statement of the theorem, $\bar{m}_n$ are the coefficients introduced in Lemma~\ref{lem:nearlystat}, and $A_M=A_{M,0}>0$ is the unique value such that the stationary solution $g_p(\cdot;A_M,0)$ has total mass $M$ (see \eqref{peak2}). The full argument for the proof of Theorem~\ref{thm:stability} will be given in Section~\ref{sect:proof}.

\begin{remark} \label{rm:wasserstein}
The proof of Theorem~\ref{thm:stability} shows that a much stronger convergence result holds: indeed we also obtain a uniform exponential decay in time of the 2-Wasserstein distance between the (normalized) restriction of the measure $g(\cdot,t)$ to each interval $(n-\frac12,n+\frac12)$, $n\in\Z$, and the corresponding Dirac delta centered at the point $n+\rho$:
\begin{equation} \label{wasserstein}
\sup_{n\in\Z} W_2(g_n(t),\delta_\rho) \leq Ce^{-\frac{\nu}{2} t},
\qquad\text{where}\quad g_n(t):= \frac{1}{m_n(t)}g(\cdot+n,t)\chi_{(-\frac12,\frac12)}
\end{equation}
(see \eqref{mn} for the definition of $m_n(t)$), where $W_2$ denotes the 2-Wasserstein distance
$$
W_2(\mu,\nu) := \biggl( \inf_{\pi\in\Pi(\mu,\nu)}\int_{\R\times\R} |x-y|^2\de\pi(x,y) \biggr)^\frac12.
$$
Indeed denoting by $p_n(t)$ and $q_n(t)$ the first and second moment of $g(\cdot,t)$ around each peak (see \eqref{pn} and \eqref{qn}), from the definition of $W_2$ one finds
\begin{align*}
W_2^2(g_n(t),\delta_\rho)= \frac{1}{m_n(t)}\int_{(n-\frac12,n+\frac12)} |x-n-\rho|^2 g(x,t)\de x \leq 2q_n(t) + 2|p_n(t)-\rho|^2 \leq Ce^{-\nu t},
\end{align*}
where the uniform exponential convergence of $q_n(t)\to0$ and $p_n(t)\to\rho$ is obtained in the proof (see for instance \eqref{decaypn1}--\eqref{decayqn1}).
\end{remark}


\section{Strategy of the proof and auxiliary results}\label{sect:strategy}

In the following we assume that $g_0\in\measg$ is a given initial datum satisfying the assumptions of Theorem~\ref{thm:stability}: in particular, $M>0$ is a fixed quantity denoting the total mass of the solution (see \eqref{mass0}), and $g_0$ is supported in the union of intervals $\bigcup_{n\in\Z} I_n$, where $I_n:=(n-\delta_0,n+\delta_0)$ (see \eqref{supp0}), for some $\delta_0\in(0,1)$ to be chosen later. We now present the general strategy for the proof of Theorem~\ref{thm:stability}.

We will first show in Lemma~\ref{lem:supp} that the structure of the support of the solution is preserved by the evolution, that is, every weak solution $g(\cdot,t)$ starting from $g_0$ remains supported for all positive times in the union of the intervals $I_n$. In view of this property, we define for $n\in\Z$ the following quantities:
\begin{align}
m_n(t) &:= \int_{I_n} g(x,t)\de x, \label{mn}\\
p_n(t) &:= \frac{1}{m_n(t)} \int_{I_n}g(x,t)(x-n)\de x, \label{pn}\\
q_n(t) &:= \frac{1}{m_n(t)} \int_{I_n}g(x,t)(x-n-p_n(t))^2\de x. \label{qn}
\end{align}
Notice in particular that $p_n(t)\in[-\delta_0,\delta_0]$ is chosen so that
\begin{equation} \label{pn2}
\int_{I_n} g(x,t)\bigl(x-n-p_n(t)\bigr)\de x =0,
\end{equation}
and that $0\leq q_n(t) \leq 4\delta_0^2$.
The proof of Theorem~\ref{thm:stability} will be achieved by showing that $q_n(t)$ decays exponentially to zero as $t\to\infty$ (this gives concentration in peaks), and that each $p_n(t)$ aligns to a constant value $\rho$ independent of $n$ (which determines the asymptotic position of the peaks as $t\to\infty$).

The evolution equation for $m_n(t)$, which we derive below in Lemma~\ref{lem:approx} starting from the weak formulation of the equation, can be seen as a perturbation of the equation
\begin{equation*}
\frac{\de m_n}{\de t}
= \frac{\gamma(2^{n+\rho})}{4} \Bigl( \zeta_{n-1,\rho}(m_{n-1}(t))^2 - m_n(t) \Bigr) - \frac{\gamma(2^{n+1+\rho})}{2} \Bigl( \zeta_{n,\rho}(m_{n}(t))^2 - m_{n+1}(t) \Bigr) \,.
\end{equation*}
We studied this equation in detail in the companion paper \cite{BNVd}: it corresponds to the particular case in which all the variances $q_n(t)$ are identically equal to zero, and all the first moments $p_n(t)$ are equal to a constant value $\rho$. In this case, the functions $m_n(t)$ would represent the coefficients of a solution which remains concentrated for all times at the points $\{n+\rho\}_{n\in\Z}$. In the main result of \cite{BNVd} we showed that, if this solution is initially a small perturbation of one of the stationary states given by Proposition~\ref{prop:stationary}, then this property is preserved for later times. The approach is then to reproduce the argument of \cite{BNVd} adapting it to this more general setting: that is, we try to construct a solution such that $m_n(t)$ is at each time a perturbation of a stationary state. More precisely, we will show that the quantities $m_n(t)$ can be represented in the form
\begin{equation} \label{yn}
m_n(t) = \bar{m}_n(A(t),p(t))\bigl( 1 + 2^ny_n(t) \bigr),
\end{equation}
where $\bar{m}_n$ are given by Lemma~\ref{lem:nearlystat}, for suitable functions $A(t)$ and $y(t)=\{y_n(t)\}_{n\in\Z}$ to be determined via a fixed point argument (Proposition~\ref{prop:fp}).

The structure \eqref{yn} allows us to write the evolution equations for the first and second moments $p_n(t)$, $q_n(t)$ in a handier way, and to prove the desired decay of these quantities.

The key idea is that the structure of all the evolution equations for $m_n$, $p_n$, $q_n$ can be seen as a perturbation of the same linearized problem, for which we provide a full regularity theory in Section~\ref{sect:linear} (whose proof is postponed to Appendix~\ref{sect:appendix}). The most technical part of the proof consists then in proving uniform estimates on the remainder terms.

A drawback in this approach is that  in order to prove the representation \eqref{yn} we would need to assume \emph{a priori} that the expected decay of $p_n(t)$ 
and $q_n(t)$ holds, which on the other hand we can prove only by exploiting \eqref{yn} itself. To avoid the circularity of the argument, we first need to consider 
a truncated problem, in which we cut-off the tail of the solution at large distance $n>N>>1$. The advantage is that for the truncated problem we can assume that
$p_n$ and $q_n$ have the expected decay \emph{for a short interval of time}, which allows us to start the process and to show that \eqref{yn} holds for small times. 
Then by a continuation argument we can extend all the estimates globally in time. In a final step we will complete the proof by sending the truncation parameter $N$ to infinity.

In the rest of this section we prove some auxiliary results: we show the property of preservation of the support, already mentioned, and we 
compute the evolution equations for the quantities $m_n(t)$, $p_n(t)$, $q_n(t)$.


\subsection{Preservation of the support} \label{subsect:support}
We start by showing that any weak solution $g$ remains concentrated on the support of the initial datum for all positive times.

\begin{lemma} \label{lem:supp}
Assume that $g$ is a weak solution to \eqref{eq:coagfrag2} with initial datum $g_0$, according to Definition~\ref{def:weakg}. Suppose that the measure $g_0$ is supported in a union of intervals of the form $\bigcup_{n\in\Z}(n-\delta_0,n+\delta_0)$, for a given $\delta_0>0$ such that
\begin{equation} \label{delta0a}
\delta_0 < \frac{1-\e_0}{2} \,,
\end{equation}
where $\e_0$ is defined in \eqref{suppK}.
Then this condition is preserved for later times.
\end{lemma}

\begin{proof}
We write $I:=[-\delta_0,\delta_0]$ and we consider the test function
\begin{equation*}
\vphi:=\sum_{n=-\infty}^\infty \chi_I(\cdot-n)\,.
\end{equation*}
Let us split the total mass $M=\int_{\R} 2^xg(x,t)\de x$ of the solution, which is preserved along the evolution, as
\begin{equation} \label{lemsupp1}
M=M_{\mathrm{int}}(t) + M_{\mathrm{ext}}(t)\,,
\qquad\text{where }
M_{\mathrm{int}}(t) := \int_{\R} 2^{x}g(x,t)\vphi(x)\de x \,.
\end{equation}
We obtain an evolution equation for $M_{\mathrm{int}}(t)$ by using a sequence of continuous test functions approximating $\psi(x):=2^x\vphi(x)$ in the weak formulation \eqref{weakg}. Recall that the right-hand side of \eqref{weakg} is the difference of the two operators $B_{\mathrm c}$ and $B_{\mathrm f}$ introduced in \eqref{rhsweakc}--\eqref{rhsweakf}; in particular, as  $B_{\mathrm c}$ is a quadratic form we have
\begin{equation} \label{lemsupp2}
B_{\mathrm c}[g,g;\psi] = B_{\mathrm c}[g\vphi,g\vphi;\psi] + 2 B_{\mathrm c}[g\vphi, g(1-\vphi);\psi] + B_{\mathrm c}[g(1-\vphi),g(1-\vphi);\psi] \,.
\end{equation}
The measure $g\vphi$ is supported on the union of intervals $\bigcup_{n\in\Z}(n+I)$. To compute the term $B_{\mathrm c}[g\vphi,g\vphi;\psi]$ it is then sufficient to consider values $y,z\in\bigcup_{n\in\Z}(n+I)$, since otherwise the corresponding contribution would vanish.
By \eqref{suppK} we integrate on the region $\{ |y-z|<\e_0 \}$, otherwise $K(2^y,2^z)$ would vanish; then the assumption \eqref{delta0a} gives that different intervals $(n+I)$ and $(m+I)$ do not interact, and we can assume that $y,z$ always belong to the same interval $(n+I)$:
\begin{multline} \label{lemsupp6}
B_{\mathrm c}[g\vphi,g\vphi;\psi] = \frac{\ln2}{2}\sum_{n\in\Z}\int_{n+I}\int_{n+I} K(2^y,2^z) g(y,t)g(z,t) \\ \times\biggl[(2^y+2^z)\vphi\Bigl(\frac{\ln(2^y+2^z)}{\ln2}\Bigr) - 2^y\vphi(y)- 2^z\vphi(z) \biggr] \de y \de z \,.
\end{multline}
For $y,z\in(n+I)$ we have $\vphi(y)=\vphi(z)=1$. We now claim that the following implication holds:
\begin{equation} \label{claimsupp}
y,z\in (n+I) \quad\Longrightarrow\quad \frac{\ln(2^y+2^z)}{\ln2} \in (n+1+I) \,.
\end{equation}
Indeed by elementary computations
\begin{align*}
\frac{\ln(2^y+2^z)}{\ln2} = \frac{\ln\bigl(2(2^{y-1}+2^{z-1})\bigr)}{\ln2} = 1 + \frac{\ln(2^{y-1}+2^{z-1})}{\ln2}\,,
\end{align*}
and by monotonicity it is easily seen that $\frac{\ln(2^{y-1}+2^{z-1})}{\ln2}\in (n+I)$ whenever $y,z\in (n+I)$.
From the previous considerations it follows that the term in brackets in the integral \eqref{lemsupp6} vanishes identically. Therefore $B_{\mathrm c}[g\vphi,g\vphi;\psi]=0$. Similarly for the fragmentation term
\begin{equation*}
B_{\mathrm f}[g;\psi] = 
2\int_{\R}\frac{\gamma(2^{y+1})}{4} g(y+1,t)2^y \bigl[\vphi(y+1)-\vphi(y)\bigr]\de y = 0 \,.
\end{equation*}
Therefore by \eqref{lemsupp2} we have
\begin{equation} \label{lemsupp3}
\frac{\de}{\de t} M_{\mathrm{ext}}(t) = - \frac{\de}{\de t} M_{\mathrm{int}}(t)
= - 2 B_{\mathrm c}[g\vphi, g(1-\vphi);\psi] - B_{\mathrm c}[g(1-\vphi),g(1-\vphi);\psi] \,.
\end{equation}

We now bound the right-hand side of \eqref{lemsupp3} in terms of $M_{\mathrm{ext}}(t)$: using \eqref{estK},
\begin{equation*}
\begin{split}
\big|2 B_{\mathrm c}[g\vphi, & g(1-\vphi);\psi]\big| \\
& \leq \ln2\sum_{n\in\Z}\int_{[n,n+1)} \de z\, g(z,t)(1-\vphi(z)) \int_{\{ |y-z|<1\}} K(2^y,2^z)g(y,t)(2^y+2^z)\de y \\
& \leq C_K \ln2\sum_{n\in\Z}\int_{[n,n+1)} \de z\, g(z,t)(1-\vphi(z)) \int_{\{ |y-z|<1\}} (1+2^y2^z)g(y,t)\de y \\ 
& \leq C_K' \sum_{n<0}\int_{[n,n+1)} \de z\, g(z,t)(1-\vphi(z)) \int_{(n-1,n+2)} g(y,t)\de y \\ 
& \qquad + C_K' \sum_{n\geq0}\int_{[n,n+1)} \de z\, 2^zg(z,t)(1-\vphi(z)) \int_{(n-1,n+2)} 2^yg(y,t)\de y \\ 
& \leq C_K'' \|g(t)\| \sum_{n<0}\int_{[n,n+1)} 2^n g(z,t)(1-\vphi(z))\de z \\
& \qquad + C_K'' M \sum_{n\geq0}\int_{[n,n+1)} 2^{z}g(z,t)(1-\vphi(z))\de z \\
& \leq C_K''' \Bigl( \|g(t)\| + M \Bigr) M_{\mathrm{ext}}(t)
\end{split}
\end{equation*}
(where the norm $\|\cdot\|$ is defined in \eqref{moment0}).
The term $|B_{\mathrm c}[g(1-\vphi),g(1-\vphi);\psi]|$ can be bounded in a similar way.
We then deduce, from \eqref{lemsupp3} and from the uniform bound \eqref{wpestg} on the solution, that $\frac{\de}{\de t} M_{\mathrm{ext}}(t) \leq C M_{\mathrm{ext}}(t)$. Now the assumption on the support of $g_0$ implies that  $M_{\mathrm{ext}}(0)=0$. It then follows by Gr\"onwall's Lemma that $M_{\mathrm{ext}}(t)\equiv0$ for all $t>0$, that is, $g(t)$ remains supported in $\bigcup_{n\in\Z}(n+I)$.
\end{proof}


\subsection{Derivation of the moment equations} \label{subsect:moments}
We now assume that $g$ is a weak solution to \eqref{eq:coagfrag2} with initial datum $g_0$, and we derive three evolution equations for the quantities $m_n(t)$, $p_n(t)$, $q_n(t)$, introduced in \eqref{mn}, \eqref{pn}, \eqref{qn}, by using the test functions
\begin{equation*}
\vphi_n^0(x) = \chi_{I_n}(x)\,, \quad \vphi^1_n(x) := (x-n)\chi_{I_n}(x)\,, \quad \vphi^2_n(x;t) := (x-n-p_n(t))^2\chi_{I_n}(x)
\end{equation*}
in the weak formulation \eqref{weakg} of the equation. In the first case, by using the condition \eqref{suppK} and the implication \eqref{claimsupp} we find
\begin{equation} \label{mneq}
\begin{split}
\frac{\de m_n}{\de t}
&= B_{\mathrm c}[g(\cdot,t),g(\cdot,t);\vphi_n^0] - B_{\mathrm f}[g(\cdot,t);\vphi_n^0] \\
& = \frac{\ln2}{2}\int_{I_{n-1}}\int_{I_{n-1}} K(2^y,2^z)g(y)g(z)\de y\de z - \ln2\int_{I_n}\int_{I_n} K(2^y,2^z)g(y)g(z)\de y\de z \\
& \qquad - \frac14\int_{I_n}\gamma(2^y)g(y)\de y + \frac12\int_{I_{n+1}}\gamma(2^y)g(y)\de y \,.
\end{split}
\end{equation}
Similarly for $p_n(t)$ we have
\begin{equation*}
\begin{split}
\frac{\de}{\de t}\bigl(m_n p_n\bigr)
&= B_{\mathrm c}[g(\cdot,t),g(\cdot,t);\vphi_n^1] - B_{\mathrm f}[g(\cdot,t);\vphi_n^1] \\
& = \frac{\ln2}{2}\int_{I_{n-1}}\int_{I_{n-1}} K(2^y,2^z)g(y)g(z) \biggl(\frac{\ln(2^y+2^z)}{\ln 2}-n\biggr)\de y\de z \\
& \qquad - \ln2\int_{I_n}\int_{I_n} K(2^y,2^z)g(y)g(z)(y-n)\de y\de z \\
& \qquad - \frac14\int_{I_n}\gamma(2^y)g(y)(y-n)\de y + \frac12\int_{I_{n+1}}\gamma(2^y)g(y)(y-(n+1))\de y \,.
\end{split}
\end{equation*}
By writing the right-hand side as $m_n'p_n+m_np_n'$ and using the equation \eqref{mneq} for $m_n$ we therefore have
\begin{align} \label{pneq}
\frac{\de p_n}{\de t}
& = \frac{1}{m_n} \biggl[ \frac{\ln2}{2}\int_{I_{n-1}}\int_{I_{n-1}} K(2^y,2^z)g(y)g(z) \biggl(\frac{\ln(2^y+2^z)}{\ln 2}-n-p_n\biggr)\de y\de z \nonumber \\
& \qquad - \ln2\int_{I_n}\int_{I_n} K(2^y,2^z)g(y)g(z)\bigl(y-n-p_n\bigr)\de y\de z \\
& \qquad - \frac14\int_{I_n}\gamma(2^y)g(y) \bigl(y-n-p_n\bigr)\de y + \frac12\int_{I_{n+1}}\gamma(2^y)g(y) \bigl(y-(n+1)-p_n\bigr)\de y \biggr]\,. \nonumber
\end{align}
Finally we compute the evolution equation for $q_n(t)$:
\begin{equation*}
\begin{split}
\frac{\de}{\de t}\bigl(m_n q_n \bigr)
&= B_{\mathrm c}[g(\cdot,t),g(\cdot,t);\vphi_n^2(\cdot;t)] - B_{\mathrm f}[g(\cdot,t);\vphi_n^2(\cdot;t)] \\
& = \frac{\ln2}{2}\int_{I_{n-1}}\int_{I_{n-1}} K(2^y,2^z)g(y)g(z) \biggl(\frac{\ln(2^y+2^z)}{\ln 2}-n-p_n\biggr)^2\de y\de z \\
& \qquad - \ln2\int_{I_n}\int_{I_n} K(2^y,2^z)g(y)g(z) \bigl(y-n-p_n\bigr)^2 \de y\de z \\
& \qquad - \frac14\int_{I_n}\gamma(2^y)g(y)\bigl(y-n-p_n\bigr)^2\de y + \frac12\int_{I_{n+1}}\gamma(2^y)g(y)\bigl(y-(n+1)-p_n\bigr)^2\de y \,,
\end{split}
\end{equation*}
which gives, after using the equation \eqref{mneq} for $m_n'(t)$,
\begin{equation}  \label{qneq}
\begin{split}
\frac{\de q_n}{\de t}
&= \frac{1}{m_n}\biggl[
\frac{\ln2}{2}\int_{I_{n-1}}\int_{I_{n-1}} K(2^y,2^z)g(y)g(z) \biggl( \biggl(\frac{\ln(2^y+2^z)}{\ln 2}-n-p_n\biggr)^2 - q_n\biggr)\de y\de z \\
& \qquad - \ln2\int_{I_n}\int_{I_n} K(2^y,2^z)g(y)g(z) \Bigl( \bigl(y-n-p_n\bigr)^2 -q_n \Bigr) \de y\de z \\
& \qquad - \frac14\int_{I_n}\gamma(2^y)g(y) \Bigl( \bigl(y-n-p_n\bigr)^2 - q_n \Bigr) \de y \\
& \qquad + \frac12\int_{I_{n+1}}\gamma(2^y)g(y) \Bigl( \bigl(y-(n+1)-p_n\bigr)^2 - q_n \Bigr) \de y \biggr] \,.
\end{split}
\end{equation}


\subsection{Approximate equations} \label{subsect:momentsapprox}
We next identify the leading order terms in the equations \eqref{mneq}, \eqref{pneq}, \eqref{qneq}.
In the following computations we will omit the dependence on the variable $t$.

\begin{lemma} \label{lem:taylor}
For all $h\in\Z$ the following identities hold:
\begin{equation} \label{approx1a}
\int_{I_h}\int_{I_h} K(2^y,2^z)g(y)g(z) \de y\de z = \frac{k(2^{h+p_h})}{2^{h+p_h+1}} m_h^2 (1+O(q_h)),
\end{equation}
\begin{equation} \label{approx1b}
\int_{I_h} \gamma(2^y)g(y)\de y = \gamma(2^{h+p_h}) m_h (1+O(q_h)),
\end{equation}
\begin{equation} \label{approx2a}
\int_{I_h}\int_{I_h} K(2^y,2^z)g(y)g(z)(y-h-p_h)\de y\de z = \frac{k(2^{h+p_h})}{2^{h+p_h+1}} m_h^2 O(q_h),
\end{equation}
\begin{equation} \label{approx2b}
\int_{I_h} \gamma(2^y)g(y)(y-h-p_h)\de y = \gamma(2^{h+p_h})m_hO(q_h),
\end{equation}
\begin{equation} \label{approx3a}
\int_{I_h}\int_{I_h} K(2^y,2^z)g(y)g(z) \bigl( (y-h-p_h)^2-q_h \bigr) \de y\de z = \frac{k(2^{h+p_h})}{2^{h+p_h+1}} \delta_0 m_h^2O(q_h),
\end{equation}
\begin{equation} \label{approx3b}
\int_{I_h} \gamma(2^y)g(y) \bigl( (y-h-p_h)^2 -q_h \bigr) \de y = \gamma(2^{h+p_h}) \delta_0 m_hO(q_h),
\end{equation}
where the notation $f_1=O(f_2)$ means that there exists a constant $C$ (depending only on the kernels) such that $|f_1|\leq C f_2$.
\end{lemma}

\begin{proof}
By a Taylor expansion of the function $K(2^y,2^z)$ at the point $(y,z)=(h+p_h,h+p_h)$ we have, for $y,z\in I_h$,
\begin{multline} \label{taylorK}
K(2^y,2^z)
= K(2^{h+p_h},2^{h+p_h}) + \frac{\partial K}{\partial\xi}(2^{h+p_h},2^{h+p_h})2^{h+p_h}\ln2 \bigl[(y-h-p_h)+(z-h-p_h)\bigr] \\
+ O\Bigl( K(2^{h+p_h},2^{h+p_h}) \bigl[(y-h-p_h)^2 + (z-h-p_h)^2\bigr] \Bigr)
\end{multline}
(here we used the fact that the quadratic term can be controlled in terms of the function $K$ itself, thanks to the growth assumptions \eqref{kernel2ter}--\eqref{kernel2quater} on the derivatives of the kernel).
Then, inserting this expression into the integral \eqref{approx1a}, and observing that in view of \eqref{pn2} the linear term vanishes, we obtain
\begin{multline*}
\int_{I_h}\int_{I_h} K(2^y,2^z)g(y)g(z) \de y\de z
= K(2^{h+p_h},2^{h+p_h}) \int_{I_h}g(y)\de y\int_{I_h} g(z)\de z \\
+ O\biggl( K(2^{h+p_h},2^{h+p_h}) \int_{I_h}\int_{I_h}g(y)g(z) (y-h-p_h)^2 \de y\de z \biggr),
\end{multline*}
from which \eqref{approx1a} follows by using the expression of the coagulation kernel in \eqref{kernel1}.
The equalities \eqref{approx2a} and \eqref{approx3a} are obtained similarly by inserting the Taylor expansion \eqref{taylorK} into the integrals, and recalling \eqref{pn2}.

The identities \eqref{approx1b}, \eqref{approx2b}, \eqref{approx3b} can be proved by analogous arguments, using the Taylor expansion of the fragmentation kernel, for $y\in I_h$,
\begin{equation} \label{taylorgamma}
\gamma(2^y) = \gamma(2^{h+p_h}) + \gamma'(2^{h+p_h})2^{h+p_h}\ln2(y-h-p_h) + O\bigl(\gamma(2^{h+p_h})(y-h-p_h)^2\bigr)
\end{equation}
(also in this case the quadratic term is controlled in terms of the function $\gamma$ itself thanks to the growth assumptions \eqref{kernel5ter}--\eqref{kernel5quater} on the derivatives fo $\gamma$).
\end{proof}

\begin{lemma} \label{lem:approx}
The functions $m_n(t)$, $p_n(t)$, $q_n(t)$, introduced in \eqref{mn}, \eqref{pn}, and \eqref{qn} respectively, obey the following equations:
\begin{equation}\label{mneq2}
\begin{split}
\frac{\de m_n}{\de t}
&= \frac{\gamma(2^{n+p_n})}{4}\Bigl( \zeta_{n-1}(p) m_{n-1}^2(1+O(q_{n-1})) - m_n(1+O(q_n)) \Bigr) \\
& \qquad - \frac{\gamma(2^{n+1+p_{n+1}})}{2} \Bigl( \zeta_n(p) m_n^2(1+O(q_n)) - m_{n+1}(1+O(q_{n+1})) \Bigr) .
\end{split}
\end{equation}
\begin{equation} \label{pneq2}
\begin{split}
\frac{\de p_n}{\de t}
& = \frac{\ln2}{2}\frac{k(2^{n-1+p_{n-1}})}{2^{n+p_{n-1}}}\frac{m_{n-1}^2}{m_n} \Bigl( (p_{n-1}-p_n) + O(q_{n-1}) \Bigr) -\ln2\frac{k(2^{n+p_n})}{2^{n+p_n+1}}m_nO(q_n) \\
& \qquad -\frac{\gamma(2^{n+p_n})}{4}O(q_n) + \frac{\gamma(2^{n+1+p_{n+1}})}{2}\frac{m_{n+1}}{m_n} \Bigl( (p_{n+1}-p_n) + O(q_{n+1}) \Bigr),
\end{split}
\end{equation}
\begin{equation} \label{qneq2}
\begin{split}
\frac{\de q_n}{\de t}
&= \frac{\ln2}{2}\frac{k(2^{n-1+p_{n-1}})}{2^{n+p_{n-1}}}\frac{m_{n-1}^2}{m_n} \Bigl[ \Bigl(\frac12q_{n-1}-q_n\Bigr) +\delta_0O(q_{n-1}) + (p_{n-1}-p_n)^2 \Bigr]\\
& \qquad -\ln2\frac{k(2^{n+p_n})}{2^{n+p_n+1}}m_n\delta_0 O(q_n) - \frac{\gamma(2^{n+p_{n}})}{4}\delta_0O(q_n) \\
& \qquad + \frac{\gamma(2^{n+1+p_{n+1}})}{2}\frac{m_{n+1}}{m_n} \Bigl[ (q_{n+1}-q_n) + \delta_0O(q_{n+1}) + (p_{n+1}-p_n)^2 \Bigr] .
\end{split}
\end{equation}
\end{lemma}

\begin{proof}
By inserting \eqref{approx1a} and \eqref{approx1b} into \eqref{mneq} we obtain the equation
\begin{equation*}
\begin{split}
\frac{\de m_n}{\de t} &= \frac{\ln2}{2}\frac{k(2^{n-1+p_{n-1}})}{2^{n+p_{n-1}}} m_{n-1}^2(1+O(q_{n-1})) - \ln2\frac{k(2^{n+p_n})}{2^{n+p_n+1}}m_n^2(1+O(q_n))\\
& \qquad -\frac{\gamma(2^{n+p_n})}{4}m_n(1+O(q_n)) + \frac{\gamma(2^{n+1+p_{n+1}})}{2}m_{n+1}(1+O(q_{n+1})) ,
\end{split}
\end{equation*}
which can be rewritten in the form \eqref{mneq2} by using the coefficients $\zeta_n(p)$ introduced in \eqref{nearlystat1}.

To prove \eqref{pneq2}, we first observe that by Taylor expansion, for all $h\in\Z$ and $y,z\in I_{h}$
\begin{align} \label{taylorlog}
\frac{\ln(2^y+2^z)}{\ln 2}
& = \frac{1}{\ln2}\ln\Bigl( 2^{h+p_h} \Bigl[ 2+\ln2(y-h-p_h)+\ln2(z-h-p_h) \nonumber\\
& \qquad\qquad + O\bigl( (y-h-p_h)^2+(z-h-p_h)^2 \bigr) \Bigr] \Bigr) \nonumber \\
& = h + p_h + 1 + \frac{1}{2}\bigl(y-h-p_h+z-h-p_h\bigr) + O\bigl((y-h-p_h)^2+(z-h-p_h)^2\bigr) \nonumber \\
& = 1 + \frac{y+z}{2} + O\bigl((y-h-p_h)^2+(z-h-p_h)^2\bigr).
\end{align}
The first term in \eqref{pneq} becomes, using the symmetry of the kernel, \eqref{taylorlog}, \eqref{approx1a}, and \eqref{approx2a},
\begin{equation*}
\begin{split}
\int_{I_{n-1}}&\int_{I_{n-1}} K(2^y,2^z)g(y)g(z) \biggl(\frac{\ln(2^y+2^z)}{\ln 2}-n-p_n\biggr)\de y\de z \\
& = \int_{I_{n-1}}\int_{I_{n-1}} K(2^y,2^z)g(y)g(z) \Bigl( y-(n-1)-p_n + O\bigl((y-(n-1)-p_{n-1})^2\bigr) \Bigr)\de y\de z \\
& = \frac{k(2^{n-1+p_{n-1}})}{2^{n+p_{n-1}}}m_{n-1}^2 (p_{n-1}-p_n)(1+O(q_{n-1})) + \frac{k(2^{n-1+p_{n-1}})}{2^{n+p_{n-1}}}m_{n-1}^2O(q_{n-1}) .
\end{split}
\end{equation*}
For the other terms in \eqref{pneq} we can use directly \eqref{approx1b}, \eqref{approx2a}, and \eqref{approx2b}: we obtain
\begin{equation*}
\begin{split}
m_n\frac{\de p_n}{\de t}
& = \frac{\ln2}{2}\frac{k(2^{n-1+p_{n-1}})}{2^{n+p_{n-1}}}m_{n-1}^2 \Bigl( (p_{n-1}-p_n) (1+O(q_{n-1})) + O(q_{n-1}) \Bigr) \\
& \qquad -\ln2\frac{k(2^{n+p_n})}{2^{n+p_n+1}}m_n^2O(q_n) -\frac{\gamma(2^{n+p_n})}{4}m_nO(q_n) \\
& \qquad + \frac{\gamma(2^{n+1+p_{n+1}})}{2} m_{n+1} \Bigl( (p_{n+1}-p_n)(1+O(q_{n+1})) +O(q_{n+1}) \Bigr) ,
\end{split}
\end{equation*}
from which \eqref{pneq2} follows.

It remains to prove the equation \eqref{qneq2} for $q_n$, starting from \eqref{qneq}. The first term in \eqref{qneq} becomes, using \eqref{taylorlog}, the symmetry of the kernel, \eqref{approx3a}, \eqref{approx2a}, and \eqref{approx1a},
\begin{equation*}
\begin{split}
\int_{I_{n-1}}&\int_{I_{n-1}} K(2^y,2^z)g(y)g(z) \biggl( \biggl(\frac{\ln(2^y+2^z)}{\ln 2}-n-p_n\biggr)^2 - q_n\biggr)\de y\de z \\
& = \int_{I_{n-1}}\int_{I_{n-1}} K(2^y,2^z)g(y)g(z) \biggl[ \biggl( \frac{y+z}{2} - (n-1) - p_n \\
& \qquad\qquad + O\bigl( (y-(n-1)-p_{n-1})^2+(z-(n-1)-p_{n-1})^2 \bigr) \biggr)^2 - q_n\biggr] \de y\de z \\
& = \Bigl(\frac12q_{n-1}-q_n\Bigr)\int_{I_{n-1}}\int_{I_{n-1}} K(2^y,2^z)g(y)g(z)\de y\de z \\
& \qquad + \frac12\int_{I_{n-1}}\int_{I_{n-1}} K(2^y,2^z)g(y)g(z) \Bigl( (y-(n-1)-p_{n-1})^2-q_{n-1} \Bigr) \de y \de z \\
& \qquad + \frac12\int_{I_{n-1}}\int_{I_{n-1}} K(2^y,2^z)g(y)g(z)\bigl(y-(n-1)-p_{n-1}\bigr)\bigl(z-(n-1)-p_{n-1}\bigr) \de y \de z \\
& \qquad + 2(p_{n-1}-p_n)\int_{I_{n-1}}\int_{I_{n-1}} K(2^y,2^z)g(y)g(z)\bigl(y-(n-1)-p_{n-1}\bigr) \de y \de z \\
& \qquad + (p_{n-1}-p_n)^2\int_{I_{n-1}}\int_{I_{n-1}} K(2^y,2^z)g(y)g(z)\de y \de z + \frac{k(2^{n-1+p_{n-1}})}{2^{n+p_{n-1}}}m_{n-1}^2\delta_0O(q_{n-1})\\
& = \frac{k(2^{n-1+p_{n-1}})}{2^{n+p_{n-1}}}m_{n-1}^2 \biggl[ \Bigl(\frac12q_{n-1}-q_n\Bigr)(1+O(q_{n-1})) + \delta_0 O(q_{n-1}) \\
& \qquad\qquad +(p_{n-1}-p_n)O(q_{n-1}) + (p_{n-1}-p_n)^2(1+O(q_{n-1}))\biggr].
\end{split}
\end{equation*}
For the second and third term in \eqref{qneq} we can use directly \eqref{approx3a} and \eqref{approx3b}, while for the last term in \eqref{qneq} we have, by using \eqref{approx3b}, \eqref{approx1b} and \eqref{approx2b},
\begin{equation*}
\begin{split}
\frac12 & \int_{I_{n+1}}\gamma(2^y)g(y) \Bigl( \bigl(y-(n+1)-p_n\bigr)^2 - q_n \Bigr) \de y \\
& = \frac12\int_{I_{n+1}}\gamma(2^y)g(y)\Bigl( (y-(n+1)-p_n)^2-(y-(n+1)-p_{n+1})^2\Bigr)\de y \\
& \qquad + \frac12\bigl(q_{n+1}-q_n\bigr)\int_{I_{n+1}}\gamma(2^y)g(y)\de y + \frac12 \gamma(2^{n+1+p_{n+1}}) \delta_0 m_{n+1}O(q_{n+1}) \\
& = (p_{n+1}-p_n)\int_{I_{n+1}}\gamma(2^y)g(y)(y-(n+1)-p_{n+1})\de y + \frac12(p_{n+1}-p_n)^2\int_{I_{n+1}}\gamma(2^y)g(y)\de y \\
& \qquad + \frac{\gamma(2^{n+1+p_{n+1}})}{2}m_{n+1} \Bigl( (q_{n+1}-q_n)(1+O(q_{n+1})) + \delta_0 O(q_{n+1}) \Bigr) \\
& = \frac{\gamma(2^{n+1+p_{n+1}})}{2}m_{n+1} \Bigl( (p_{n+1}-p_n)O(q_{n+1}) + (p_{n+1}-p_n)^2(1+O(q_{n+1})) \\
& \qquad + (q_{n+1}-q_n)(1+O(q_{n+1})) + \delta_0 O(q_{n+1}) \Bigr).
\end{split}
\end{equation*}
The equality \eqref{qneq2} follows.
\end{proof}

Suppose now that the coefficients $m_n(t)$ can be represented as perturbations of the functions $\bar{m}_n$, according to \eqref{yn}.
By plugging the expression \eqref{yn} into \eqref{mneq2}, and using \eqref{nearlystat1}, \eqref{nearlystat3bis}, and \eqref{nearlystat6}, we find
\begin{align} \label{yneq2}
\frac{\de y_n}{\de t}
& = - \frac{1}{2^n\bar{m}_n}\frac{\partial\bar{m}_n}{\partial A}\frac{\de A}{\de t} (1+2^ny_n)
- \frac{1}{2^n}\sum_{k=-\infty}^\infty \frac{1}{\bar{m}_n}\frac{\partial\bar{m}_n}{\partial p_k}\frac{\de p_k}{\de t}(1+2^ny_n) \nonumber\\
& + \frac{\gamma(2^{n+p_n})}{2^{n+2}}\Bigl( \zeta_{n-1}(p) \frac{\bar{m}_{n-1}^2}{\bar{m}_n}(1+2^{n-1}y_{n-1})^2(1+O(q_{n-1})) - (1+2^ny_n)(1+O(q_n)) \Bigr) \nonumber\\
& - \frac{\gamma(2^{n+1+p_{n+1}})}{2^{n+1}} \Bigl( \zeta_n(p) \bar{m}_n(1+2^ny_n)^2(1+O(q_n)) - \frac{\bar{m}_{n+1}}{\bar{m}_n}(1+2^{n+1}y_{n+1})(1+O(q_{n+1})) \Bigr) \nonumber\\
& = (1+2^ny_n)\frac{\de A}{\de t} - (1+2^ny_n)\frac{1}{2^n}\sum_{k=n}^\infty \frac{1}{\bar{m}_n}\frac{\partial\bar{m}_n}{\partial p_k}\frac{\de p_k}{\de t} \nonumber\\
& + \frac{\gamma(2^{n+p_n})}{2^{n+2}} \biggl[ (1+2^{n-1}y_{n-1})^2(1+O(q_{n-1})) - (1+2^ny_n)(1+O(q_n)) \nonumber\\
& - 4\bar{\mu}_n\frac{\gamma(2^{n+1+p_{n+1}})}{\gamma(2^{n+p_n})} \Bigl( (1+2^ny_n)^2(1+O(q_n)) - (1+2^{n+1}y_{n+1})(1+O(q_{n+1})) \Bigr) \biggr] ,
\end{align}
where $\bar{\mu}_n=\bar{\mu}_n(A(t),p(t))$. Similarly, under the assumption \eqref{yn} we can write the approximate equations \eqref{pneq2}, \eqref{qneq2} for $p_n(t)$ and $q_n(t)$ in the following form:
\begin{multline} \label{pneq3}
\frac{\de p_n}{\de t}
= \frac{\gamma(2^{n+p_n})}{4} \biggl[ \frac{(1+2^{n-1}y_{n-1})^2}{(1+2^ny_n)} \Bigl( (p_{n-1}-p_n) + O(q_{n-1}) \Bigr) + O(q_n) \\
- 4\bar{\mu}_n\frac{\gamma(2^{n+1+p_{n+1}})}{\gamma(2^{n+p_n})} \biggl( \frac{(1+2^{n+1}y_{n+1})}{(1+2^ny_n)}\Bigl( (p_{n}-p_{n+1}) + O(q_{n+1}) \Bigr) + (1+2^ny_n)O(q_n) \biggr)\biggr],
\end{multline}
\begin{align} \label{qneq3}
\frac{\de q_n}{\de t}
& = \frac{\gamma(2^{n+p_n})}{4} \biggl[ \frac{(1+2^{n-1}y_{n-1})^2}{(1+2^ny_n)} \Bigl( \frac{q_{n-1}}{2}-q_n + \delta_0O(q_{n-1}) + (p_{n-1}-p_n)^2 \Bigr) + \delta_0O(q_n) \nonumber \\
& \qquad - 4\bar{\mu}_n\frac{\gamma(2^{n+1+p_{n+1}})}{\gamma(2^{n+p_n})} \frac{(1+2^{n+1}y_{n+1})}{(1+2^ny_n)}\Bigl( (q_{n}-q_{n+1}) + \delta_0O(q_{n+1}) - (p_{n+1}-p_n)^2 \Bigr) \nonumber \\
& \qquad - 4\bar{\mu}_n\frac{\gamma(2^{n+1+p_{n+1}})}{\gamma(2^{n+p_n})} (1+2^ny_n)\delta_0O(q_n) \biggr],
\end{align}
with $\bar{\mu}_n=\bar{\mu}_n(A(t),p(t))$. 
This is the most convenient form of the equations for $p_n$, $q_n$ to prove their desired decay properties.


\section{The linearized equation}\label{sect:linear}

The main evolution equations that we are going to investigate in the paper, namely \eqref{yneq2}, \eqref{pneq3}, \eqref{qneq3}, can be seen as perturbation of the same linearized problem. The proof of the main result relies therefore on the properties of solutions to the linearized equation, which in its simplest form is
\begin{equation} \label{linear1}
\frac{\de y_n}{\de t} = \frac{\gamma(2^{n+\rho})}{4} \Bigl[ y_{n-1}-y_n - \sigma_n\bigl( y_n-y_{n+1}\bigr) \Bigr] ,
\end{equation}
for given $\rho\in[0,1)$ and coefficients $\sigma_n$ with suitable asymptotic properties as $n\to\pm\infty$.
We studied this linear equation in full details in \cite[Section~5]{BNVd}, with $\sigma_n$ given by
\begin{equation} \label{linear3}
\sigma_n:=4\zeta_{n,\rho}a_n(A_{M,\rho},\rho)\frac{\gamma(2^{n+1+\rho})}{\gamma(2^{n+\rho})}
\end{equation}
(here $\zeta_{n,\rho}$ are defined in \eqref{eq:stat} and $a_n(A_{M,\rho},\rho)$ are the coefficients of a stationary solution with total mass $M$, see Proposition~\ref{prop:stationary}).
We remark that for this analysis the particular form of $\sigma_n$ is not relevant, but what matters is only their asymptotic behaviour, namely $\sigma_n\to8$ as $n\to-\infty$ and $\sigma_n=O(e^{-A_M2^n})$ as $n\to\infty$.

The goal of this section is to extend the linear theory in order to treat the more general case of small time-dependent perturbations of the coefficients in \eqref{linear1}, which appears in the context of this paper.


\subsection{The constant coefficients case}
We start by recalling the result proved in \cite{BNVd} for \eqref{linear1}. We will use in the following the notation introduced in \eqref{spacesequences} for the space of sequences $\mathcal{Y}_\theta$ and the norm $\|\cdot\|_\theta$. It is convenient to denote the linear operator on the right-hand side of \eqref{linear1}, acting on a sequence $y=\{y_n\}_{n\in\Z}$, by
\begin{equation} \label{linear4}
\mathscr{L}_n(y) := \frac{\gamma(2^{n})}{4} \Bigl[ y_{n-1}-y_n - \sigma_n\bigl( y_n-y_{n+1}\bigr) \Bigr], 
\qquad
\mathscr{L}(y) := \{ \mathscr{L}_n(y) \}_{n\in\Z}
\end{equation}
(we assume here for simplicity $\rho=0$). We also introduce a symbol for the discrete derivatives
\begin{equation}  \label{linearD}
D^+_n(y) := y_{n+1}-y_n,
\qquad
D^-_n(y) := y_n-y_{n-1},
\qquad
D^{\pm}(y) := \{ D^{\pm}_n(y) \}_{n\in\Z} \,.
\end{equation}
The following result was proved in \cite[Theorem~5.1]{BNVd}.

\begin{theorem} \label{thm:linear}
Let $\sigma_n$ be as in \eqref{linear3}, for a fixed value of $M$ (and $\rho=0$). Let also
\begin{equation*}
\theta\in(-1,\beta), \qquad \tilde{\theta}\in[\theta,\beta], \qquad\text{with }\tilde{\theta}-\theta<\beta,
\end{equation*}
be fixed parameters, and let $y^0=\{y^0_n\}_{n\in\Z}\in\mathcal{Y}_\theta$ be a given initial datum.	
Then there exists a unique solution $t\mapsto S(t)(y^0)=\{S_n(t)(y^0)\}_{n\in\Z}\in\mathcal{Y}_{{\theta}}$ to the linear problem \eqref{linear1} with initial datum $y^0$.
Furthermore there exist constants $\nu>0$ (depending only on $M$), $C_1=C_1(M,\theta,\tilde{\theta})$, and $C_2=C_2(M,\theta,\tilde{\theta})$ such that for all $t>0$
\begin{equation} \label{linear5b}
\|D^+(S(t)(y^0))\|_{\tilde{\theta}} \leq C_1 \|y^0\|_\theta \, t^{-\frac{\tilde{\theta}-\theta}{\beta}} e^{-\nu t},
\end{equation}
\begin{equation} \label{linear5}
\|S(t)(y^0)-S_{\infty}(t)(y^0)\|_{\tilde{\theta}} \leq C_2 \|y^0\|_\theta \, t^{-\frac{\tilde{\theta}-\theta}{\beta}} e^{-\nu t} \qquad\text{if $\tilde{\theta}>0$,}
\end{equation}
where $S_{\infty}(t)(y^0):=\lim_{n\to\infty}S_n(t)(y^0)$.
\end{theorem}

All the constants in the statement above (and in the rest of this section) depend also on the properties of the coagulation and fragmentation kernels; however we will not mention this dependence explicitly, as the kernels are fixed throughout the paper. As observed in \cite[Remark~A.5]{BNVd}, the constant $C_1$ in \eqref{linear5b} blows up if $\theta\to\beta$ or $\tilde{\theta}-\theta\to\beta$, while the constant $C_2$ in \eqref{linear5} explodes also if $\tilde\theta\to0$.


\subsection{The case of time-dependent coefficients} \label{subsect:linear}
We now extend the analysis of the linearized problem \eqref{linear1} to the case of time-dependent coefficients, which appears in the context of this paper. More precisely, we assume that a sequence $p(t)=\{p_n(t)\}_{n\in\Z}$ is given, with the following decay properties: for all $t>0$
\begin{equation} \label{asspn}
\begin{split}
|p_n(t)| \leq\eta_0,
\quad
\big\| D^+(p(t)) \big\|_{\bar{\theta}_1} \leq \eta_0 t^{-\bar{\theta}_1/\beta} e^{-\frac{\nu}{2} t},
\quad
\Big\| \frac{\de p(t)}{\de t} \Big\|_{\bar{\theta}_1-\beta} \leq \eta_0 \bigl( 1+t^{-\bar{\theta}_1/\beta}\bigr) e^{-\frac{\nu}{2} t}.
\end{split}
\end{equation}
Here $\eta_0\in(0,1)$ and $\nu>0$ are constants to be determined later, depending only on the given value of the mass $M>0$, and $\bar{\theta}_1\in(\beta-1,1)$ is a fixed parameter. The bounds \eqref{asspn} are the expected decay behaviour of the sequence of first moments introduced in \eqref{pn}, that will be recovered a posteriori.

We study the equation
\begin{equation} \label{slinear1}
\begin{split}
\frac{\de y_n}{\de t}
& = \frac{\gamma(2^{n+p_n(t)})}{4} \Bigl( y_{n-1}-y_n - \sigma_n(t) (y_n-y_{n+1}) \Bigr) ,
\end{split}
\end{equation}
where the coefficients $\sigma_n(t)$ are given by
\begin{equation} \label{slinear2}
\sigma_n(t) := 8\bar{\mu}_n(A_M,p(t))\frac{\gamma(2^{n+1+p_{n+1}(t)})}{\gamma(2^{n+p_n(t)})}
\end{equation}
(recall here that $\bar{\mu}_n$ are the coefficients explicitly defined in \eqref{nearlystat5}, depending on a positive parameter $A_M$ and on the sequence $p(t)$). The value of the constant $A_M>0$ is fixed throughout this section, and is uniquely determined by the mass $M>0$. The equation \eqref{slinear1} can be seen as a perturbation of \eqref{linear1}, when the shifting parameter $\rho$ is not constant but depends on the peak $n$, and on time. It is convenient to introduce a symbol for the linear operator on the right-hand side of \eqref{slinear1}, acting on a sequence $y=\{y_n\}_{n\in\Z}$ at time $t$:
\begin{equation} \label{slinear3}
\mathscr{L}_n(y;t) := \frac{\gamma(2^{n+p_n(t)})}{4} \Bigl( y_{n-1}-y_n - \sigma_n(t) (y_n-y_{n+1}) \Bigr) ,
\quad \mathscr{L}(y;t) := \{ \mathscr{L}_n(y;t) \}_{n\in\Z}.
\end{equation}
We remark that, in view of \eqref{kernel5}, \eqref{kernel5bis}, \eqref{nearlystat9}, \eqref{nearlystat5}, and \eqref{asspn} we have
\begin{equation} \label{slinear8}
\limsup_{n\to-\infty}|\sigma_n(t)-8| \leq c\eta_0, \qquad \limsup_{n\to\infty}\sigma_n(t) \leq ce^{-A_M2^n}
\end{equation}
(for a uniform constant $c>0$, depending only on the kernels). The following result extends Theorem~\ref{thm:linear} to this situation, and its technical proof is postponed to the Appendix~\ref{sect:appendix}.

\begin{theorem} \label{thm:slinear}
There exist $\eta_0\in(0,1)$ and $\nu>0$, depending only on $M$, with the following property. Let $p(t)$ be a given sequence satisfying \eqref{asspn}.
Let also
\begin{equation*}
\theta\in(-1,\beta), \qquad \tilde{\theta}\in[\theta,\beta], \qquad\text{with }\tilde{\theta}-\theta<\beta,
\end{equation*}
be fixed parameters, and let $y^0=\{y_n^0\}_{n\in\Z}\in\mathcal{Y}_\theta$ be a given initial datum.
Then, for every $t_0\geq0$, there exists a unique solution
\begin{equation} \label{slinear4}
t\mapsto T(t;t_0)(y^0) = \{T_n(t;t_0)(y^0)\}_{n\in\Z}\in\mathcal{Y}_{\theta}
\end{equation}
to \eqref{slinear1} in $t\in(t_0,\infty)$, with $T(t_0;t_0)(y^0)=y^0$.

Furthermore, there exist constants $C_1=C_1(M,\theta,\tilde{\theta})$ and $C_2=C_2(M,\theta,\tilde{\theta})$ such that
\begin{equation} \label{slinear6b}
\|D^+(T(t;t_0)(y^0))\|_{\tilde{\theta}} \leq C_1 \|y^0\|_\theta (t-t_0)^{-\frac{\tilde{\theta}-\theta}{\beta}} e^{-\nu(t-t_0)},
\end{equation}
\begin{equation} \label{slinear6}
\|T(t;t_0)(y^0)-T_{\infty}(t;t_0)(y^0)\|_{\tilde{\theta}} \leq C_2 \|y^0\|_\theta (t-t_0)^{-\frac{\tilde{\theta}-\theta}{\beta}} e^{-\nu(t-t_0)} \qquad\text{if $\tilde{\theta}>0$}
\end{equation}
for all $t>t_0$, where $T_\infty(t;t_0)(y^0) := \lim_{n\to\infty} T_n(t;t_0)(y^0)$.
\end{theorem}


\subsection{The truncated problem} \label{subsect:lineartrunc}
We eventually consider the equivalent of Theorem~\ref{thm:slinear} when we truncate the operator $\mathscr{L}_n$ for large values of $n$. More precisely, we fix a (large) $N\in\N$ and we study the equation
\begin{equation} \label{tlinear1}
\begin{split}
\frac{\de y_n}{\de t}
& = \frac{\gamma(2^{n+p_n^N(t)})}{4} \Bigl( y_{n-1}-y_n - \sigma_n^N(t) (y_n-y_{n+1}) \Bigr) \qquad\text{for }n\leq N,
\end{split}
\end{equation}
$y_n(t)=0$ for $n>N$, where we assume that $p^N(t)=\{p_n^N(t)\}_{n\in\Z}$ is a given sequence satisfying \eqref{asspn}, and the coefficients $\sigma_n^N(t)$ are given by
\begin{equation} \label{tlinear2}
\sigma^N_n(t) :=
\begin{cases}
8\bar{\mu}_n(A_M,p^N(t))\frac{\gamma(2^{n+1+p^N_{n+1}(t)})}{\gamma(2^{n+p^N_n(t)})} & \text{if $n<N$,} \\
0 & \text{if $n=N$.}
\end{cases}
\end{equation}
As in \eqref{slinear3}, we introduce a symbol for the linear operator on the right-hand side of \eqref{tlinear1}, acting on a sequence $y=\{y_n\}_{n\in\Z}$ at time $t$:
\begin{equation} \label{tlinear3}
\mathscr{L}_n^N(y;t) := \frac{\gamma(2^{n+p_n^N(t)})}{4} \Bigl( y_{n-1}-y_n - \sigma_n^N(t) (y_n-y_{n+1}) \Bigr) ,
\end{equation}
and $\mathscr{L}^N(y;t) := \{ \mathscr{L}_n^N(y;t) \}_{n\in\Z}$. We then have the following result, equivalent to Theorem~\ref{thm:slinear}.

\begin{theorem} \label{thm:lineartrunc}
There exist $\eta_0\in(0,1)$ and $\nu>0$, depending only on $M$, with the following property. Let $p^N(t)$ be a given sequence satisfying \eqref{asspn}, and let $\sigma^N_n(t)$ be defined by \eqref{tlinear2}. Let also
\begin{equation*}
\theta\in(-1,\beta), \qquad \tilde{\theta}\in[\theta,\beta], \qquad\text{with }\tilde{\theta}-\theta<\beta,
\end{equation*}
be fixed parameters, and let $y^0=\{y_n^0\}_{n\in\Z}\in\mathcal{Y}_\theta$ be a given initial datum with $y^0_n=0$ for $n>N$.
Then, for every $t_0\geq0$, there exists a unique solution
\begin{equation} \label{tlinear4}
t\mapsto T^N(t;t_0)(y^0) = \{T_n^N(t;t_0)(y^0)\}_{n\in\Z}\in\mathcal{Y}_{\theta}
\end{equation}
to \eqref{tlinear1} in $t\in(t_0,\infty)$, with $T^N(t_0;t_0)(y^0)=y^0$.

Furthermore, there exist constants $C_1=C_1(M,\theta,\tilde{\theta})$ and $C_2=C_2(M,\theta,\tilde{\theta})$ such that
\begin{equation} \label{tlinear5}
\|D^+(T^N(t;t_0)(y^0))\|_{\tilde{\theta}} \leq C_1 \|y^0\|_\theta (t-t_0)^{-\frac{\tilde{\theta}-\theta}{\beta}} e^{-\nu(t-t_0)},
\end{equation}
\begin{equation} \label{tlinear6}
\|T^N(t;t_0)(y^0)-T^N_{N}(t;t_0)(y^0)\|_{\tilde{\theta}} \leq C_2 \|y^0\|_\theta (t-t_0)^{-\frac{\tilde{\theta}-\theta}{\beta}} e^{-\nu(t-t_0)} \qquad\text{if $\tilde{\theta}>0$,}
\end{equation}
\begin{equation} \label{tlinear7}
|\mathscr{L}^N_N(T^N(t;t_0)(y^0),t)| \leq  C_2 \|y^0\|_\theta (t-t_0)^{-\frac{\tilde{\theta}-\theta}{\beta}} e^{-\nu(t-t_0)} \qquad\text{if $\tilde{\theta}>0$.}
\end{equation}
\end{theorem}

The proof of the theorem can be obtained by a slight refinement of the argument given for Theorem~\ref{thm:slinear}.


\section{Proof of the main result}\label{sect:proof}

Along this section we assume that $g_0\in\measg$ is a given initial datum, with total mass $M>0$, satisfying the assumptions of Theorem~\ref{thm:stability}, for some $\delta_0>0$ that will be chosen at the end of the proof. For the moment we assume that $\delta_0\in(0,1)$ satisfies the condition \eqref{delta0a}, which guarantees the conservation of the support. We also let $\nu>0$ be the constant given by Theorem~\ref{thm:lineartrunc}, determined solely by $M$.


\subsection{Truncation} \label{subsect:truncation}
The first step in the proof is to obtain a truncated solution by a cut-off of the tail for large $n$. More precisely, we fix a parameter $N\in\N$ ($N>>1$) and we truncate the kernels and the initial datum by setting
\begin{equation} \label{truncation1}
g_0^N(x):=g_0(x)\chi_{(-\infty,N+\frac12)}(x),
\end{equation}
\begin{equation} \label{truncation2}
K_N(2^y,2^z) := K(2^y,2^z)\chi_{(-\infty,N-\frac12)}(y),
\qquad
\gamma_N(2^y) := \gamma(2^y)\chi_{(-\infty,N+\frac12)}(y).
\end{equation}
We are in the position to apply the well-posedness Theorem~\ref{thm:wp} (replacing the original kernels by the truncated ones), with the initial datum $g_0^N$: indeed the assumption \eqref{m0} guarantees that the integrability conditions \eqref{moment0} are satisfied. This yields the existence of a global-in-time weak solution $t\mapsto g^N(t)\in\measg$, in the sense of Definition~\ref{def:weakg}, with conserved mass; the estimates \eqref{wpestg} hold uniformly with respect to $N$.

Furthermore, by Lemma~\ref{lem:supp} (which continues to hold in this truncated setting) we have that $\supp g^N(t) \subset \bigcup_{n\in\Z}I_n$ for all positive times, where $I_n=(n-\delta_0,n+\delta_0)$. By using this information and bearing in mind that by \eqref{suppK} and \eqref{delta0a} different intervals do not interact in the coagulation term, we can write the weak formulation \eqref{weakg} of the equation as follows:
\begin{equation*}
\begin{split}
\partial_t\biggl(\int_{\R} &g^N(x,t)\vphi(x)\de x \biggr) \\
& = \frac{\ln2}{2}\sum_{k\leq N-1}\int_{I_k}\int_{I_k} K_N(2^y,2^z)g^N(y,t)g^N(z,t)\biggl[\vphi\Bigl(\frac{\ln(2^y+2^z)}{\ln2}\Bigr) - \vphi(y) - \vphi(z) \biggr]\de y \de z \\
& \qquad - \frac14\sum_{k\leq N-1} \int_{I_k} \gamma_N(2^{y+1}) g^N(y+1,t) \bigl[\vphi(y+1) - 2\vphi(y) \bigr]\de y
\end{split}
\end{equation*}
for every test function $\vphi\in \cc(\R)$. It is then clear (recall the implication \eqref{claimsupp}) that
\begin{equation} \label{truncation4}
\supp g^N(t) \subset \bigcup_{n\leq N} I_n \qquad\text{for all $t>0$.}
\end{equation}

We define the quantities $m_n^N(t)$, $p_n^N(t)$, $q_n^N(t)$, according to \eqref{mn}, \eqref{pn}, \eqref{qn}, respectively, with $g$ replaced by $g^N$.
The evolution equations for these quantities (for a non-truncated weak solution) have been obtained in Lemma~\ref{lem:approx}; since we  are now working with the truncated kernels, we have that those equations continue to hold for $m_n^N(t)$, $p_n^N(t)$, $q_n^N(t)$ for all $n\leq N-1$, while they have to be slightly modified for $n=N$. For larger values $n>N$ all the quantities vanish identically by \eqref{truncation4}.
For instance, the equation for $m^N_n(t)$ for $n=N$ becomes
\begin{equation}\label{mneq2N}
\frac{\de m^N_N}{\de t} = \frac{\gamma(2^{N+p^N_N})}{4}\Bigl( \zeta_{N-1}(p^N) (m^N_{N-1})^2 (1+O(q^N_{N-1})) - m^N_N(1+O(q^N_N)) \Bigr).
\end{equation}


\subsection{Decay of the first and second moments} \label{subsect:pnqn}
We remark that by \eqref{truncation4}
\begin{equation} \label{decaypn0}
|p_n^N(t)| \leq \delta_0, \qquad 0\leq q_n^N(t) \leq 4\delta_0^2 \qquad\text{for all $t\geq0$.}
\end{equation}
We now fix two auxiliary parameters $\bar\theta_1$ and $\bar{\theta}_2$ satisfying
\begin{equation} \label{theta12}
\beta-1 < \bar{\theta}_1 < \bar{\theta}_2 <1.
\end{equation}
These parameters are fixed throughout the paper, therefore in the following we will not mention explicitly the dependence of all the constants on $\bar\theta_1$, $\bar\theta_2$. We let also $L_1$, $L_2$, $L_3$ be positive constants, that will be chosen later depending only on $M$.

Thanks to the fact that we are considering a truncated problem, we can assume that there exists a small time interval $[0,t^N]$, with $t^N>0$ (depending on $N$) such that for all $t\in(0,t^N)$
\begin{equation} \label{decaypn1}
\|D^+(p^N(t))\|_{0} \leq 2L_1\delta_0e^{-\frac{\nu}{2}t},
\qquad
\| D^+(p^N(t))\|_{\bar{\theta}_1} \leq 2L_1\delta_0 \, t^{-\bar{\theta}_1/\beta} e^{-\frac{\nu}{2} t},
\end{equation}
\begin{equation} \label{decaypn2}
\Big\|\frac{\de p^N}{\de t}(t)\Big\|_{-\beta} \leq 2L_2 \delta_0 e^{-\frac{\nu}{2} t},
\qquad
\Big\|\frac{\de p^N}{\de t}(t)\Big\|_{\bar{\theta}_1 - \beta} \leq 2L_2 \delta_0 \bigl(1+t^{-\bar{\theta}_1/\beta}\bigr) e^{-\frac{\nu}{2} t},
\end{equation}
\begin{equation} \label{decayqn1}
\sup_{n\in\Z}|q_n^N(t)| \leq 8\delta_0^{3/2} e^{-\nu t},
\qquad
\sup_{n>0} 2^{\bar{\theta}_2 n} | q^N_n(t) | \leq 2L_3\delta_0^{3/2} \, t^{-\bar{\theta}_2/\beta} e^{-\nu t},
\end{equation}
where $D^+$ denotes the discrete derivative (see \eqref{linearD}).
This is the expected decay for the functions $p^N$, $q^N$ as suggested by the natural scaling in the corresponding evolution equations. For technical reasons we can not obtain the optimal decay, namely for $\bar{\theta}_1=\bar{\theta}_2=1$, and we have to introduce the two additional parameters satisfying \eqref{theta12}.

In addition to \eqref{delta0a}, we make now a second assumption on $\delta_0$, namely
\begin{equation} \label{delta0b}
\delta_0<\eta_0, \qquad 2L_1\delta_0 < \eta_0, \qquad 2L_2\delta_0 < \eta_0,
\end{equation}
where $\eta_0>0$ is the fixed constant given by Theorem~\ref{thm:lineartrunc}, determined by $M$; with this assumption the sequence $p^N$ satisfies the condition \eqref{asspn} in the interval $(0,t^N)$, and allows us to apply the linear theory developed in Section~\ref{sect:linear}.


\subsection{Fixed point} \label{subsect:fixedpoint}
The next goal is to represent the functions $m_n^N(t)$, in the time interval $[0,t^N]$, as perturbations of stationary states, as in \eqref{yn}.
In order to do this, we first need to rewrite the starting assumption \eqref{m0} in an equivalent form for the functions $m^N(0)$.

\begin{lemma} \label{lem:initialcond}
For all sufficiently large $N$ there exists $A^N(0)>0$, $y^N(0)=\{y^N_n(0)\}_{n\in\Z}$ such that
\begin{equation} \label{m0N}
m_n^N(0) = \bar{m}_n(A^N(0),p^N(0)) (1+2^ny^N_n(0)) \qquad\text{for all $n\leq N$,}
\end{equation}
with $y^N_N(0)=0$ and
\begin{equation} \label{A0Ny0N}
|A^N(0)-A^0| \leq c_02^{-N}\delta_0, \qquad \|y^N(0)-y^0\|_1 \leq c_0\delta_0
\end{equation}
for a uniform constant $c_0>0$, independent of $N$.
\end{lemma}

\begin{proof}
We use the explicit representation of $\bar{m}_n$ that can be obtained by combining \eqref{nearlystat2} and \eqref{nearlystat5}:
\begin{equation} \label{proofinitialcond1}
\bar{m}_n(A,p) = \frac{2e^{-A2^n}}{\zeta_n(p)} \exp\Biggl( -2^n\sum_{j=n+1}^{\infty}2^{-j}\ln(\theta_{j-1}(p))\Biggr),
\end{equation}
where the coefficients $\zeta_n(p)$, $\theta_n(p)$ are defined in \eqref{nearlystat1} and \eqref{nearlystat3} respectively. By observing that $m^N_n(0)=m_n(0)$ for all $n\leq N$, recalling the assumption \eqref{m0} and imposing that \eqref{m0N} holds at $n=N$ with $y^N_N(0)=0$, we obtain the condition that defines $A^N(0)$:
\begin{equation*}
\bar{m}_N(A^N(0),p^N(0)) = \bar{m}_N(A^0,p^0)(1+2^Ny_N^0),
\end{equation*}
which yields, using \eqref{proofinitialcond1},
\begin{equation*}
e^{(A^N(0)-A^0)2^N}
= \frac{\zeta_N(p^0)}{\zeta_N(p^N(0))} \exp\Biggl( 2^N\sum_{j=N+1}^{\infty}2^{-j} \ln\Bigl(\frac{\theta_{j-1}(p^0)}{\theta_{j-1}(p^N(0))}\Bigr) \Biggr)(1+2^Ny_N^0)^{-1},
\end{equation*}
or equivalently
\begin{equation*}
A^N(0)-A^0
= 2^{-N}\ln\biggl(\frac{\zeta_N(p^0)}{\zeta_N(p^N(0))}\biggr) +\sum_{j=N+1}^{\infty}2^{-j} \ln\Bigl(\frac{\theta_{j-1}(p^0)}{\theta_{j-1}(p^N(0))}\Bigr) - 2^{-N}\ln(1+2^Ny_N^0).
\end{equation*}
This equation selects the value $A^N(0)$, and moreover implies the first estimate in \eqref{A0Ny0N}. We further define $y^N_n(0)$, for $n<N$, by imposing that \eqref{m0N} holds: this gives for $n<N$
\begin{align*}
2^ny_n^N(0) &:= \frac{m_n^N(0)}{\bar{m}_n(A^N(0),p^N(0))} - 1 = \frac{\bar{m}_n(A^0,p^0)(1+2^ny_n^0)}{\bar{m}_n(A^N(0),p^N(0))} - 1 \\
& \xupref{proofinitialcond1}{=} e^{(A^N(0)-A^0)2^n}\frac{\zeta_n(p^N(0))}{\zeta_n(p^0)} \exp\Biggl( - 2^n\sum_{j=n+1}^{\infty}2^{-j}\ln\Bigl(\frac{\theta_{j-1}(p^0)}{\theta_{j-1}(p^N(0))}\Bigr)\Biggr) (1+2^ny^0_n) - 1.
\end{align*}
From this expression also the second estimate in \eqref{A0Ny0N} follows.
\end{proof}

We now look for functions $A^N(t)$, $y^N(t)=\{y_n^N(t)\}_{n\in\Z}$ such that
\begin{equation} \label{ynN}
	m^N_n(t) = \bar{m}_n(A^N(t),p^N(t))\bigl( 1 + 2^ny^N_n(t) \bigr) \qquad \text{for all $n\leq N$.}
\end{equation}
By plugging this expression into \eqref{mneq2}, as in \eqref{yneq2}, we find for all $n<N$
\begin{align*}
\frac{\de y^N_n}{\de t}
& = (1+2^ny^N_n)\frac{\de A^N}{\de t} - (1+2^ny^N_n)\frac{1}{2^n}\sum_{k=n}^\infty \frac{1}{\bar{m}_n}\frac{\partial\bar{m}_n}{\partial p_k}\frac{\de p^N_k}{\de t} \nonumber\\
& \qquad + \frac{\gamma(2^{n+p^N_n})}{2^{n+2}} \biggl[ (1+2^{n-1}y^N_{n-1})^2(1+O(q^N_{n-1})) - (1+2^ny^N_n)(1+O(q^N_n)) \nonumber\\
& \qquad - 4\bar{\mu}_n\frac{\gamma(2^{n+1+p^N_{n+1}})}{\gamma(2^{n+p^N_n})} \Bigl( (1+2^ny^N_n)^2(1+O(q^N_n)) - (1+2^{n+1}y^N_{n+1})(1+O(q^N_{n+1})) \Bigr) \biggr] ,
\end{align*}
where $\bar{m}_n=\bar{m}_n(A^N(t),p^N(t))$, $\bar{\mu}_n=\bar{\mu}_n(A^N(t),p^N(t))$ (see \eqref{nearlystat2}). For $n=N$, using instead \eqref{mneq2N}, we have
\begin{align*}
\frac{\de y^N_N}{\de t}
& = (1+2^Ny^N_N)\frac{\de A^N}{\de t} - (1+2^Ny^N_N)\frac{1}{2^N}\sum_{k=N}^\infty \frac{1}{\bar{m}_N}\frac{\partial\bar{m}_N}{\partial p_k}\frac{\de p^N_k}{\de t} \nonumber\\
& \qquad + \frac{\gamma(2^{N+p^N_N})}{2^{N+2}} \biggl[ (1+2^{N-1}y^N_{N-1})^2(1+O(q^N_{N-1})) - (1+2^Ny^N_N)(1+O(q^N_N)) \biggr] .
\end{align*}
The previous equations for the sequence $y^N(t)$ can be written in a more compact form, where we highlight the leading order linear operator:
\begin{multline} \label{yneq}
\frac{\de y^N_n}{\de t}
= \frac{\gamma(2^{n+p^N_n})}{4} \Bigl( y^N_{n-1}-y^N_n - \sigma^N_n(t) (y^N_n-y^N_{n+1}) \Bigr) \\
+ (1+2^ny^N_n)\frac{\de A^N}{\de t} + \sum_{i=1}^{3}r^{(i)}_n(y^N(t),A^N(t),t)  \qquad\text{for $n\leq N$}
\end{multline}
and $y_n^N(t)\equiv0$ for $n>N$, where we introduced the coefficients
\begin{equation} \label{sigman}
\sigma^N_n(t) :=
\begin{cases}
8\bar{\mu}_n(A_M,p^N(t))\frac{\gamma(2^{n+1+p^N_{n+1}(t)})}{\gamma(2^{n+p^N_n(t)})} & \text{if $n<N$,} \\
0 & \text{if $n=N$.}
\end{cases}
\end{equation}
The remainders $r^{(i)}_n$ in \eqref{yneq} are explicitly given (for $n<N$) by
\begin{multline} \label{r1}
r^{(1)}_n(y,A,t) := \frac{2^n\gamma(2^{n+p^N_n})}{4} \biggl[ \frac14 y_{n-1}^2 - 4\bar{\mu}_n(A,p^N)\frac{\gamma(2^{n+1+p^N_{n+1}})}{\gamma(2^{n+p^N_n})} y_n^2 \biggr] \\
+ 2\gamma(2^{n+1+p^N_{n+1}}) \bigl(\bar{\mu}_n(A_M,p^N)-\bar{\mu}_n(A,p^N)\bigr)\bigl( y_n-y_{n+1} \bigr),
\end{multline}
\begin{multline} \label{r2}
r^{(2)}_n(y,A,t) := \frac{\gamma(2^{n+p^N_n})}{2^{n+2}} \Bigl( (1+2^{n-1}y_{n-1})^2O(q^N_{n-1}) - (1+2^ny_n)O(q^N_n) \Bigr) \\
- \bar{\mu}_n(A,p^N)\frac{\gamma(2^{n+1+p^N_{n+1}})}{2^n} \Bigl( (1+2^ny_n)^2O(q^N_n) - (1+2^{n+1}y_{n+1})O(q^N_{n+1}) \Bigr),
\end{multline}
\begin{equation} \label{r3}
r^{(3)}_n(y,A,t)
:= - (1+2^ny_n)\frac{1}{2^n}\sum_{k=n}^\infty \frac{1}{\bar{m}_n(A,p^N(t))}\frac{\partial\bar{m}_n}{\partial p_k}(A,p^N(t))\frac{\de p^N_k}{\de t}
\end{equation}
(for $n=N$, the terms containing $\bar{\mu}_n$ in $r^{(1)}_N$ and $r^{(2)}_N$ are not present). 

At this point it is important to recall that $p^N(t)=\{p^N_n(t)\}_{n\in\Z}$ and $q^N(t)=\{q^N_n(t)\}_{n\in\Z}$ are given sequences, satisfying the estimates \eqref{decaypn1}--\eqref{decayqn1} in the small time interval $(0,t^N)$. The goal is to show the existence of a pair $(y^N(t),A^N(t))$ solving \eqref{yneq}.
The linearized operator in \eqref{yneq} has exactly the form \eqref{tlinear3}, with the sequence $p^N(t)$ satisfying the assumption \eqref{asspn} in $(0,t^N)$ in view of \eqref{delta0b}; we then denote by $T^N(t;s)$ the corresponding resolvent operator, according to \eqref{tlinear4}.

By Duhamel's formula, the solution to \eqref{yneq} with a given initial datum $y^N(0)$ can be represented in terms of the solution to the linearized problem as
\begin{multline} \label{yn2}
y^N_n(t) = T^N_n(t;0)(y^N(0)) + A^N(t) - A^0 + \int_0^t \frac{\de A^N}{\de t}(s)T^N_n(t;s)(P(y^N(s)))\de s \\
+ \sum_{i=1}^{3}\int_0^t T^N_n(t;s)(r^{(i)}(y^N(s),A^N(s),s))\de s
\end{multline}
for $n\leq N$, where for notational convenience we introduced the operator
\begin{equation} \label{linearP}
P_n(y) := 2^ny_n, \qquad P(y):=\{ P_n(y) \}_{n\in\Z} \,.
\end{equation}
We select the function $A^N(t)$ by imposing that $y^N_N(t)=0$ for all $t>0$ (notice that this condition is satisfied at time $t=0$, see Lemma~\ref{lem:initialcond}): this gives the equation
\begin{multline} \label{yn3}
A^N(t) - A^0 = - T^N_N(t;0)(y^N(0)) - \int_0^t \frac{\de A^N}{\de t}(s)T^N_N(t;s)(P(y^N(s)))\de s \\
- \sum_{i=1}^{3}\int_0^t T^N_N(t;s)(r^{(i)}(y^N(s),A^N(s),s))\de s ,
\end{multline}
and by differentiating with respect to $t$
\begin{multline} \label{fp2}
\frac{\de A^N}{\de t}
= - \biggl[ \mathscr{L}^N_N(T^N(t;0)(y^N(0)),t)
+ \int_0^t \frac{\de A^N}{\de t}(s) \mathscr{L}^N_N \bigl[ T^N(t;s)( P(y^N(s)) ), t \bigr] \de s \\
+ \sum_{i=1}^3\int_0^t \mathscr{L}^N_N \bigl[ T^N(t;s)(r^{(i)}(y^N(s),A^N(s),s)), t \bigr] \de s
+ \sum_{i=1}^3 r^{(i)}_N(y^N(t),A^N(t),t)\biggr],
\end{multline}
where $\mathscr{L}^N$ denotes the linear operator on the right-hand side of \eqref{yneq}, see \eqref{tlinear3}.
In turn, by inserting \eqref{yn3} into \eqref{yn2} we have for $n\leq N$
\begin{multline}  \label{fp1}
y^N_n(t)
= \bigl[T^N_n(t;0)-T^N_N(t;0)\bigr](y^N(0)) + \int_0^t \frac{\de A^N}{\de t}(s) \bigl[T^N_n(t;s)-T^N_N(t;s)\bigr]( P(y^N(s)) )\de s \\
+ \sum_{i=1}^3 \int_0^t \bigl[T^N_n(t;s)-T^N_N(t;s)\bigr](r^{(i)}(y^N(s),A^N(s),s))\de s \,.
\end{multline}
The pair $(y^N(t),A^N(t))$ will be determined by applying a fixed point argument to the two equations \eqref{fp2}--\eqref{fp1}.
This is the content of the following proposition.

\begin{proposition} \label{prop:fp}
There exists $\delta_0>0$, depending on $M$, $L_1$, $L_2$, $L_3$, with the following property. Let $g_0$ be an initial datum satisfying the assumptions of Theorem~\ref{thm:stability} and let $g^N$ be the corresponding weak solution to the truncated problem, obtained in Section~\ref{subsect:truncation}. Assume further that the maps $t\mapsto p^N(t)$, $t\mapsto q^N(t)$ satisfy the estimates \eqref{decaypn1}, \eqref{decaypn2}, \eqref{decayqn1} for all $t\in(0,t^N)$, for some $t^N>0$.

Then there exist functions $t\mapsto (y^N(t),A^N(t))$, for $t\in(0,t^N)$, such that $y^N_N(t)\equiv0$ and
\begin{equation} \label{fp}
m^N_n(t) = \bar{m}_n(A^N(t),p^N(t))\bigl( 1 + 2^ny^N_n(t) \bigr) \qquad \text{for all $n\leq N$ and $t\in(0,t^N)$.}
\end{equation}
Moreover the following estimates hold:
\begin{equation}\label{fp3}
\begin{split}
\|y^N(t)\|_1\leq C_0\delta_0,
\qquad
\|y^N(t)\|_\beta \leq C_0\delta_0\bigl(1+ t^{-\frac{\beta-1}{\beta}}) e^{-\frac{\nu}{2} t}, \\
|A^N(t)-A_M|\leq C_0\delta_0,
\qquad
\Big| \frac{\de A^N}{\de t}(t) \Big| \leq C_0\delta_0\bigl(1+t^{-\frac{\beta-1}{\beta}}\bigr) e^{-\frac{\nu}{2}t},
\end{split}
\end{equation}
for a constant $C_0$ depending only on $M$, $L_1$, $L_2$, and $L_3$.
\end{proposition}

\begin{proof}
Along the proof, we will denote by $C$ a generic constant, possibly depending on the properties of the kernels, on $M$, $L_1$, $L_2$, and $L_3$, which might change from line to line. Let $\delta>0$ be a small parameter, to be chosen later, and let $g:(0,\infty)\to\R$ be the function
\begin{equation} \label{pfp0}
g(t) := \bigl(1+t^{-\frac{\beta-1}{\beta}}\bigr) e^{-\frac{\nu}{2} t}.
\end{equation}
Since we always deal with truncated sequences, it is convenient to denote by $\mathcal{Y}_\beta^N$ the space of sequences $y\in\mathcal{Y}_\beta$ (see \eqref{spacesequences}) such that $y_n=0$ for all $n\geq N$.
We work in the space $\mathcal{X}:=\mathcal{X}_1\times\mathcal{X}_2$, where
\begin{equation} \label{pfp1}
\mathcal{X}_1 := \Bigl\{ y\in C([0,t^N];\mathcal{Y}_\beta^N) \,:\, \|y\|_{\mathcal{X}_1}\leq \delta \Bigr\},
\qquad
\|y\|_{\mathcal{X}_1} := \sup_{0<t<t^N} \Bigl( \|y(t)\|_1 + \frac{\|y(t)\|_\beta}{g(t)}\Bigr) \,,
\end{equation}
\begin{equation} \label{pfp2}
\mathcal{X}_2 := \Bigl\{ \Lambda\in C([0,t^N];\R) \,:\, \|\Lambda\|_{\mathcal{X}_2}\leq \delta \Bigr\},
\qquad
\|\Lambda\|_{\mathcal{X}_2} := \sup_{0<t<t^N} \, \frac{|\Lambda(t)|}{g(t)} \,.
\end{equation}
For $\Lambda\in\mathcal{X}_2$ we let
\begin{equation} \label{pfp5}
A_\Lambda(t) := A^N(0) + \int_0^t \Lambda(s)\de s \,.
\end{equation}
Notice that for every $\Lambda\in\mathcal{X}_2$
\begin{equation} \label{pfp5bis}
|A_\Lambda(t)-A^N(0)| = \bigg|\int_0^t\Lambda(s)\de s\bigg| \leq \|\Lambda\|_{\mathcal{X}_2}\int_0^t g(s)\de s \leq C\delta .
\end{equation}
By combining \eqref{A0y0}, \eqref{A0Ny0N}, and \eqref{pfp5bis} we find
\begin{equation}  \label{pfp5ter}
|A_{\Lambda}(t)-A_M| \leq C\delta + c_02^{-N}\delta_0 + \delta_0 \leq C(\delta+\delta_0)
\end{equation}
and in particular we can assume without loss of generality that $|A_{\Lambda}(t)-A_M| \leq \frac{A_M}{2}$ for every $\Lambda\in\mathcal{X}_2$, provided that we choose $\delta_0$ and $\delta$ sufficiently small (depending on $M$).

We define a map $\mathcal{T}:\mathcal{X}\to\mathcal{X}$ by setting $\mathcal{T}(y,\Lambda):=(\tilde{y},\tilde{\Lambda})$, where
\begin{multline} \label{pfp3}
\tilde{y}_n(t) := \bigl[T^N_n(t;0)-T^N_N(t;0)\bigr](y^N(0)) + \int_0^t \Lambda(s) \bigl[T^N_n(t;s)-T^N_N(t;s)\bigr]( P(y(s)) )\de s \\
+ \sum_{i=1}^3 \int_0^t \bigl[T^N_n(t;s)-T^N_N(t;s)\bigr](r^{(i)}(y(s),A_\Lambda(s),s))\de s
\end{multline}
for $n\leq N$, $\tilde{y}_n(t)=0$ for $n> N$, and
\begin{multline} \label{pfp4}
\tilde{\Lambda}(t) :=
- \biggl[ \mathscr{L}^N_N(T^N(t;0)(y^N(0)),t)
+ \int_0^t \Lambda(s) \mathscr{L}^N_N \bigl[ T^N(t;s)( P(y(s)) ) , t \bigr] \de s \\
+ \sum_{i=1}^3\int_0^t \mathscr{L}^N_N \bigl[ T^N(t;s)(r^{(i)}(y(s),A_\Lambda(s),s)) , t \bigr] \de s
+ \sum_{i=1}^3 r^{(i)}_N(y(t),A_\Lambda(t),t)\biggr].
\end{multline}
The rest of the proof amounts to showing that the map $\mathcal{T}$ is a contraction in $\mathcal{X}$, provided that $\delta$ and $\delta_0$ are chosen small enough.

\smallskip\noindent\textit{Step 1: $\mathcal{T}(y,\Lambda)\in\mathcal{X}$ for every $(y,\Lambda)\in\mathcal{X}$.}
Since $y^N(0)\in\mathcal{Y}_1$, by Theorem~\ref{thm:lineartrunc} we have (using also \eqref{A0y0}, \eqref{A0Ny0N})
\begin{equation} \label{pfp10}
\begin{split}
\big\| \bigl[ T^N(t;0)-T^N_N(t;0) \bigr] (y^N(0)) \big\|_\beta & \leq C_2(M,1,\beta)\|y^N(0)\|_1 t^{-\frac{\beta-1}{\beta}} e^{-\nu t} \leq C\delta_0 g(t), \\
|\mathscr{L}^N_N(T^N(t;0)(y^N(0)),t)| & \leq C_2(M,1,\beta)\|y^N(0)\|_1 t^{-\frac{\beta-1}{\beta}} e^{-\nu t} \leq C\delta_0 g(t).
\end{split}
\end{equation}
Similarly, since $y(s)\in\mathcal{Y}_\beta$ we have $P(y(s))\in\mathcal{Y}_{\beta-1}$ for every positive $s$, with
\begin{equation} \label{pfp13}
\|P(y(s))\|_{\beta-1} \leq \|y(s)\|_\beta,
\end{equation}
and it follows, again by Theorem~\ref{thm:lineartrunc}, that
\begin{equation} \label{pfp11}
\begin{split}
\big\| \bigl[ T^N(t;s)-T^N_N(t;s)\bigr] (P(y(s))) \big\|_\beta &\leq C_2(M,\beta-1,\beta)\|y(s)\|_\beta (t-s)^{-\frac{1}{\beta}} e^{-\nu(t-s)} \\
& \leq C \delta g(s) (t-s)^{-\frac{1}{\beta}} e^{-\nu(t-s)}, \\
\big|\mathscr{L}^N_N \bigl[ T^N(t;s)(P(y(s))),t \bigr] \big| &\leq C \delta g(s) (t-s)^{-\frac{1}{\beta}} e^{-\nu(t-s)} .
\end{split}
\end{equation}
In the same way, we can use the bounds \eqref{pfp30} on the remainders $r^{(i)}$ proved in Lemma~\ref{lem:fixedpointr} below (the assumptions of the lemma are satisfied in view of \eqref{pfp5ter} and $\|y(t)\|_1\leq\delta$), together with Theorem~\ref{thm:lineartrunc}, to obtain the following estimates:
\begin{equation} \label{pfp12a}
\begin{split}
\big\| \bigl[ T^N(t;s)-&T^N_N(t;s)\bigr] (r^{(1)}(y(s),A_\Lambda(s),s)) \big\|_\beta \\
& \leq C_2(M,\beta-1,\beta)\|r^{(1)}(y(s),A_\Lambda(s),s)\|_{\beta-1} (t-s)^{-\frac{1}{\beta}} e^{-\nu(t-s)} \\
& \leq C \Bigl( \|y(s)\|_{\beta}^2 + |A_{\Lambda}(s)-A_M|\|y(s)\|_\beta \Bigr) (t-s)^{-\frac{1}{\beta}} e^{-\nu(t-s)} \\
& \leq C \Bigl( \delta^2(g(s))^2 + \bigl( \delta + \delta_0 \bigr)\delta g(s) \Bigr) (t-s)^{-\frac{1}{\beta}}e^{-\nu(t-s)} ,
\end{split}
\end{equation}
\begin{equation} \label{pfp12b}
\begin{split}
\big\| \bigl[ T^N&(t;s)-T^N_N(t;s)\bigr] (r^{(2)}(y(s),A_\Lambda(s),s)) \big\|_\beta \\
& \leq C_2(M,\bar{\theta}_2-\beta+1,\beta)\|r^{(2)}(y(s),A_\Lambda(s),s)\|_{\bar{\theta}_2-\beta+1} (t-s)^{-\frac{2\beta-\bar{\theta}_2-1}{\beta}} e^{-\nu(t-s)} \\
& \leq C\delta_0 s^{-\bar{\theta}_2/\beta} (t-s)^{-\frac{2\beta-\bar{\theta}_2-1}{\beta}} e^{-\frac{\nu}{2} s}e^{-\nu(t-s)},
\end{split}
\end{equation}
\begin{equation} \label{pfp12c}
\begin{split}
\big\| \bigl[ T^N&(t;s)-T^N_N(t;s)\bigr] (r^{(3)}(y(s),A_\Lambda(s),s)) \big\|_\beta \\
& \leq C_2(M,\bar{\theta}_1-\beta+1,\beta)\|r^{(3)}(y(s),A_\Lambda(s),s)\|_{\bar{\theta}_1-\beta+1} (t-s)^{-\frac{2\beta-\bar{\theta}_1-1}{\beta}} e^{-\nu(t-s)} \\
& \leq C\delta_0 s^{-\bar{\theta}_1/\beta} (t-s)^{-\frac{2\beta-\bar{\theta}_1-1}{\beta}} e^{-\frac{\nu}{2} s}e^{-\nu(t-s)} .
\end{split}
\end{equation}
The same estimates hold for the terms $\mathscr{L}^N_N [ T^N(t;s)(r^{(i)}(y(s),A_\Lambda(s),s)) , t ]$, $i=1,2,3$.

By plugging the bounds \eqref{pfp10}, \eqref{pfp11}, \eqref{pfp12a}, \eqref{pfp12b}, \eqref{pfp12c} into \eqref{pfp3} we find
\begin{equation} \label{pfp16}
\begin{split}
\|\tilde{y}(t)\|_\beta
& \leq C\delta_0 g(t) + C\delta^2 \int_0^t \bigl(g(s)\bigr)^2 (t-s)^{-\frac{1}{\beta}} e^{-\nu(t-s)} \de s \\
& \qquad + C\bigl( \delta + \delta_0 \bigr)\delta \int_0^t g(s)(t-s)^{-\frac{1}{\beta}} e^{-\nu(t-s)}\de s \\
& \qquad + C\delta_0\sum_{i=1}^2 \int_0^t s^{-\bar{\theta}_i/\beta}(t-s)^{-\frac{2\beta-\bar{\theta}_i-1}{\beta}}  e^{-\frac{\nu}{2} s} e^{-\nu(t-s)} \de s \\
& \leq C\bigl( \delta_0 + \delta^2 + \delta_0\delta \bigr)g(t),
\end{split}
\end{equation}
where we used the elementary estimates
\begin{equation} \label{pfp15}
\begin{split}
\int_0^t \bigl(g(s)\bigr)^2 (t-s)^{-\frac{1}{\beta}} e^{-\nu(t-s)} \de s &\leq C g(t),\\
\int_0^t g(s) (t-s)^{-\frac{1}{\beta}} e^{-\nu(t-s)}\de s &\leq C g(t),\\
\int_0^t s^{-\bar{\theta}_i/\beta}(t-s)^{-\frac{2\beta-\bar{\theta}_i-1}\beta}e^{-\frac{\nu}{2} s}e^{-\nu(t-s)} \de s & \leq Cg(t) \qquad(i=1,2),
\end{split}
\end{equation}
for $C$ depending only on $\beta$, $\bar{\theta}_1$, $\bar{\theta}_2$, $\nu$ (recall the assumption \eqref{theta12}). In particular, by choosing $\delta$ and $\delta_0$ sufficiently small, depending ultimately only on $M$ (with $\delta_0$ depending on $\delta$, and $\delta_0\sim\delta$), \eqref{pfp16} yields $\|\tilde{y}(t)\|_\beta\leq \frac{\delta}{2}g(t)$.

In order to control $\|\tilde{y}(t)\|_1$, by a similar procedure we repeatedly apply Theorem~\ref{thm:lineartrunc} and we use the bounds \eqref{pfp30}: we find from \eqref{pfp3}
\begin{equation*}
\begin{split}
\|\tilde{y}(t)\|_1
& \leq C_2(M,1,1)\|y^N(0)\|_1e^{-\nu t} + C_2(M,0,1)\delta \int_0^t g(s) \|y(s)\|_1 (t-s)^{-\frac{1}{\beta}}e^{-\nu(t-s)}\de s \\
& \qquad + C_2(M,\beta-1,1) \int_0^t \|r^{(1)}(y(s),A_\Lambda(s),s)\|_{\beta-1} (t-s)^{-\frac{2-\beta}{\beta}}e^{-\nu(t-s)}\de s\\
& \qquad + C_2(M,\bar{\theta}_2-\beta+1,1)\int_0^t \|r^{(2)}(y(s),A_\Lambda(s),s)\|_{\bar{\theta}_2-\beta+1} (t-s)^{-\frac{\beta-\bar{\theta}_2}{\beta}}e^{-\nu(t-s)}\de s\\
& \qquad + C_2(M,\bar{\theta}_1-\beta+1,1)\int_0^t \|r^{(1)}(y(s),A_\Lambda(s),s)\|_{\bar{\theta}_1-\beta+1} (t-s)^{-\frac{\beta-\bar{\theta}_1}{\beta}}e^{-\nu(t-s)}\de s\\
& \leq C\delta_0 + C\delta^2\int_0^t g(s)(t-s)^{-\frac{1}{\beta}}e^{-\nu(t-s)}\de s \\
& \qquad + C \int_0^t \Bigl[ \delta^2\bigl(g(s)\bigr)^2 + (\delta+\delta_0)\delta g(s) \Bigr] (t-s)^{-\frac{2-\beta}{\beta}}e^{-\nu(t-s)}\de s  \\
& \qquad + C \delta_0 \sum_{i=1}^2\int_0^t s^{-\bar{\theta}_i/\beta} e^{-\frac{\nu}{2} s} (t-s)^{-\frac{\beta-\bar{\theta}_i}{\beta}}e^{-\nu(t-s)}\de s \,.
\end{split}
\end{equation*}
By estimates similar to \eqref{pfp15} we then find $\|\tilde{y}(t)\|_1\leq C( \delta_0 + \delta^2 )\leq\frac{\delta}{2}$, provided that we choose $\delta$ and $\delta_0$ sufficiently small; this, combined with the previous estimate, gives $\|\tilde{y}\|_{\mathcal{X}_1} \leq \delta$.

We proceed similarly for $\tilde{\Lambda}$: the first three terms on the right-hand side of \eqref{pfp4} can be estimated exactly in the same way, by using \eqref{pfp10}, \eqref{pfp11}, and the equivalent of \eqref{pfp12a}--\eqref{pfp12c} for $\mathscr{L}^N_N [ T^N(t;s)(r^{(i)}(y(s),A_\Lambda(s),s)) , t ]$. The only novelty is the last term in \eqref{pfp4}, which can be controlled thanks to \eqref{pfp30b}. In this way one obtains an estimate of the form
\begin{equation*}
|\tilde{\Lambda}(t)| \leq C\bigl( \delta_0 + \delta^2 + \delta_0\delta \bigr)g(t)
\end{equation*}
and in turn $\|\tilde{\Lambda}\|_{\mathcal{X}_2} \leq \delta$. Hence we can conclude that $\mathcal{T}$ maps $\mathcal{X}$ into itself.

\smallskip\noindent\textit{Step 2: contractivity.}
Let $(y^1,\Lambda^1),(y^2,\Lambda^2)\in\mathcal{X}$ and set $(\tilde{y}^i,\tilde{\Lambda}^i):=\mathcal{T}(y^i,\Lambda^i)$, $i=1,2$. In view of the definition \eqref{pfp3} of $\tilde{y}^i$ we have
\begin{align} \label{pfp14}
\big| \tilde{y}_n^1(t)-\tilde{y}_n^2&(t) \big| \leq \int_0^t \big| \Lambda^1(s)-\Lambda^2(s)\big| \big|T_n^N(t;s)-T_N^N(t;s)\big|( P(y^1(s)) )\de s \nonumber\\
& +  \int_0^t |\Lambda^2(s)| \big|T_n^N(t;s)-T^N_N(t;s)\big| \bigl( P(y^1(s)-y^2(s)) \bigr)\de s \\
& + \sum_{i=1}^3\int_0^t \big|T_n^N(t;s)-T_N^N(t;s)\big| \bigl( r^{(i)}(y^1(s),A_{\Lambda^1}(s),s)-r^{(i)}(y^2(s),A_{\Lambda^2}(s),s) \bigr)\de s \,. \nonumber
\end{align}
The first two integrals can be estimated using \eqref{pfp11}; for the last integral containing the remainders, similarly to \eqref{pfp12a}--\eqref{pfp12c} we find, using \eqref{pfp31} in Lemma~\ref{lem:fixedpointr} below,
\begin{equation*}
\begin{split}
\big\| \bigl(T^N&(t;s)-T_N^N(t;s)\bigr) \bigl( r^{(1)}(y^1(s),A_{\Lambda^1}(s),s)-r^{(1)}(y^2(s),A_{\Lambda^2}(s),s) \bigr) \big\|_\beta \\
& \leq C_2(M,\beta-1,\beta)\| r^{(1)}(y^1(s),A_{\Lambda^1}(s),s)-r^{(1)}(y^2(s),A_{\Lambda^2}(s),s) \|_{\beta-1} (t-s)^{-\frac{1}{\beta}} e^{-\nu(t-s)} \\
& \leq C \biggl[ \|y^1-y^2\|_{\mathcal{X}_1}g(s) \Bigl( \delta g(s) + (\delta + \delta_0) \Bigr) \\
& \qquad\qquad + \|\Lambda^1-\Lambda^2\|_{\mathcal{X}_2} \Bigl( \delta^2\bigl(g(s)\bigr)^2 + \delta g(s)\Bigr) \biggr] (t-s)^{-\frac{1}{\beta}} e^{-\nu(t-s)} \,,
\end{split}
\end{equation*}
where we used the bound $|A_{\Lambda^1}(s)-A_{\Lambda^2}(s)|\leq C\|\Lambda^1-\Lambda^2\|_{\mathcal{X}_2}$, which follows from \eqref{pfp5}; and
\begin{equation*}
\begin{split}
\big\| \bigl(T^N&(t;s)-T_N^N(t;s)\bigr) \bigl( r^{(2)}(y^1(s),A_{\Lambda^1}(s),s)-r^{(2)}(y^2(s),A_{\Lambda^2}(s),s) \bigr) \big\|_\beta \\
& \leq C_2\| r^{(2)}(y^1(s),A_{\Lambda^1}(s),s)-r^{(2)}(y^2(s),A_{\Lambda^2}(s),s) \|_{\bar{\theta}_2-\beta+1} (t-s)^{-\frac{2\beta-\bar{\theta}_2-1}{\beta}} e^{-\nu(t-s)} \\
& \leq C\delta_0 \Bigl( \|y^1-y^2\|_{\mathcal{X}_1} + \|\Lambda^1-\Lambda^2\|_{\mathcal{X}_2} \Bigr) s^{-\bar{\theta}_2/\beta} e^{-\frac{\nu}{2}s} (t-s)^{-\frac{2\beta-\bar{\theta}_2-1}{\beta}} e^{-\nu(t-s)}
\end{split}
\end{equation*}
(the same estimate holds for $r^{(3)}$, with $\bar{\theta}_1$ in place of $\bar{\theta}_2$). Hence from \eqref{pfp14} it is straightforward to obtain an estimate of the form
\begin{multline*}
\big\| \tilde{y}^1(t) - \tilde{y}^2(t) \big\|_{\beta}
\leq C(\delta+\delta_0) \Bigl( \|y^1-y^2\|_{\mathcal{X}_1} + \|\Lambda^1-\Lambda^2\|_{\mathcal{X}_2} \Bigr) \\
\biggl( \int_0^t \Bigl[ \bigl(g(s)\bigr)^2 + g(s) \Bigr] (t-s)^{-\frac{1}{\beta}} e^{-\nu(t-s)} \de s + \sum_{i=1}^2\int_0^t  s^{-\bar{\theta}_i/\beta} e^{-\frac{\nu}{2}s} (t-s)^{-\frac{2\beta-\bar{\theta}_i-1}{\beta}} e^{-\nu(t-s)} \de s \biggr),
\end{multline*}
which in turn yields, recalling \eqref{pfp15},
\begin{equation*}
\| \tilde{y}^1(t) - \tilde{y}^2(t) \|_{\beta} \leq C(\delta+\delta_0) \Bigl( \|y^1-y^2\|_{\mathcal{X}_1} + \|\Lambda^1-\Lambda^2\|_{\mathcal{X}_2} \Bigr) g(t).
\end{equation*}
In a similar way we obtain an estimate for $\|\tilde{y}^1(t)-\tilde{y}^2(t)\|_1$, which combined with the previous one gives
\begin{equation} \label{pfp20}
\big\| \tilde{y}^1 - \tilde{y}^2 \big\|_{\mathcal{X}_1} \leq C(\delta+\delta_0) \Bigl( \|y^1-y^2\|_{\mathcal{X}_1} + \|\Lambda^1-\Lambda^2\|_{\mathcal{X}_2} \Bigr) .
\end{equation}
Starting from the inequality
\begin{align*}
\big| \tilde{\Lambda}^1(t) - \tilde{\Lambda}^2(t) \big|
& \leq \int_0^t \big| \Lambda^1(s)-\Lambda^2(s)\big| \big| \mathscr{L}^N_N\bigl[T^N(t;s)(P(y^1(s))),t\bigr] \big| \de s \\
& +  \int_0^t |\Lambda^2(s)| \big| \mathscr{L}^N_N\bigl[T^N(t;s)(P(y^1(s)-y^2(s))),t\bigr] \big| \de s \\
& + \sum_{i=1}^3 \int_0^t \big| \mathscr{L}^N_N\bigl[ T^N(t;s) \bigl( r^{(i)}(y^1(s),A_{\Lambda^1}(s),s)-r^{(i)}(y^2(s),A_{\Lambda^2}(s),s) \bigr) , t \bigr] \big| \de s\\
& + \sum_{i=1}^3 \big| r^{(i)}_N(y^1(t),A_{\Lambda^1}(t),t)-r^{(i)}_N(y^2(t),A_{\Lambda^2}(t),t) \big| \,,
\end{align*}
the same argument (using also \eqref{pfp31b} in Lemma~\ref{lem:fixedpointr}) shows that
\begin{equation} \label{pfp21}
\big\| \tilde{\Lambda}^1 - \tilde{\Lambda}^2 \big\|_{\mathcal{X}_1} \leq C(\delta+\delta_0) \Bigl( \|y^1-y^2\|_{\mathcal{X}_1} + \|\Lambda^1-\Lambda^2\|_{\mathcal{X}_2} \Bigr)  .
\end{equation}
Therefore by \eqref{pfp20}--\eqref{pfp21} it follows that the map $\mathcal{T}$ is a contraction in the space $\mathcal{X}$, provided that $\delta$ and $\delta_0$ are small enough.

\smallskip\noindent\textit{Step 3: conclusion.}
In view of the previous steps, Banach's fixed point theorem yields the existence of a unique pair $(y^N,\Lambda^N)$ in the space $\mathcal{X}$ such that $(y^N,\Lambda^N)=\mathcal{T}(y^N,\Lambda^N)$; that is, denoting by $A^N(t):=A^N(0)+\int_0^t\Lambda^N(s)\de s$, the maps $t\mapsto y^N(t)$, $t\mapsto A^N(t)$ satisfy the two equations \eqref{fp2}--\eqref{fp1}. Moreover by construction $y^N_n(t)=0$ for all $n\geq N$.

Now, by integrating \eqref{fp2} in $(0,t)$ we have
\begin{equation*}
\begin{split}
A^N(t) - A^N(0)
& = - \int_0^t \frac{\de}{\de s} \bigl[ T^N_N(s;0)(y^N(0))\bigr] \de s \\
& \qquad\qquad - \int_0^t \de s \int_0^s \frac{\de A^N}{\de\xi}(\xi) \frac{\de}{\de s}\bigl[ T_N^N(s;\xi)( P(y^N(\xi)) ) \bigr] \de\xi \\
& \qquad\qquad - \sum_{i=1}^3\int_0^t \de s \int_0^s \frac{\de}{\de s}\bigl[ T_N^N(s;\xi)(r^{(i)}(y^N(\xi),A^N(\xi),\xi)) \bigr] \de\xi \\
& \qquad\qquad - \sum_{i=1}^3\int_0^t r^{(i)}_N(y^N(s),A^N(s),s)\de s \\
& = y^N_N(0) - T_N^N(t;0)(y^N(0)) \\
& \qquad\qquad + \int_0^t \frac{\de A^N}{\de\xi}(\xi) \Bigl( 2^Ny^N_N(\xi) - T_N^N(t;\xi)( P(y^N(\xi)) ) \Bigr) \de\xi \\
& \qquad\qquad + \sum_{i=1}^3\int_0^t \Bigl( r^{(i)}_N(y^N(\xi),A^N(\xi),\xi) - T^N_N(t;\xi)(r^{(i)}(y^N(\xi),A^N(\xi),\xi)) \Bigr) \de \xi \\
& \qquad\qquad - \sum_{i=1}^3\int_0^t r^{(i)}_N(y^N(s),A^N(s),s)\de s \,.
\end{split}
\end{equation*}
Recalling that $y^N_N(t)=0$ for all $t\geq0$, we see that the previous equation is exactly the identity \eqref{yn3}. Finally, combining \eqref{yn3} and \eqref{fp1}, we conclude that the pair $t\mapsto(y^N(t),A^N(t))$ satisfies \eqref{yn2}, and, in turn, \eqref{yneq}.

Finally, if we define the quantities
\begin{equation*}
\tilde{m}^N_n(t) := \bar{m}_n(A^N(t),p^N(t))\bigl(1+2^ny_n^N(t)\bigr) \qquad \text{for all }n\leq N,
\end{equation*}
we see that \eqref{yneq} implies that the functions $\tilde{m}^N_n(t)$ satisfy the same evolution equation as $m^N_n(t)$, with the same initial datum $\tilde{m}^N_n(0)=m^N_n(0)$ (see Lemma~\ref{lem:initialcond}). Therefore by uniqueness of the solution to this system of ODEs (which follows from the fact that we are considering a truncated problem) we conclude that $\tilde{m}^N(t)=m^N(t)$, that is the property in the statement holds.
\end{proof}

The following lemma contains the main estimates on the remainders needed in the proof of Proposition~\ref{prop:fp}.

\begin{lemma} \label{lem:fixedpointr}
Let $y\in\mathcal{Y}_\beta$, $A>0$, and let $r^{(i)}(y,A,t)$, $i=1,2,3$, be the sequences defined by \eqref{r1}, \eqref{r2}, and \eqref{r3}.
Assume also that $\|y\|_1\leq1$ and $A\geq \frac12A_M$. Then there exists a constant $C_3$, depending on $M$, $L_1$, $L_2$, $L_3$, such that
\begin{equation} \label{pfp30}
\begin{split}
\|r^{(1)}(y,A,t)\|_{\beta-1} &\leq C_3 \|y\|_{\beta}^2 + C_3 |A_M-A|\|y\|_\beta \,, \\
\|r^{(2)}(y,A,t)\|_{\bar{\theta}_2-\beta+1} & \leq C_3 \delta_0^{3/2} t^{-\bar{\theta}_2/\beta} e^{-\frac{\nu}{2} t} \,, \\
\|r^{(3)}(y,A,t)\|_{\bar{\theta}_1-\beta+1} &\leq C_3 \delta_0 t^{-\bar{\theta}_1/\beta} e^{-\frac{\nu}{2}t} \,,
\end{split}
\end{equation}
and
\begin{equation} \label{pfp30b}
\begin{split}
|r^{(1)}_N(y,A,t)| & \leq C_3 \|y\|_\beta\|y\|_1 \,,\\
|r^{(2)}_N(y,A,t)| & \leq C_3 \delta_0^{3/2} t^{-\frac{\beta-1}{\beta}}e^{-\nu t}\,, \\
|r^{(3)}_N(y,A,t)| & \leq C_3 \delta_0 t^{-\frac{\beta-1}{\beta}}e^{-\frac{\nu}{2} t}\,.
\end{split}
\end{equation}
Furthermore, for every $y^1,y^2\in\mathcal{Y}_\beta$ with $\|y^i\|_1\leq1$ and $A^1,A^2>\frac12 A_M$ we have
\begin{align} \label{pfp31}
\| r^{(1)}(y^1,A^1,t)-r^{(1)}(y^2,A^2,t) \|_{\beta-1} &\leq C_3 \|y^1-y^2\|_\beta \Bigl( \max\{ \|y^1\|_\beta, \|y^2\|_\beta\} + |A_M-A^1| \Bigr) \nonumber\\
& \qquad + C_3|A^1-A^2| \Bigl( \|y^2\|_\beta^2 + \|y^2\|_\beta \Bigr) \,, \nonumber \\
\| r^{(2)}(y^1,A^1,t)-r^{(2)}(y^2,A^2,t) \|_{\bar\theta_2-\beta+1} &\leq C_3\delta_0^{3/2} \Bigl( \|y^1-y^2\|_1 + |A^1-A^2| \Bigr) t^{-\bar\theta_2/\beta}e^{-\frac{\nu}{2} t}\,, \\
\| r^{(3)}(y^1,A^1,t)-r^{(3)}(y^2,A^2,t) \|_{\bar{\theta}_1-\beta+1} &\leq C_3\delta_0\|y^1-y^2\|_1 t^{-\bar{\theta}_1/\beta}e^{-\frac{\nu}{2}t} \,, \nonumber
\end{align}
and 
\begin{equation} \label{pfp31b}
\begin{split}
|r^{(1)}_N(y^1,A^1,t)-r^{(1)}_N(y^2,A^2,t)| & \leq C_3\bigl(\|y^1\|_1+\|y^2\|_1\bigr) \|y^1-y^2\|_\beta \,,\\
|r^{(2)}_N(y^1,A^1,t)-r^{(2)}_N(y^2,A^2,t)| & \leq C_3 \delta_0^2 \|y^1-y^2\|_\beta\,, \\
|r^{(3)}_N(y^1,A^1,t)-r^{(3)}_N(y^2,A^2,t)| & \leq C_3 \delta_0 \|y^1-y^2\|_\beta\,.
\end{split}
\end{equation}
\end{lemma}

\begin{proof}
Along the proof, the symbol $\lesssim$ will be used for inequalities up to constants which can depend only on the properties of the kernels, on $M$, $L_1$, $L_2$, $L_3$. We first consider the remainder $r^{(1)}$. The estimates \eqref{pfp30} and \eqref{pfp31} are proved in \cite[Lemma~6.3]{BNVd} (with minor modifications). For \eqref{pfp30b} and \eqref{pfp31b}, it is sufficient to observe that for $n=N$ the expression of $r^{(1)}_N$ simplifies and yields (using \eqref{kernel5})
\begin{equation*}
|r^{(1)}_N(y,A,t)| = \frac{2^N\gamma(2^{N+p^N_N(t)})}{16}y_{N-1}^2 \lesssim 2^{(\beta+1)N}y_{N-1}^2 \lesssim \|y\|_{\beta}\|y\|_1,
\end{equation*}
\begin{equation*}
\begin{split}
|r^{(1)}_N(y^1,A^1,t)-r^{(1)}_N(y^2,A^2,t)|
& \lesssim 2^{(\beta+1)N} |y^1_{N-1}-y^2_{N-1}| |y^1_{N-1}+y^2_{N-1}| \\
& \lesssim \bigl(\|y^1\|_1+\|y^2\|_1\bigr) \|y^1-y^2\|_\beta\,.
\end{split}
\end{equation*}

We next consider the term $r^{(2)}$. We have for $n<0$, using the bound $\|y\|_1\leq 1$, \eqref{kernel5bis}, the asymptotics of $\bar{\mu}_n$ as $n\to-\infty$, and the estimate \eqref{decayqn1},
\begin{equation*}
\begin{split}
|r^{(2)}_n(y,A,t)|
&\lesssim 2^{-n}\bigl( |q^N_{n-1}(t)| + |q^N_n(t)| \bigr) + 2^{-n}\bar{\mu}_n(A,p^N(t)) \bigl( |q^N_n(t)| + |q^N_{n+1}(t)|\bigr) \\
&\lesssim \delta_0^{3/2} 2^{-n}e^{-\nu t}.
\end{split}
\end{equation*}
Similarly, for $n\geq0$ we have, using \eqref{kernel5} and \eqref{decayqn1},
\begin{equation*}
\begin{split}
|r^{(2)}_n(y,A,t)|
&\lesssim 2^{(\beta-1)n}\bigl( |q^N_{n-1}(t)| + |q^N_n(t)| \bigr) + 2^{(\beta-1)n}\bar{\mu}_n(A,p^N(t)) \bigl( |q^N_n(t)| + |q^N_{n+1}(t)|\bigr) \\
&\lesssim \delta_0^{3/2} 2^{(\beta-1-\bar{\theta}_2)n}t^{-\bar{\theta}_2/\beta}e^{-\nu t} .
\end{split}
\end{equation*}
Then the estimate \eqref{pfp30} for $r^{(2)}$ follows.
For $n=N$, we first observe that by interpolating between the two estimates in \eqref{decayqn1} we have for all $n>0$
\begin{equation*}
2^{(\beta-1)n}|q_n^N(t)| = \Bigl(2^{\bar{\theta}_2n}|q_n^N(t)|\Bigr)^{\frac{\beta-1}{\bar{\theta}_2}} |q_n^N(t)|^{1-\frac{\beta-1}{\bar{\theta}_2}} \lesssim \delta_0^{3/2}t^{-\frac{\beta-1}{\beta}}e^{-\nu t},
\end{equation*}
which yields, using also $\|y\|_1\leq1$,
\begin{equation*}
\begin{split}
|r^{(2)}_N(y,A,t)|
& = \frac{\gamma(2^{N+p^N_N(t)})}{2^{N+2}}\Big| (1+2^{N-1}y_{N-1})^2O(q^N_{N-1}(t)) - (1+2^Ny_N)O(q^N_N(t)) \Big| \\
& \lesssim 2^{(\beta-1)N}\Bigl( |q^N_{N-1}(t)| + |q^N_N(t)| \Bigr)
\lesssim \delta_0^{3/2}t^{-\frac{\beta-1}{\beta}}e^{-\nu t},
\end{split}
\end{equation*}
which is the second estimate in \eqref{pfp30b}. To prove the Lipschitz continuity of $r^{(2)}$ (estimate \eqref{pfp31}), we first observe that in view of the explicit expression of $\bar{\mu}_n$ in \eqref{nearlystat5} and of the assumption $A^1,A^2\geq \frac12 A_M$ one can show that
\begin{equation} \label{estmun}
|\bar{\mu}_n(A^1,p^N)-\bar{\mu}_n(A^2,p^N)| \lesssim 2^n e^{-\frac12 A_M 2^n}|A^1-A^2|.
\end{equation}
Then the bound in \eqref{pfp31} can be obtained straightforwardly, using this estimate and \eqref{decayqn1}. We obtain the estimate in \eqref{pfp31b} for $r^{(2)}$ by using the trivial bound $|q^N_n(t)|\lesssim\delta_0^2$:
\begin{equation*}
\begin{split}
|r^{(2)}_N(y^1,A^1,t)-r^{(2)}_N(y^2,A^2,t)|
& \lesssim 2^{(\beta-1)N} \big|(1+2^{N-1}y_{N-1}^1)^2-(1+2^{N-1}y_{N-1}^2)^2\big| |q^N_{N-1}(t)| \\
& \qquad + 2^{\beta N} \big|y_N^1-y_N^2\big| |q^N_{N}(t)| \\
& \lesssim \delta_0^2\|y^1-y^2\|_\beta\,.
\end{split}
\end{equation*}

We eventually consider the remainder $r^{(3)}$: by \eqref{nearlystat6}, \eqref{decaypn2}, and the assumption $\|y\|_1\leq1$,
\begin{align*}
|r^{(3)}_n(y,A,t)|
& \lesssim (1+\|y\|_1)\delta_0t^{-\bar{\theta}_1/\beta}e^{-\frac{\nu}{2}t} \biggl( \sum_{k=n\wedge1}^0 2^{-k} + \sum_{k=n\vee 1}^N 2^{-k}2^{(\beta-\bar{\theta}_1)k} \biggr) \\
& \lesssim
\begin{cases}
\delta_0 2^{-n} t^{-\bar{\theta}_1/\beta} e^{-\frac{\nu}{2}t} &\text{for $n\leq0$,}\\
\delta_0 2^{(\beta-1-\bar{\theta}_1)n} t^{-\bar{\theta}_1/\beta}e^{-\frac{\nu}{2}t} &\text{for $n>0$,}
\end{cases} 
\end{align*}
from which the last estimate in \eqref{pfp30} follows.
We next observe that by interpolating between the two estimates in \eqref{decaypn2} we have for $n>0$
\begin{equation*}
2^{-n}\bigg| \frac{\de p^N_n}{\de t}(t)\bigg|
\leq \biggl( 2^{(\bar{\theta}_1-\beta)n}\bigg| \frac{\de p^N_n}{\de t}(t)\bigg| \biggr)^{\frac{\beta-1}{\bar{\theta}_1}}
\biggl( 2^{-\beta n}\bigg|\frac{\de p^N_n}{\de t}(t)\bigg|\biggr)^{1-\frac{\beta-1}{\bar{\theta}_1}}
\lesssim \delta_0 t^{-\frac{\beta-1}{\beta}} e^{-\frac{\nu}{2}t},
\end{equation*}
hence for $n=N$ we obtain the third estimate in \eqref{pfp30b} (using also \eqref{nearlystat6}):
\begin{equation*}
|r^{(3)}_N(y,A,t)| \lesssim (1+\|y\|_1)2^{-N}\bigg| \frac{\de p^N_N}{\de t}(t)\bigg|
\lesssim \delta_0t^{-\frac{\beta-1}{\beta}}e^{-\frac{\nu}{2}t}.
\end{equation*}
Observe that, in view of \eqref{nearlystat7}, the quantity $\frac{1}{\bar{m}_n(A,p^N(t))}\frac{\partial\bar{m}_n}{\partial p_k}(A,p^N(t))$ is actually independent of $A$: then using \eqref{nearlystat6} and \eqref{decaypn2} we find
\begin{align*}
|r^{(3)}_n(y^1,A^1,t)-r^{(3)}_n(y^2,A^2,t)|
& = \bigg| (y_n^1-y_n^2)\sum_{k=n}^\infty \frac{1}{\bar{m}_n}\frac{\partial\bar{m}_n}{\partial p_k}\frac{\de p^N_k}{\de t}\bigg| \\
& \lesssim
\begin{cases}
\delta_0 2^{-n}2^n|y_n^1-y_n^2| t^{-\bar{\theta}_1/\beta}e^{-\frac{\nu}{2}t} &\text{for $n\leq0$,}\\
\delta_0 2^{(\beta-1-\bar{\theta}_1)n}2^n|y_n^1-y_n^2| t^{-\bar{\theta}_1/\beta}e^{-\frac{\nu}{2} t} &\text{for $n>0$,}
\end{cases} 
\end{align*}
so that also the third estimate in \eqref{pfp31} holds.
For $n=N$, arguing similarly we obtain the last estimate in \eqref{pfp31b}.
\end{proof}


\subsection{Continuation argument} \label{subsect:continuation}
The next goal is to extend the representation \eqref{fp} of the functions $m_n^N(t)$, obtained in Proposition~\ref{prop:fp}, for all positive times. This will be achieved by a continuation argument. Indeed, recall that the fundamental assumption in Proposition~\ref{prop:fp} is that the maps $p^N(t)$, $q^N(t)$ satisfy the estimates \eqref{decaypn1}--\eqref{decayqn1} in the time interval $[0,t^N]$. The idea is now to show that \emph{at the time $t^N$} the same estimates \eqref{decaypn1}--\eqref{decayqn1} hold with strict inequality, and therefore they can be extended for larger times $t\in[t^N,t^N+\e]$; in turn, this condition allows to repeat the proof of Proposition~\ref{prop:fp} and to extend also the representation \eqref{fp} in $[t^N,t^N+\e]$.

In order to prove the claim, we take advantage of the representation \eqref{fp} in order to write the evolution equations for $p^N$ and $q^N$ in a handier form. In Lemma~\ref{lem:approx} we computed the equations \eqref{pneq2}--\eqref{qneq2} for the (not truncated) functions $p_n(t)$, $q_n(t)$; then, at the end of Section~\ref{sect:strategy} we have seen that those equations can be written in the form \eqref{pneq3}--\eqref{qneq3} under the assumption that $m_n(t)$ can be represented as in \eqref{yn}. Now, since the truncated functions $m^N_n(t)$ satisfy \eqref{fp} for $t\in[0,t^N]$, the very same equations hold for $p^N_n(t)$, $q^N_n(t)$ for all $t\in[0,t^N]$ and for all $n<N$:
\begin{multline} \label{pneqN}
\frac{\de p_n^N}{\de t}
 = \frac{\gamma(2^{n+p^N_n})}{4} \biggl[ \frac{(1+2^{n-1}y^N_{n-1})^2}{(1+2^ny^N_n)} \Bigl( (p^N_{n-1}-p^N_n) + O(q^N_{n-1}) \Bigr) + O(q_n^N) \\
 \qquad -4\bar{\mu}_{n}\frac{\gamma(2^{n+1+p^N_{n+1}})}{\gamma(2^{n+p^N_n})} \biggl( \frac{(1+2^{n+1}y^N_{n+1})}{(1+2^ny^N_n)}\Bigl( (p^N_{n}-p^N_{n+1}) + O(q^N_{n+1}) \Bigr) + (1+2^ny_n^N)O(q_n^N) \biggr) \biggr] ,
\end{multline}
\begin{align} \label{qneqN}
\frac{\de q_n^N}{\de t}
&= \frac{\gamma(2^{n+p^N_n})}{4} \biggl[ \frac{(1+2^{n-1}y^N_{n-1})^2}{(1+2^ny^N_n)} \Bigl( \frac{q^N_{n-1}}{2}-q^N_n +\delta_0O(q^N_{n-1}) + (p^N_{n-1}-p^N_n)^2\Bigr) +\delta_0O(q^N_n) \nonumber \\
& \qquad -4\bar{\mu}_n\frac{\gamma(2^{n+1+p^N_{n+1}})}{\gamma(2^{n+p^N_n})} \frac{(1+2^{n+1}y^N_{n+1})}{(1+2^ny^N_n)}
\Bigl( (q^N_{n}-q^N_{n+1}) + \delta_0O(q^N_{n+1}) - (p^N_{n+1}-p^N_n)^2 \Bigr) \nonumber \\
& \qquad -4\bar{\mu}_n\frac{\gamma(2^{n+1+p^N_{n+1}})}{\gamma(2^{n+p^N_n})} (1+2^ny^N_n)\delta_0O(q^N_n) \biggr],
\end{align}
where $\bar{\mu}_n=\bar{\mu}_n(A^N(t),p^N(t))$. For $n=N$, recalling also that $y^N_N=0$,
\begin{equation} \label{pneqNb}
\frac{\de p_N^N}{\de t} = \frac{\gamma(2^{N+p_N^N})}{4} \biggl[ (1+2^{N-1}y^N_{N-1})^2 \Bigl( (p^N_{N-1}-p^N_N) + O(q^N_{N-1}) \Bigr) + O(q^N_N) \biggr],
\end{equation}
\begin{multline} \label{qneqNb}
\frac{\de q^N_N}{\de t}
= \frac{\gamma(2^{N+p^N_N})}{4} \biggl[ (1+2^{N-1}y^N_{N-1})^2 \Bigl( \frac{q^N_{N-1}}{2}-q^N_N + \delta_0O(q^N_{N-1}) + (p^N_{N-1}-p^N_N)^2 \Bigr) + \delta_0O(q^N_N) \biggr],
\end{multline}
and $p^N_n(t)=q^N_n(t)=0$ for $n>N$.

\begin{proposition} \label{prop:continuation}
There exist positive constants $L_1$, $L_2$, $L_3$, depending only on $M$, and $\delta_0>0$ sufficiently small, such that the estimates \eqref{decaypn1}, \eqref{decaypn2}, \eqref{decayqn1} and the conclusion of Proposition~\ref{prop:fp} hold for all $t>0$.
\end{proposition}

\begin{proof}
We let $\overline{T}$ be the supremum of the times $T>0$ such that the estimates \eqref{decaypn1}, \eqref{decaypn2}, \eqref{decayqn1} hold for every $t\in(0,T)$.
Notice that $\overline{T}>0$, as the estimates are satisfied in $(0,t^N)$. The proof amounts to show that $\overline{T}=\infty$: indeed, this allows to repeat the proof of Proposition~\ref{prop:fp} in the time interval $(0,\infty)$. We assume by contradiction that $\overline{T}<\infty$, and we will show that \eqref{decaypn1}--\eqref{decayqn1} hold at $t=\overline{T}$ with strict inequality: this would allow to extend the estimates for larger times, leading to a contradiction.
	
In view of the fact, already observed, that the proof of Proposition~\ref{prop:fp} can be repeated in the time interval $(0,\overline{T})$, the sequences $p^N(t)$, $q^N(t)$ obey the evolution equations \eqref{pneqN}--\eqref{qneqNb} for $t\in(0,\overline{T})$.
As usual, we denote by $C$ a constant which can depend only on the properties of the kernels and on $M$, and might change from line to line.

\smallskip\noindent\textit{Step 1: decay of $p^N$.}
We first show that \eqref{decaypn1} holds at $t=\overline{T}$ with strict inequality, with the choice
\begin{equation} \label{L1}
L_1\geq\max\Bigl\{C_1(M,0,0), C_1(M,0,\bar{\theta}_1)\Bigr\}
\end{equation}
(where $C_1$ is the constant given by Theorem~\ref{thm:lineartrunc}).
We highlight the leading order linear operator in the equations \eqref{pneqN}, \eqref{pneqNb} for $p^N(t)$: for $n\leq N$
\begin{equation} \label{pcont1}
\frac{\de p_n^N}{\de t} = \frac{\gamma(2^{n+p^N_n})}{4}\Bigl( p^N_{n-1}-p^N_n - \sigma^N_n(t)(p^N_n-p^N_{n+1}) \Bigr) + R^{1}_n(t) + R^{2}_n(t) + R^{3}_n(t)
\end{equation}
where we introduced the following quantities:
\begin{equation} \label{pcont3}
\sigma_n^N(t) :=
\begin{cases}
4\bar{\mu}_n(A_M,p^N(t))\frac{\gamma(2^{n+1+p^N_{n+1}(t)})}{\gamma(2^{n+p^N_n(t)})}  & \text{if } n<N,\\
0 & \text{if }n=N,
\end{cases}
\end{equation}
\begin{multline} \label{pcont4}
R^{1}_n(t) := \frac{\gamma(2^{n+p^N_n})}{4} \biggl(\frac{(1+2^{n-1}y^N_{n-1})^2}{(1+2^ny^N_n)}-1\biggr) (p^N_{n-1}-p^N_n) \\
-\bar{\mu}_n(A^N(t),p^N(t))\gamma(2^{n+1+p^N_{n+1}}) \biggl( \frac{(1+2^{n+1}y^N_{n+1})}{(1+2^ny^N_n)}-1\biggr) (p^N_{n}-p^N_{n+1}) ,
\end{multline}
\begin{multline} \label{pcont4b}
R^{2}_n(t)
:= \frac{\gamma(2^{n+p^N_n})}{4} \biggl[ \frac{(1+2^{n-1}y^N_{n-1})^2}{(1+2^ny^N_n)} O(q^N_{n-1}) + O(q_n^N) \biggr] \\
-\bar{\mu}_n(A^N(t),p^N(t))\gamma(2^{n+1+p^N_{n+1}}) \biggl[ \frac{(1+2^{n+1}y^N_{n+1})}{(1+2^ny^N_n)} O(q^N_{n+1}) + (1+2^ny_n^N)O(q_n^N)\biggr],
\end{multline}
\begin{equation} \label{pcont4c}
R^{3}_n(t)
:= \gamma(2^{n+1+p^N_{n+1}}) \Bigl[ \bar{\mu}_n(A_M,p^N(t)) - \bar{\mu}_n(A^N(t),p^N(t)) \Bigr] \bigl(p^N_n-p^N_{n+1})
\end{equation}
(for $n=N$ the terms containing $\bar{\mu}_n$ are not present in $R^1$, $R^2$, $R^3$).
The linearized operator in \eqref{pcont1} is of the form considered in Section~\ref{subsect:lineartrunc}, see \eqref{tlinear1}; since also the assumption \eqref{asspn} is satisfied, we are
in the position to apply Theorem~\ref{thm:lineartrunc}. We can write the solution to \eqref{pcont1} in terms of the solution to the linearized problem as
\begin{equation}
p^N_n(t) = T^N_n(t;0)(p^N(0)) + \sum_{i=1}^3\int_0^t T^N_n(t;s)(R^{i}(s))\de s
\end{equation}
(see \eqref{tlinear4} for the definition of the operator $T^N$). By applying Theorem~\ref{thm:lineartrunc} and using the estimates \eqref{contrp1}--\eqref{contrp3} proved in Lemma~\ref{lem:continuation} below we then find
\begin{align*}
\big\| D^+\bigl(p^N(t)\bigr) \big\|_{0}
& \leq \big\| D^+\bigl(T^N(t;0)(p^N(0))\bigr) \big\|_{0} + \sum_{i=1}^3\int_0^t \big\| D^+\bigl(T^N(t;s)(R^{i}(s)) \bigr) \big\|_{0} \de s \\
& \leq C_1(M,0,0) \|p^N(0)\|_0 e^{-\nu t} \\
& \qquad + \sum_{i=1}^3 C_1(M,\bar{\theta}_1-\beta,0)\int_0^t \big\| R^{i}(s) \big\|_{\bar{\theta}_1-\beta}(t-s)^{-\frac{\beta-\bar{\theta}_1}{\beta}}e^{-\nu(t-s)} \de s \\
& \leq L_1\delta_0 e^{-\nu t} + C\delta_0^{3/2} \int_0^t \bigl(1+s^{-\bar{\theta}_1/\beta}\bigr)e^{-\frac{\nu}{2} s}(t-s)^{-\frac{\beta-\bar{\theta}_1}{\beta}}e^{-\nu(t-s)} \de s \,.
\end{align*}
It can be checked by elementary arguments that the integral on the right-hand side of the previous inequality is bounded by $Ce^{-\frac{\nu}{2}t}$. Then
\begin{equation*}
\big\| D^+\bigl(p^N(t)\bigr) \big\|_{0}
\leq L_1\delta_0 e^{-\nu t} + C\delta_0^{3/2} e^{-\frac{\nu}{2}t}
\leq \frac32 L_1\delta_0e^{-\frac{\nu}{2}t} 
\end{equation*}
for all $t\in(0,\overline{T})$, by choosing $\delta_0$ small enough. Therefore the first estimate in \eqref{decaypn1} holds with strict inequality at $t=\overline{T}$.
By the same argument, using also \eqref{contrp1bis}, we have
\begin{align*}
\big\| D^+\bigl(p^N(t)\bigr) \big\|_{\bar{\theta}_1}
& \leq \big\| D^+\bigl(T^N(t;0)(p^N(0))\bigr) \big\|_{\bar{\theta}_1} + \sum_{i=1}^3\int_0^t \big\| D^+\bigl(T^N(t;s)(R^{i}(s)) \bigr) \big\|_{\bar{\theta}_1} \de s \\
& \leq C_1(M,0,\bar{\theta}_1) \|p^N(0)\|_0 t^{-\bar{\theta}_1/\beta} e^{-\nu t} \\
& \qquad + C_1(M,\bar{\theta}_1-1,\bar{\theta}_1)\int_0^t \big\| R^{1}(s) \big\|_{\bar{\theta}_1-1} (t-s)^{-\frac{1}{\beta}}e^{-\nu(t-s)} \de s \\
& \qquad + C_1(M,\bar{\theta}_2-\beta,\bar{\theta}_1) \int_0^t \big\| R^{2}(s) \big\|_{\bar{\theta}_2-\beta}(t-s)^{-\frac{\beta-(\bar{\theta}_2-\bar{\theta}_1)}{\beta}}e^{-\nu(t-s)} \de s \\
& \qquad + C_1(M,0,\bar{\theta}_1)\int_0^t \|R^3(s)\|_{0}(t-s)^{-\bar{\theta}_1/\beta}e^{-\nu(t-s)}\de s \\
& \leq L_1\delta_0 t^{-\bar{\theta}_1/\beta}e^{-\nu t} \\
& \qquad + C\delta_0^2 \int_0^t \bigl(1+s^{-\frac{\beta-1}{\beta}}\bigr)s^{-\bar{\theta}_1/\beta}e^{-\frac{\nu}{2}s}(t-s)^{-\frac{1}{\beta}}e^{-\nu(t-s)} \de s \\
& \qquad + C\delta_0^{3/2} \int_0^t  \bigl(1+s^{-\bar{\theta}_2/\beta}\bigr) e^{-\nu s} (t-s)^{-\frac{\beta-(\bar{\theta}_2-\bar{\theta}_1)}{\beta}}e^{-\nu(t-s)} \de s \\
& \leq L_1\delta_0 t^{-\bar{\theta}_1/\beta}e^{-\nu t} + C\delta_0^{3/2}t^{-\bar{\theta}_1/\beta} e^{-\frac{\nu}{2}t}
\leq \frac32 L_1\delta_0 t^{-\bar{\theta}_1/\beta}e^{-\nu t}
\end{align*}
for all $t\in(0,\overline{T})$, by choosing $\delta_0$ small enough. Therefore also the second estimate in \eqref{decaypn1} holds with strict inequality at $t=\overline{T}$.

\smallskip\noindent\textit{Step 2: decay of $\frac{\de p^N}{\de t}$.}
We next choose $L_2$ such that
\begin{equation} \label{L2}
L_2\geq 2L_1 \biggl( \sup_{\xi\leq2}\gamma(\xi)+ \sup_{\xi\geq1}\xi^{-\beta}\gamma(\xi) \biggr) \Bigl(1+\sup_{n,t}\sigma_n^N(t)\Bigr).
\end{equation}
Going back to the equation \eqref{pcont1}, we see that we can bound, using the estimates \eqref{contrp1}--\eqref{contrp3} in Lemma~\ref{lem:continuation} below and the definition of $L_2$,
\begin{align*}
\Big\|\frac{\de p^N}{\de t}(t)\Big\|_{-\beta}
& \leq \frac{L_2}{2L_1}\|D^+(p^N(t))\|_{0} + \| R^{1}(t) \|_{-\beta} + \|R^{2}(t) \|_{-\beta} + \|R^{3}(t) \|_{-\beta}\\
& \leq L_2\delta_0 e^{-\frac{\nu}{2}t} + 5C_4\delta_0^{3/2}e^{-\frac{\nu}{2}t}
\leq \frac32L_2\delta_0 e^{-\frac{\nu}{2}t} 
\end{align*}
by choosing $\delta_0$ small enough, and similarly
\begin{align*}
\Big\|\frac{\de p^N}{\de t}(t)\Big\|_{\bar{\theta}_1-\beta}
& \leq \frac{L_2}{2L_1}\|D^+p^N(t)\|_{\bar\theta_1} + \| R^{1}(t) \|_{\bar{\theta}_1-\beta} + \|R^{2}(t) \|_{\bar{\theta}_1-\beta} + \|R^{3}(t) \|_{\bar{\theta}_1-\beta} \\
& \leq L_2\delta_0t^{-\bar{\theta}_1/\beta}e^{-\frac{\nu}{2}t} + 3C_4\delta_0^{3/2}(1+t^{-\bar{\theta}_1/\beta})e^{-\frac{\nu}{2}t} \\
& \leq \frac32L_2\delta_0 \bigl(1+t^{-\bar{\theta}_1/\beta}\bigr)e^{-\frac{\nu}{2}t}.
\end{align*}
These two estimates imply that \eqref{decaypn2} holds with strict inequality at $t=\overline{T}$, as claimed.

\smallskip\noindent\textit{Step 3: decay of $q^N$.}
It remains to prove the decay \eqref{decayqn1} of $q^N$. This will be obtained by comparison with an explicit supersolution for the equation \eqref{qneqN}, \eqref{qneqNb} satisfied by $q^N$.

We first consider the sequence $\bar{q}_n(t)=4\delta_0^{3/2}e^{-\nu t}$: thanks to the bounds \eqref{decaypn1}, \eqref{decayqn1} and \eqref{fp3}, one can see that, possibly taking a smaller $\nu>0$ (depending only on the kernels), the sequence $\bar{q}_n$ is a supersolution for the equation \eqref{qneqN}, namely for $t\in[0,\overline{T}]$ and $n\leq N$
\begin{align*}
\frac{\de \bar{q}_n}{\de t}
&\geq \frac{\gamma(2^{n+p^N_n})}{4} \biggl[ \frac{(1+2^{n-1}y^N_{n-1})^2}{(1+2^ny^N_n)} \Bigl( \frac{\bar{q}_{n-1}}{2}-\bar{q}_n +\delta_0O(q^N_{n-1}) + (p^N_{n-1}-p^N_n)^2\Bigr) +\delta_0O(q^N_n) \nonumber \\
& \qquad -\sigma_n^N(t)\frac{(1+2^{n+1}y^N_{n+1})}{(1+2^ny^N_n)}
\Bigl( (\bar{q}_{n}-\bar{q}_{n+1}) + \delta_0O(q^N_{n+1}) - (p^N_{n+1}-p^N_n)^2 \Bigr) \nonumber \\
& \qquad - \sigma_n^N(t) (1+2^ny^N_n)\delta_0O(q^N_n) \biggr] .
\end{align*}
It follows that $w_n(t):=\bar{q}_n(t)+\e2^{-n}-q^N_n(t)$ satisfies, for $t\in[0,\overline{T}]$ and $n\leq N$,
\begin{equation*}
\frac{\de w_n}{\de t} \geq
\frac{\gamma(2^{n+p^N_n})}{4} \biggl[ \frac{(1+2^{n-1}y^N_{n-1})^2}{(1+2^ny^N_n)} \Bigl( \frac{w_{n-1}}{2}-w_n \Bigr)
-\sigma_n^N(t)\frac{(1+2^{n+1}y^N_{n+1})}{(1+2^ny^N_n)}(w_n-w_{n+1}) \biggr]
\end{equation*}
with $w_n(0)\geq0$ by \eqref{decaypn0}, $w_{N+1}(t)\geq0$, $w_{-N_1}(t)\geq0$ for $N_1$ large enough, depending on $\e$; hence by applying the maximum principle in the region $n\in[-N_1,N]$, $t\in[0,\overline{T}]$ we obtain $w_n(t)\geq0$, which yields (by passing to the limit first as $N_1\to\infty$, then as $\e\to0$)
\begin{equation*}
\bar{q}_n(t)\geq q_n^N(t) \qquad\text{for all $n\leq N$ and $t\in[0,\overline{T}]$.}
\end{equation*}
Hence the first estimate in \eqref{decayqn1} holds with strict inequality at $t=\overline{T}$.

Next, we let $n_0$ be given by Lemma~\ref{lem:continuation2} below. Notice that the first estimate in \eqref{decayqn1} yields
\begin{equation} \label{qcont2}
\sup_{0<n\leq n_0} 2^{\bar{\theta}_2n}|q_n^N(t)| \leq 2^{\bar{\theta}_2n_0+2}\delta_0^{3/2}e^{-\nu t}.
\end{equation}
We then let $\hat{q}_n$ be the solution to the initial/boundary value problem \eqref{hatqn1}. If $\delta_0$ is small enough, one can show that $\hat{q}_n$ is a supersolution for the equation \eqref{qneqN} solved by $q_n^N$, in the sense that for $n>n_0$ and $t\in[0,\overline{T}]$
\begin{align*}
\frac{\de \hat{q}_n}{\de t}
& = \frac{\gamma(2^n)}{4} \Bigl( \frac12(1+\delta_1)\hat{q}_{n-1} - (1-\delta_1)\hat{q_n} \Bigr) \\
&\geq \frac{\gamma(2^{n+p^N_n})}{4} \biggl[ \frac{(1+2^{n-1}y^N_{n-1})^2}{(1+2^ny^N_n)} \Bigl( \frac{\hat{q}_{n-1}}{2}-\hat{q}_n +\delta_0O(q^N_{n-1}) + (p^N_{n-1}-p^N_n)^2\Bigr) +\delta_0O(q^N_n) \nonumber \\
& \qquad -\sigma_n^N(t)\frac{(1+2^{n+1}y^N_{n+1})}{(1+2^ny^N_n)}
\Bigl( ({q}^N_{n}-{q}^N_{n+1}) + \delta_0O(q^N_{n+1}) - (p^N_{n+1}-p^N_n)^2 \Bigr) \nonumber \\
& \qquad - \sigma_n^N(t) (1+2^ny^N_n)\delta_0O(q^N_n) \biggr] .
\end{align*}
Indeed, by using the decay estimates \eqref{decaypn1} and \eqref{decayqn1} for $D^+(p^N)$ and $q^N$, the estimate \eqref{fp3} on $y^N$, the fast decay \eqref{contsigma} of $\sigma_n^N$, and the estimate from below \eqref{hatqn2} on $\hat{q}_n$, one can show that all the terms on the right-hand side can be bounded in terms of $C(\delta_0)\hat{q}_n$, where $C(\delta_0)$ can be made arbitrarily small by choosing $\delta_0$ small enough.

Hence the function $w_n(t)=\hat{q}_n(t)-q^N_n(t)$ satisfies for $t\in[0,\overline{T}]$ and $n> n_0$
\begin{equation*}
\frac{\de w_n}{\de t} \geq
\frac{\gamma(2^{n+p^N_n})}{4} \frac{(1+2^{n-1}y^N_{n-1})^2}{(1+2^ny^N_n)} \Bigl( \frac{w_{n-1}}{2}-w_n \Bigr),
\end{equation*}
with $w_n(0)\geq0$, $w_{n_0}(t)\geq0$ (by \eqref{qcont2}). The maximum principle gives $w_n(t)\geq0$ for all $t\in[0,\overline{T}]$ and $n>n_0$, that is, in view of \eqref{hatqn2},
\begin{equation*}
\sup_{n>n_0} 2^{\bar\theta_2n}|q_n^N(t)| \leq c_2\delta_0^{3/2}t^{-\bar{\theta}_2/\beta}e^{-\nu t}.
\end{equation*}
By combining this estimate with \eqref{qcont2}, we eventually find that also the second estimate in \eqref{decayqn1} holds at $t=\overline{T}$ with strict inequality, as claimed, choosing $L_3=\max\{c_2,2^{\bar{\theta}_2n_0+2} \}$.
\end{proof}

The following two lemmas are instrumental in the proof of Proposition~\ref{prop:continuation}.

\begin{lemma} \label{lem:continuation}
Let $R^{1}$, $R^{2}$, $R^{3}$ be the sequences defined in \eqref{pcont4}, \eqref{pcont4b}, \eqref{pcont4c} respectively. Then there exists a constant $C_4$, depending on $M$, $L_1$, $L_2$, $L_3$, such that for all  $t\in(0,\overline{T})$
\begin{equation} \label{contrp1}
\|R^{1}(t)\|_{\theta-\beta} \leq C_4\delta_0^2 \bigl(1+t^{-\theta/\beta}\bigr)e^{-\frac{\nu}{2} t},
\qquad\text{for all $\theta\in[0,\bar{\theta}_1]$,}
\end{equation}
\begin{equation} \label{contrp2}
\|R^{2}(t)\|_{\theta-\beta} \leq C_4\delta_0^{3/2} \bigl(1+t^{-\theta/\beta}\bigr)e^{-\nu t},
\qquad\text{for all $\theta\in[0,\bar{\theta}_2]$,}
\end{equation}
\begin{equation} \label{contrp3}
\|R^{3}(t)\|_{\theta-\beta} \leq C_4\delta_0^2 e^{-\frac{\nu}{2} t},
\qquad\text{for all $\theta\in[0,\beta]$,}
\end{equation}
\begin{equation} \label{contrp1bis}
\|R^{1}(t)\|_{\bar{\theta}_1-1} \leq C_4\delta_0^2 \bigl(1+t^{-\frac{\beta-1}{\beta}}\bigr)t^{-\bar{\theta}_1/\beta}e^{-\frac{\nu}{2}t}.
\end{equation}
\end{lemma}

\begin{proof}
Along the proof, the symbol $\lesssim$ will be used for inequalities up to constants which can depend only on $M$, $L_1$, $L_2$, $L_3$. We remark that, in view of the explicit expression \eqref{nearlystat5} of $\bar{\mu}_n$ and of the fact that, by construction, $A^N(t)\geq\frac{A_M}{2}$, we have
\begin{equation} \label{contsigma}
\limsup_{n\to-\infty} | \bar{\mu}_n(A^N(t),p^N(t))-1|\lesssim\delta_0,
\quad
\bar{\mu}_n(A^N(t),p^N(t)) = O(e^{-\frac12 A_M2^n}) \quad\text{as $n\to\infty$.}
\end{equation}
The estimates below follow essentially by using the assumptions \eqref{kernel5}--\eqref{kernel5bis} on $\gamma$, the asymptotics \eqref{contsigma} of $\bar{\mu}_n$, the bounds \eqref{fp3} on $y^N$ and $A^N$, and the estimates \eqref{decaypn1}--\eqref{decayqn1} on $p^N$, $q^N$ (which by assumption hold for $t\in(0,\overline{T})$).

We first consider $R^{1}$. Observe that by interpolating between the two estimates in \eqref{decaypn1} we have for $\theta\in[0,\bar{\theta}_1]$
\begin{equation*}
\|D^+(p^N(t))\|_{\theta} \leq \bigl(\|D^+(p^N(t))\|_{\bar\theta_1}\bigr)^{\frac{\theta}{\bar{\theta}_1}} \bigl(\|D^+(p^N(t))\|_{0}\bigr)^{1-\frac{\theta}{\bar{\theta}_1}} \lesssim \delta_0 t^{-\theta/\beta}e^{-\frac{\nu}{2}t}.
\end{equation*}
For $n\leq0$ we then find
\begin{align*}
|R^{1}_n(t)|
& \lesssim \Bigl(2^n|y^N_{n-1}-y^N_n| + 2^{2n}|y^N_{n-1}|^2 \Bigr)|p^N_{n-1}-p^N_n| + |2^{n+1}y^N_{n+1}-2^ny^N_n||p^N_n-p^N_{n+1}| \\
& \lesssim 2^{-n}\|y^N(t)\|_1 \|D^+(p^N(t))\|_{0} \lesssim 2^{-n}\delta_0^2 e^{-\frac{\nu}{2}t},
\end{align*}
whereas for $n\geq0$
\begin{align*}
|R^{1}_n(t)|
& \lesssim  2^{\beta n}\Bigl(2^n|y^N_{n-1}-y^N_n| + 2^{2n}|y^N_{n-1}|^2 \Bigr)|p^N_{n-1}-p^N_n| \\
& \qquad + 2^{\beta n} e^{-\frac12 A_M2^n} |2^{n+1}y^N_{n+1}-2^ny^N_n||p^N_n-p^N_{n+1}| \\
& \lesssim 2^{(\beta-\theta)n}\|y^N(t)\|_1 \|D^+(p^N(t))\|_{\theta} \lesssim 2^{(\beta-\theta)n}\delta_0^2 t^{-\theta/\beta}e^{-\frac{\nu}{2}t}.
\end{align*}
The previous estimates combined yield \eqref{contrp1}.

For the proof of \eqref{contrp1bis} we estimate as before ($n>0$)
\begin{align*}
|R^{1}_n(t)|
& \lesssim 2^{(1-\bar{\theta}_1)n}\|y^N(t)\|_\beta \|D^+(p^N(t))\|_{\bar{\theta}_1} \\
& \lesssim 2^{(1-\bar{\theta}_1)n} \delta_0^2 \bigl(1+t^{-\frac{\beta-1}{\beta}}\bigr)e^{-\frac{\nu}{2}t}t^{-\bar{\theta}_1/\beta}e^{-\frac{\nu}{2}t}.
\end{align*}

We next consider $R^{2}$: we first observe that by interpolating between the two estimates in \eqref{decayqn1} we have for all $n>0$ and $\theta\in[0,\bar{\theta}_2]$
\begin{equation*}
2^{\theta n}|q_n^N(t)| = \Bigl(2^{\bar{\theta}_2n}|q_n^N(t)|\Bigr)^{\frac{\theta}{\bar{\theta}_2}} |q_n^N(t)|^{1-\frac{\theta}{\bar{\theta}_2}} \lesssim \delta_0^{3/2}t^{-\theta/\beta}e^{-\nu t}.
\end{equation*}
Hence
\begin{align*}
|R^{2}_n(t)| &\lesssim |q^N_{n-1}(t)| + |q^N_{n}(t)| + |q^N_{n+1}(t)| \lesssim \delta_0^{3/2} e^{-\nu t} & & \text{for $n<0$,} \\
|R^{2}_n(t)| &\lesssim 2^{\beta n} \Bigl( |q^N_{n-1}(t)| + |q^N_{n}(t)| + |q^N_{n+1}(t)| \Bigr) \lesssim 2^{(\beta-\theta)n}\delta_0^{3/2}t^{-\theta/\beta}e^{-\nu t} & & \text{for $n\geq0$,}
\end{align*}
and the two estimates combined yield \eqref{contrp2}.

Finally, to prove the estimate \eqref{contrp3} for $R^{3}$, we recall \eqref{estmun} and we obtain
\begin{align*}
2^n|R^{3}_n(t)| &\lesssim \delta_02^n |p^N_n-p^N_{n+1}| \lesssim \delta_0^{2} e^{-\frac{\nu}{2} t} & & \text{for $n<0$,} \\
2^{(\theta-\beta)n}|R^{3}_n(t)| &\lesssim \delta_0 2^{\theta n}2^ne^{-\frac12 A_M2^n} |p^N_n-p^N_{n+1}| \lesssim \delta_0|p^N_n-p^N_{n+1}| \lesssim \delta_0^{2} e^{-\frac{\nu}{2} t} & & \text{for $n\geq0$.}
\end{align*}
The conclusion follows.
\end{proof}

\begin{lemma} \label{lem:continuation2}
There exist $\delta_1>0$, $n_0\in\N$, and $c_2>c_1>0$ such that if $\hat{q}(t)=\{\hat{q}_n(t)\}_{n\in\Z}$ solves
\begin{equation} \label{hatqn1}
\begin{cases}
\frac{\de\hat{q}_n}{\de t} = \frac{\gamma(2^n)}{4} \Bigl( \frac12(1+\delta_1)\hat{q}_{n-1} - (1-\delta_1)\hat{q_n} \Bigr) & \text{for $n>n_0$,}\\
\hat{q}_n(0) = 4\delta_0^{3/2} & \text{for $n>n_0$,}\\
\hat{q}_{n_0}(t) = 4\delta_0^{3/2}e^{-\nu t} & \text{for $t\geq0$,}
\end{cases}
\end{equation}
then
\begin{equation} \label{hatqn2}
c_1\delta_0^{3/2} e^{-\nu t} \leq 2^{\bar{\theta}_2n}\hat{q}_n(t) \leq c_2\delta_0^{3/2} t^{-\bar{\theta}_2/\beta} e^{-\nu t}
\qquad\text{for all $n>n_0$.}
\end{equation}
\end{lemma}

\begin{proof}
Let $\sigma>0$ satisfy $2^{-\sigma}=\frac{1-\delta_1}{1+\delta_1}$. With the change of variables $w_n(t):=2^{(1-\sigma)n}\hat{q}_n(\frac{t}{1-\delta_1})$, the equation \eqref{hatqn1} becomes
\begin{equation} \label{hatqn3}
\frac{\de w_n}{\de t} = \frac{\gamma(2^n)}{4}\bigl( w_{n-1}-w_n \bigr).
\end{equation}
We can then express $w_n$ in terms of the fundamental solution to \eqref{hatqn3}, which has been computed in \cite[Lemma~A.3]{BNVd}:
\begin{equation*}
w_n(t) = \frac{\gamma(2^{n_0+1})}{4} \int_0^t \Psi_n^{(n_0+1)}(t-s)w_{n_0}(s)\de s + \sum_{\ell=n_0+1}^n \Psi_n^{(\ell)}(t)w_\ell(0), \qquad n\geq n_0+1.
\end{equation*}
From the explicit expression of $\Psi$ provided by \cite[Lemma~A.3]{BNVd} one can obtain an estimate of the form
\begin{equation*}
c_1 e^{-\frac{\gamma(2^\ell)}{4}t} \leq \Psi_n^{(\ell)}(t) \leq c_2 e^{-\frac{\gamma(2^\ell)}{4}t} \qquad (n\geq\ell\geq n_0)
\end{equation*}
for $c_2>c_1>0$ depending only on the fragmentation kernel $\gamma$; then, combining the previous estimate with the assumptions on $w_n(0)$, $w_{n_0}(t)$, we find for all $n\geq n_0+1$
\begin{equation*}
\begin{split}
w_n(t)
& \leq C_{n_0}\delta_0^{3/2}\int_0^t e^{-\frac{\gamma(2^{n_0+1})}{4}(t-s)} e^{-\frac{\nu s}{1-\delta_1}}\de s + 4c_2\delta_0^{3/2}\sum_{\ell=n_0+1}^n 2^{(1-\sigma)\ell}e^{-\frac{\gamma(2^\ell)}{4}t} \\
& \leq C_{n_0}\delta_0^{3/2}e^{-\frac{\nu t}{1-\delta_1}} + C\delta_0^{3/2} t^{-\frac{1-\sigma}{\beta}}e^{-\frac{\gamma(2^{n_0+1})}{8} t}
\leq C_{n_0}\delta_0^{3/2} t^{-\frac{1-\sigma}{\beta}}e^{-\frac{\nu t}{1-\delta_1}}
\end{split}
\end{equation*}
(the sum can be bounded by arguing as in \cite[equation~(A.43)]{BNVd}), and similarly
\begin{equation*}
w_n(t) \geq C'_{n_0}\delta_0^{3/2} e^{-\frac{\nu t}{1-\delta_1}}.
\end{equation*}
Going back to the function $\hat{q}_n$ with the change of variables, and choosing $\delta_1>0$ such that $1-\sigma=\bar{\theta}_2$, we obtain the estimate in the statement.
\end{proof}


\subsection{Conclusion} \label{subsect:conclusion}
We are now in a position to conclude the proof of the main result of the paper, by passing to the limit in the truncation parameter $N\to\infty$.

\begin{proof}[Proof of Theorem~\ref{thm:stability}]
For every sufficiently large $N\in\N$ we constructed in Section~\ref{subsect:truncation} a weak solution corresponding to the truncated initial datum and the truncated kernels, see \eqref{truncation1} and \eqref{truncation2}. These solutions exist for all positive times, remain supported in small intervals around the integers \eqref{truncation4}, and the corresponding sequence of masses $m^N(t)$ can be represented as in \eqref{fp} for all $t>0$, for suitable functions $t\mapsto(y^N(t),A^N(t))$.
Moreover, the sequences of first and second moments $p^N(t)$, $q^N(t)$ obey the estimates \eqref{decaypn1}--\eqref{decayqn1} for $t\in(0,\infty)$.
All the constants in the estimates are in particular independent of $N$.

For every $r>0$ we have the uniform bound
\begin{equation} \label{concl2}
\begin{split}
\int_{\R} 2^{rx}g^N(x,t)\de x
& \leq 2^{r\delta_0}\sum_{n\in\Z} 2^{rn}m^N_n(t) \\
& \xupref{fp}{\leq} 2^{r\delta_0}\bigl(1+\|y^N(t)\|_1\bigr)\sum_{n\in\Z} 2^{rn}\bar{m}_n(A^N(t),p^N(t)) \leq C_r
\end{split}
\end{equation}
for $C_r$ independent of $N$, in view of the asymptotics \eqref{nearlystat4} of $\bar{m}_n$ and of \eqref{fp3}. This in particular implies that, for every fixed $t$, the sequence of measures $\{g^N(\cdot,t)\}_N$ is tight, and hence relatively compact with respect to narrow convergence. Moreover, the family $\{g^N\}_N$ is equicontinuous, in the sense that for every $0<s<t<T$ and $\vphi\in C_{\mathrm b}(\R)$ we have, by using the weak formulation of the equation and the assumptions on the kernels,
\begin{equation*}
\begin{split}
\bigg| & \int_{\R}\vphi(x)g^N(x,t)\de x - \int_{\R}\vphi(x)g^N(x,s)\de x\bigg| \\
& \leq \frac{\ln2}{2}\sum_{k\leq N-1} \int_s^t \int_{I_k}\int_{I_k} K_N(2^y,2^z)g^N(y,\tau)g^N(z,\tau)\bigg| \vphi\Bigl(\frac{\ln(2^y+2^z)}{\ln2}\Bigr) - \vphi(y) - \vphi(z) \bigg| \de y \de z \de \tau\\
& \qquad + \frac14\sum_{k\leq N-1} \int_s^t\int_{I_k} \gamma_N(2^{y+1}) g^N(y+1,\tau) \big| \vphi(y+1) - 2\vphi(y) \big| \de y\de \tau \\
& \leq C\|\vphi\|_\infty\int_s^t \biggl( \sum_{k\leq0}2^{-k}\bigl(m^N_k(\tau)\bigr)^2 + \sum_{k>0}2^{\alpha k}\bigl(m^N_k(\tau)\bigr)^2 + \sum_{k\leq0} m^N_k(\tau) + \sum_{k>0}2^{\beta k}m^N_k(\tau) \biggr) \de\tau \\
& \leq C\|\vphi\|_\infty |t-s|,
\end{split}
\end{equation*}
with $C$ independent of $N$. Hence by Ascoli-Arzel\`a Theorem we can find a subsequence $N_j\to\infty$ such that the measures $g^{N_j}(\cdot,t)$ narrowly converge to some limit measure $g(\cdot,t)$ for every $t>0$, in the sense
\begin{equation} \label{concl1}
\int_{\R}\vphi(x)g^{N_j}(x,t)\de x \to\int_{\R}\vphi(x)g(x,t)\de x \qquad\text{for every }\vphi\in C_{\mathrm b}(\R).
\end{equation}

Next, thanks to \eqref{concl1} we can pass to the limit in the weak formulation \eqref{weakg} of the equation and obtain that $g$ is a weak solution, in the sense of Definition~\ref{def:weakg}, with initial datum $g_0$. Moreover $\supp g(\cdot,t)\subset\bigcup_{n\in\Z}I_n$. By defining $m_n(t)$, $p_n(t)$, $q_n(t)$ as in \eqref{mn}, \eqref{pn}, \eqref{qn} respectively, the convergence \eqref{concl1} implies
\begin{equation} \label{concl3}
m^{N_j}_n(t)\to m_n(t),
\quad
p^{N_j}_n(t)\to p_n(t),
\quad
q^{N_j}_n(t)\to q_n(t)
\qquad\text{as $j\to\infty$, for every $n\in\Z$.}
\end{equation}
In particular the sequences $p(t)$, $q(t)$ obey the bounds \eqref{decaypn1}--\eqref{decayqn1} (which are uniform in $N$), and in turn there exists $\rho\in[-\delta_0,\delta_0]$ such that
\begin{equation} \label{concl4}
p_n(t)\to\rho, \qquad q_n(t)\to0 \qquad\text{as $t\to\infty$, for all $n\in\Z$.}
\end{equation}

We next show that also the limit sequence $m_n(t)$ can be represented in terms of the coefficients $\bar{m}_n$. Indeed, by \eqref{fp3} we have that (up to further subsequences) $y^{N_j}_n(t)\to y_n(t)$ and $A^{N_j}(t)\to A(t)$ as $j\to\infty$, for some limit functions $y_n(t)$, $A(t)$. We deduce that
\begin{equation} \label{concl6}
\begin{split}
m_n(t)
& = \lim_{j\to\infty} m^{N_j}_n(t)
\xupref{fp}{=} \lim_{j\to\infty} \bar{m}_n(A^{N_j}(t),p^{N_j}(t)) \bigl(1+2^ny_n^{N_j}(t)\bigr) \\
& = \bar{m}_n(A(t),p(t))(1+2^ny_n(t)).
\end{split}
\end{equation}

We eventually pass to the limit as $t\to\infty$. We extract a subsequence $t_j\to\infty$ such that $g(\cdot,t_j)\to \mu$ in the sense of measures as $j\to\infty$, for some limit measure $\mu$. As the sequence $y_n(t)$ satisfies the bound \eqref{fp3}, we have $2^ny_n(t_j)\to0$ as $j\to\infty$, and we can further assume that there exists the limit $A_\infty:=\lim_{j\to\infty}A(t_j)$. Hence by \eqref{concl6} we have for every $n\in\Z$
\begin{equation} \label{concl5}
\lim_{j\to\infty}m_n(t_j)=\lim_{t\to\infty}\bar{m}_n(A(t_j),p(t_j))(1+2^ny_n(t_j)) = a_n(A_{\infty},\rho) .
\end{equation}
We claim that $A_\infty=A_{M,\rho}$: indeed we have in view of the conservation of mass of the solution $g(x,t)$ and by a Taylor expansion
\begin{equation*}
M= \sum_{n=-\infty}^\infty\int_{I_n}2^xg(x,t_j)\de x = \sum_{n=-\infty}^\infty \biggl( 2^{n+p_n(t_j)}m_n(t_j) + O(2^nm_n(t_j)q_n(t_j)) \biggr),
\end{equation*}
so that by passing to the limit as $j\to\infty$ and using \eqref{concl4}, \eqref{concl5} we find
\begin{equation*}
M= \sum_{n=-\infty}^\infty 2^{n+\rho}a_n(A_\infty,\rho),
\end{equation*}
which implies $A_\infty=A_{M,\rho}$ (see \eqref{peak2}).
Thanks to \eqref{concl4} we have for every $n\in\Z$
\begin{equation*}
\int_{I_n}(x-n-\rho)^2\de\mu(x) = \lim_{j\to\infty}\int_{I_n}(x-n-p_n(t_j))^2g(x,t_j)\de x = \lim_{j\to\infty} m_n(t_j)q_n(t_j)=0,
\end{equation*}
therefore $\supp\mu\subset\bigcup_{n\in\Z}\{n+\rho\}$, and $\mu=\sum_{n\in\Z}b_n\delta_{n+\rho}$ for suitable coefficients $b_n$. Moreover by \eqref{concl5}
\begin{equation*}
b_n = \lim_{j\to\infty} m_n(t_j)=a_n(A_{M,\rho},\rho),
\end{equation*}
and we conclude that the limit measure $\mu$ coincides with the stationary solution $g_p(A_{M,\rho},\rho)$. Finally, by uniqueness of the limit we also have that the full family of measures $g(\cdot,t)$ converges to $g_p(A_{M,\rho},\rho)$ as $t\to\infty$.
\end{proof}


\appendix
\section{Proof of the regularity result for the linearized problem} \label{sect:appendix}

We provide in this section the proof of the regularity result for the linearized problem \eqref{slinear1}, Theorem~\ref{thm:slinear}. The proof follows the same strategy	as that of Theorem~\ref{thm:linear}, given in \cite[Appendix~A]{BNVd}, which can be seen as a particular case of Theorem~\ref{thm:slinear}.

\begin{proof}[Proof of Theorem~\ref{thm:slinear}]
Along the proof, we will denote by $C$ a generic constant, possibly depending on the properties of the kernels, on $M$, and on $\bar{\theta}_1$, which might change from line to line. The estimate in the statement will be proved for the exponent $\nu>0$ given by
\begin{equation} \label{sprooflinear0}
\nu = \frac{1}{4c_0},
\end{equation}
where $c_0$ is the constant given by Lemma~\ref{lem:poincare} below. We divide the proof into several steps.

\smallskip\noindent\textit{Step 1.}
A maximum principle argument as in the first step of the proof of \cite[Lemma~A.1]{BNVd} can be applied also in this case, with minor changes, and shows, for a given initial datum $y^0\in\mathcal{Y}_\theta$, the existence and uniqueness of a solution $t\mapsto y(t)$ with $y(0)=y^0$ in the space $\mathcal{Y}_0$ for $\theta\geq0$ and in the space $\mathcal{Y}_\theta$ if $\theta<0$, satisfying in addition
\begin{equation} \label{sprooflinear1}
\begin{split}
\|y(t)\|_{0} &\leq 2\|y^0\|_\theta e^{\mu t} \qquad\text{(if $\theta\geq0$),}\\
\|y(t)\|_{\theta} &\leq 2\|y^0\|_\theta e^{\mu t} \qquad\text{(if $\theta<0$),}
\end{split}
\end{equation}
for some $\mu>0$ (we omit the details here).

The maximum principle also yields a uniform estimate on $y_n(t)$ in the region $n\geq N$, for a fixed $N$ sufficiently large, in terms of the initial values $y^0$ and on the boundary values $y_N(t)$. In order to obtain such an estimate, we again distinguish between the two cases $\theta\geq0$ and $\theta<0$. In the first case ($\theta\geq0$), the initial datum $y^0_n$ is bounded as $n\to\infty$, and we can directly use a comparison principle with the constant $\sup_{n\geq N}|y^0_n| + \sup_{0\leq s\leq t}|y_N(s)|$ (since the constants are solutions to \eqref{slinear1}). In the other case ($\theta<0$), we can compare with the sequence $2^{-\theta n}$, which is a supersolution to \eqref{slinear1} in the region $n\geq N$ for $N$ large enough (exploiting the fact that $\sigma_n\to0$ as $n\to\infty$). In conclusion, we find for every $t>0$
\begin{equation} \label{sprooflinear1bis}
\begin{split}
\sup_{n\geq N} |y_n(t)| &\leq \sup_{n\geq N}|y^0_n| + \sup_{0\leq s\leq t}|y_N(s)| \leq \|y^0\|_\theta + \sup_{0\leq s\leq t}|y_N(s)| \qquad \text{(if $\theta\geq0$),}\\
\sup_{n\geq N} 2^{\theta n}|y_n(t)| &\leq \sup_{n\geq N}2^{\theta n}|y^0_n| + \sup_{0\leq s\leq t}|y_N(s)| \leq \|y^0\|_\theta + \sup_{0\leq s\leq t}|y_N(s)| \qquad\text{(if $\theta<0$).}
\end{split}
\end{equation}

\smallskip\noindent\textit{Step 2.}
We will now prove a uniform decay estimate in bounded regions $n\in[-n_0,n_0]$, for $n_0\in\N$ sufficiently large.
To this aim, we introduce the quantities
\begin{equation} \label{sprooflinear2}
\bar{m}(t) := \frac{\sum_{n=-\infty}^\infty 2^{2n}\bar{m}_n(A_M,p(t)) y_n(t)}{\sum_{n=-\infty}^\infty 2^{2n}\bar{m}_n(A_M,p(t))} \,,
\qquad
I(t) := \sum_{n=-\infty}^\infty 2^{2n}\bar{m}_n(A_M,p)(y_n-\bar{m})^2 \,,
\end{equation}
where the coefficients $\bar{m}_n$ are defined in Lemma~\ref{lem:nearlystat}. In view of the rough bound \eqref{sprooflinear1}, and of the decay \eqref{nearlystat4} of the sequence $\bar{m}_n$, we easily obtain a uniform estimate for small times:
\begin{equation} \label{sprooflinear201}
|\bar{m}(t)| \leq \overline{C} \|y^0\|_\theta,
\qquad
|I(t)| \leq \overline{C} \|y^0\|_\theta^2
\qquad\text{for all $t\leq1$,}
\end{equation}
for a uniform constant $\overline{C}$.
We now compute the evolution equations for $\bar{m}(t)$ and $I(t)$ and show that these quantities decay exponentially to 0, by a Gr\"onwall-type argument.
In view of \eqref{sprooflinear201} we can restrict to times $t\geq1$, so that we do not have to take into account the time singularity $t^{-\bar{\theta}_1/\beta}$ in \eqref{asspn}. We preliminary notice that, by using the estimates \eqref{nearlystat6} and \eqref{asspn}, we find
\begin{equation} \label{sprooflinear202}
\begin{split}
\bigg| \frac{\de}{\de t}\Bigl[\bar{m}_n(A_M,p(t))\Bigr] \bigg|
& \leq \sum_{k=n}^\infty \bigg| \frac{\partial\bar{m}_n}{\partial p_k}(A_M,p(t)) \frac{\de p_k}{\de t} \bigg| \\
& \leq C\eta_0 e^{-\frac{\nu}{2}t} \bar{m}_n(A_M,p(t)) \biggl( \sum_{k=n\wedge 1}^0 2^{n-k} +\sum_{k=n\vee 1}^\infty 2^{n-k}2^{(\beta-\bar{\theta}_1)k} \biggr) \\
& \leq C\eta_0 e^{-\frac{\nu}{2}t} \bar{m}_n(A_M,p(t)) \max\bigl\{ 1,2^{(\beta-\bar{\theta}_1)n} \bigr\} .
\end{split}
\end{equation}
By elementary computations using the equation \eqref{slinear1}, the definition \eqref{slinear2} of $\sigma_n(t)$, and the relation \eqref{nearlystat3bis}, one can check that
\begin{equation*}
\begin{split}
\frac{\de\bar{m}(t)}{\de t}
= \frac{1}{\sum_{n=-\infty}^\infty 2^{2n}\bar{m}_n(A_M,p(t))} {\sum_{n=-\infty}^\infty 2^{2n} \frac{\de}{\de t}\Bigl[\bar{m}_n(A_M,p(t))\Bigr] \bigl(y_n(t)-\bar{m}(t)\bigr) },
\end{split}
\end{equation*}
hence in view of \eqref{sprooflinear202}, for all $t\geq1$,
\begin{equation} \label{sprooflinear203}
\begin{split}
\bigg| \frac{\de\bar{m}(t)}{\de t} \bigg|
& \leq C\eta_0 e^{-\frac{\nu}{2} t} \biggl( \sum_{n=-\infty}^0 2^{2n}\bar{m}_n\big|y_n-\bar{m}\big| + \sum_{n=1}^\infty 2^{(\beta-\bar{\theta}_1)n}2^{2n}\bar{m}_n\big|y_n-\bar{m}\big| \biggr) \\
& \leq C\eta_0 e^{-\frac{\nu}{2} t} \sqrt{I(t)}
\end{split}
\end{equation}
(where the last passage follows by H\"older inequality).
In order to obtain an evolution equation for $I(t)$, we observe that, recalling the notation \eqref{linearD} for the discrete derivatives, the equation \eqref{slinear1} can be written in the form
\begin{equation*}
2^{2n}\bar{m}_n(A_M,p(t)) \frac{\de y_n}{\de t} = D^-_n\Bigl( \bigl\{ 2^{2k}\gamma(2^{k+1+p_{k+1}(t)})\bar{m}_{k+1}(A_M,p(t)) D_k^+(y(t)) \bigr\}_k \Bigr);
\end{equation*}
from this identity, by subtracting $\bar{m}(t)$, multiplying by $(y_n(t)-\bar{m}(t))$, and summing over $n$, a straightforward computation yields
\begin{multline*}
\frac{\de I(t)}{\de t}
= - 2\sum_{n=-\infty}^\infty 2^{2n} \gamma(2^{n+1+p_{n+1}})\bar{m}_{n+1}(D^+_n(y))^2
+ \sum_{n=-\infty}^\infty 2^{2n} \frac{\de}{\de t}\Bigl[\bar{m}_n(A_M,p(t))\Bigr] (y_n-\bar{m})^2.
\end{multline*}
Then, by using the Poincar\'e-type inequality in Lemma~\ref{lem:poincare}, \eqref{sprooflinear202}, and \eqref{sprooflinear1bis}, we find
\begin{align} \label{sprooflinear207}
\bigg| \frac{\de I(t)}{\de t} \bigg|
& \leq -\frac{2}{c_0}I(t) + C \eta_0 e^{-\frac{\nu}{2} t} \sum_{n=-\infty}^{\infty} 2^{2n}\max\{ 1,2^{(\beta-\bar{\theta}_1)n} \}\bar{m}_n(A_M,p)(y_n-\bar{m})^2 \nonumber \\
& \leq -\Bigl(\frac{2}{c_0} - C\eta_0 e^{-\frac{\nu}{2} t} \Bigr) I(t) + C\eta_0 e^{-\frac{\nu}{2} t} \sum_{n=N+1}^\infty 2^{(\beta-\bar\theta_1)n}2^{2n}\bar{m}_n(A_M,p)(y_n-\bar{m})^2 \nonumber\\
& \leq -\frac{1}{c_0}I(t) + C\eta_0 e^{-\frac{\nu}{2} t} \Bigl( |\bar{m}(t)|^2 + \|y^0\|_\theta^2 + \sup_{0\leq s\leq t}|y_N(s)|^2 \Bigr),
\end{align}
provided that we choose $\eta_0$ sufficiently small. Hence, recalling the choice \eqref{sprooflinear0},
\begin{equation} \label{sprooflinear204}
\bigg| \frac{\de I(t)}{\de t} \bigg|
\leq -4\nu I(t) + C\eta_0 e^{-\frac{\nu}{2} t} \Bigl( \|y^0\|_\theta^2 + \sup_{0\leq s\leq t}|\bar{m}(s)|^2 + \sup_{0\leq s\leq t}|I(s)| \Bigr).
\end{equation}
Setting 
\begin{equation*}
\bar{t} := \sup\Bigl\{ t\geq1 \,:\,  |\bar{m}(s)| \leq 2\overline{C} \|y^0\|_\theta, \, |I(s)| \leq 2\overline{C} \|y^0\|_\theta^2 \text{ for all $s\in[1,t]$} \Bigr\},
\end{equation*}
we find that $\bar{t}=\infty$ by a standard continuation argument, using \eqref{sprooflinear201} and the two estimates \eqref{sprooflinear203}--\eqref{sprooflinear204} (which hold for $t\geq1$), and choosing $\eta_0$ small enough; hence
\begin{equation} \label{sprooflinear205}
|\bar{m}(t)| \leq 2\overline{C} \|y^0\|_\theta,
\qquad
|I(t)| \leq 2\overline{C} \|y^0\|_\theta^2
\qquad\text{for all $t>0$.}
\end{equation}

By inserting \eqref{sprooflinear205} into \eqref{sprooflinear204} we find by Gr\"onwall's inequality
\begin{equation} \label{sprooflinear206}
|I(t)| \leq C \|y^0\|_\theta^2 e^{-\frac{\nu}{4} t}
\qquad\text{for all $t>0$,}
\end{equation}
which in turn implies that for every $n_0\in\N$ there exists a constant $C_{n_0}$, depending on $n_0$, on the kernels, and on the fixed parameter $M$, such that
\begin{equation} \label{sprooflinear3}
\sup_{-n_0\leq n \leq n_0}|y_n(t)-\bar{m}(t)| \leq C_{n_0} \|y^0\|_\theta e^{-\frac{\nu}{8}t} \qquad\text{for every }t>0.
\end{equation}

Moreover, by \eqref{sprooflinear203} we have $|\frac{\de\bar{m}}{\de t}|\leq C\eta_0 \|y^0\|_\theta e^{-\frac{\nu}{2}t}$ for all $t\geq1$, which yields the existence of the limit $\bar{m}_\infty:=\lim_{t\to\infty}\bar{m}(t)$ and the estimate
\begin{equation} \label{sprooflinear24}
\big| \bar{m}(t_2)-\bar{m}(t_1) \big| \leq C \|y^0\|_\theta e^{-\frac{\nu}{2} t_1}  \quad\text{for all $t_2\geq t_1 \geq 1$.}
\end{equation}
In the rest of the proof, $C_{n_0}$ always denotes a constant depending on $n_0$, on the kernels, and on the fixed parameter $M$, possibly changing from line to line.

\smallskip\noindent\textit{Step 3.}
We now want to extend the estimate \eqref{sprooflinear3} in the region $n\leq-n_0$. We first notice that, thanks to \eqref{sprooflinear1} and \eqref{sprooflinear205}, we have for all $n\leq -n_0$
\begin{equation} \label{sprooflinear25}
|y_n(t)-\bar{m}_\infty| \leq C2^{-n}\|y^0\|_\theta \qquad\text{for all $t\leq1$.}
\end{equation}
We next extend \eqref{sprooflinear25} for times $t>1$. This can be achieved by a maximum principle argument, similar to the third step in the proof of \cite[Lemma~A.1]{BNVd}.
For fixed $T>1$ and $\e>0$ we consider the sequence
\begin{equation*}
z_n(t) := N 2^{-n}e^{-\frac{\nu}{8} t} + \e 4^{-n},
\end{equation*}
where $N>0$ is a constant to be fixed later. We first observe that
\begin{equation*}
\frac{\de z_n}{\de t} - \mathscr{L}_n(z;t) = N 2^{-n}e^{-\frac{\nu}{8} t} \Bigl( -\frac{\nu}{8} - \frac{\gamma(2^{n+p_n(t)})}{4} \bigl(1-\sigma_n(t)/2\bigr) \Bigr) - \e\frac{\gamma(2^{n+p_n(t)})}{4^{n+1}}\Bigl(3-\frac34\sigma_n(t)\Bigr).
\end{equation*}
Recalling the asymptotics \eqref{kernel5bis} and \eqref{slinear8} as $n\to-\infty$ and taking $\eta_0$ sufficiently small, assuming without loss of generality that $\nu<2\gamma_0$, we obtain that $z_n$ is a supersolution for \eqref{slinear1} in the region $n\in(-\infty,-n_0]$, for every sufficiently large $n_0$. Furthermore, for $t=1$
\begin{equation*}
|y_n(1)-\bar{m}_\infty| \leq C2^{-n}\|y^0\|_{\theta} < N 2^{-n} \leq z_n(0) \qquad\text{for all }n\leq -n_0,
\end{equation*}
provided that we choose $N>C\|y^0\|_{\theta}$. By \eqref{sprooflinear3} and \eqref{sprooflinear24}, for $n=-n_0$ we have
\begin{equation*}
|y_{-n_0}(t)-\bar{m}_\infty| \leq (C_{n_0}+C)\|y^0\|_{\theta}e^{-\frac{\nu}{8}t} \leq z_{-n_0}(t) \qquad\text{for every }t\geq1,
\end{equation*}
if we choose $N>2^{-n_0}(C_{n_0}+C) \|y^0\|_{\theta}$.
Finally, by \eqref{sprooflinear1} and \eqref{sprooflinear205} we can choose $n_1>n_0$ sufficiently large, depending on $\e$ and $T$, such that
\begin{equation*}
|y_{-n_1}(t)-\bar{m}_\infty| \leq (2^{n_1+1}e^{\mu T}+2\overline{C})\|y^0\|_{\theta} \leq \e4^{n_1} \leq z_{-n_1}(t) \qquad\text{for every }t\in[1,T].
\end{equation*}
Therefore, the choice $N > \max\{ C, 2^{-n_0}(C_{n_0}+C)\}\|y^0\|_{\theta}$, we can apply the Maximum Principle in the compact region $(n,t)\in[-n_1,-n_0]\times[1,T]$:
\begin{equation*}
|y_n(t)-\bar{m}_\infty| \leq z_n(t) \qquad\text{for all $n\in[-n_1,-n_0]$ and $t\in[1,T]$.}
\end{equation*}
Letting firstly $n_1\to\infty$, and then $\e\to0$, $T\to\infty$, the previous argument shows that
\begin{equation} \label{sprooflinear20}
|y_n(t)-\bar{m}_\infty| \leq N 2^{-n}e^{-\frac{\nu}{8} t} \qquad\text{for all $n\leq -n_0$ and $t\geq1$.}
\end{equation}

By combining the estimates \eqref{sprooflinear3}, \eqref{sprooflinear24}, \eqref{sprooflinear25}, and \eqref{sprooflinear20}, we obtain that for every $n_0\in\N$ sufficiently large there exists a constant $C_{n_0}$ such that
\begin{equation} \label{sprooflinear21}
|y_n(t)-\bar{m}_\infty| \leq C_{n_0} 2^{-n} \|y^0\|_\theta e^{-\frac{\nu}{8} t} \qquad\text{for all $n\leq n_0$ and $t>0$.}
\end{equation}

\smallskip\noindent\textit{Step 4.}
We eventually investigate the behaviour of solutions to \eqref{slinear1} as $n\to\infty$. We therefore now restrict to the region $n\geq n_0$, where $n_0\in\N$ is a sufficiently large constant. In particular, we are allowed to use the asymptotics \eqref{kernel5}, and we will always assume without loss of generality that
\begin{equation} \label{sprooflinear30}
\gamma(2^n)<\gamma(2^{n+1}),
\qquad
\frac{1}{2} \, 2^{\beta(n-m)} \leq \frac{\gamma(2^n)}{\gamma(2^m)} \leq \frac32 \, 2^{\beta(n-m)}
\qquad
\text{for all $n,m\geq n_0$.}
\end{equation}
For $\ell\in\Z$, $\ell\geq n_0$, let $\Psi^{(\ell)}_n$ be the solution to the problem
\begin{equation} \label{linearfund1}
\begin{cases}
\frac{\de \Psi_n^{(\ell)}}{\de t} = \frac{2^{\beta n}}{4}\bigl( \Psi_{n-1}^{(\ell)} - \Psi_n^{(\ell)} \bigr) ,\\
\Psi_n^{(\ell)}(0) = \delta(n-\ell).
\end{cases}
\end{equation}
The functions $\Psi_n^{(\ell)}$ have been explicitly computed in \cite[Lemma~A.3]{BNVd}.
As a particular case of \cite[Lemma~A.3]{BNVd}, there exists a uniform constant $c>0$ such that
\begin{equation} \label{linearfund2}
\big| \Psi_n^{(\ell)}(t) - \Psi_{n+1}^{(\ell)}(t) \big| \leq c 2^{-\beta(n-\ell)}e^{-\frac{2^{\beta\ell}}{4} t} \qquad\text{for all $n\geq\ell\geq n_0$.}
\end{equation}
In particular, there exists the limit $\Psi^{(\ell)}_\infty(t):=\lim_{n\to\infty}\Psi^{(\ell)}_n(t)$, which satisfies
\begin{equation} \label{linearfund3}
\big| \Psi_n^{(\ell)}(t) - \Psi_{\infty}^{(\ell)}(t) \big| \leq c 2^{-\beta(n-\ell)}e^{-\frac{2^{\beta\ell}}{4} t} \qquad\text{for all $n\geq\ell\geq n_0$.}
\end{equation}

By means of the fundamental solutions $\Psi_n^{(\ell)}$ we can write a representation formula for the solution to \eqref{slinear1} in the region $n\geq n_0$ in terms of the initial values $y_n^0$ and of the values of the solution for $n=n_0-1$. More precisely, we denote by $p_\infty(t):=\lim_{n\to\infty}p_n(t)$ (which exists by \eqref{asspn}), and we solve the initial/boundary value problem
\begin{equation} \label{sprooflinear31}
\begin{cases}
\frac{\de y_n}{\de t} = \frac{2^{\beta(n+p_\infty(t))}}{4} \bigl( y_{n-1}-y_n \bigr) + r_n(t) & n\geq n_0, \\
y_n(0) = y_n^0 & n\geq n_0,\\
y_{n_0-1}(t) = \lambda(t) & t>0,
\end{cases}
\end{equation}
where $r_n(t):=r_n^{(1)}(t) + r_n^{(2)}(t)$,
\begin{align*}
r_n^{(1)}(t) &:= \frac14\Bigl[ \gamma(2^{n+p_n(t)}) - 2^{\beta(n+p_\infty(t))} \Bigr] \bigl( y_{n-1}-y_n \bigr), \\
r_n^{(2)}(t) &:= -\frac{\gamma(2^{n+p_n(t)})}{4}\sigma_n(t)\bigl(y_n(t)-y_{n+1}(t)\bigr).
\end{align*}
By rescaling the time variable, that is by introducing $\tau:=\int_0^t 2^{\beta p_\infty(s)}\de s$ and $\tilde{y}_n(\tau):=y_n(t)$ (and defining similarly $\tilde{p}_n(\tau)$, $\tilde{\lambda}(\tau)$, $\tilde{r}_n(\tau)$), the problem \eqref{sprooflinear31} takes the form
\begin{equation} \label{sprooflinear32}
\begin{cases}
\frac{\de \tilde{y}_n}{\de\tau} = \frac{2^{\beta n}}{4} \bigl( \tilde{y}_{n-1}-\tilde{y}_n \bigr) + 2^{-\beta \tilde{p}_\infty(\tau)}\tilde{r}_n(\tau) & n\geq n_0, \\
\tilde{y}_n(0) = y_n^0 & n\geq n_0,\\
\tilde{y}_{n_0-1}(\tau) = \tilde\lambda(\tau) & \tau>0.
\end{cases}
\end{equation}
Notice that $c_1t\leq \tau \leq c_2t$ for positive constants $c_2>c_1>0$.
By Duhamel's Principle we can write the solution to \eqref{sprooflinear32}, for all $n\geq n_0$, as
\begin{equation} \label{sprooflinear33}
\begin{split}
\tilde{y}_n(\tau)
& = \frac{2^{\beta n_0}}{4}\int_0^\tau \Psi_n^{(n_0)}(\tau-s)\tilde\lambda(s)\de s + \sum_{\ell=n_0}^n \Psi_n^{(\ell)}(\tau)y_\ell^0 \\
& \qquad\qquad + \int_0^\tau \sum_{\ell=n_0}^n \Psi_n^{(\ell)}(\tau-s) 2^{-\beta \tilde{p}_\infty(s)}\tilde{r}_\ell(s)\de s \,.
\end{split}
\end{equation}
Using the representation formula \eqref{sprooflinear33}, together with the fact that $\frac{2^{\beta\ell}}{4}\int_0^\infty \Psi^{(\ell)}_n(s)\de s =1$ for every $n\geq\ell$ (see \cite[(A.26)]{BNVd}),  we can write the difference between $\tilde{y}_n(\tau)$ and $\tilde{y}_{n+1}(\tau)$, $n\geq n_0$, as follows:
\begin{align}\label{sprooflinear33b}
\tilde{y}_n(\tau) & - \tilde{y}_{n+1}(\tau)
= \frac{2^{\beta n_0}}{4} \int_0^\tau \bigl( \Psi_n^{(n_0)}-\Psi_{n+1}^{(n_0)} \bigr) (\tau-s) ( \tilde\lambda(s)-\bar{m}_\infty) \de s \nonumber \\
& - \bar{m}_\infty \frac{2^{\beta n_0}}{4} \int_\tau^\infty \bigl( \Psi_n^{(n_0)}-\Psi_{n+1}^{(n_0)} \bigr) (s)\de s
+ \sum_{\ell=n_0}^{n+1} \bigl( \Psi_n^{(\ell)}(\tau)-\Psi_{n+1}^{(\ell)}(\tau) \bigr) y_\ell^0 \nonumber \\
& + \int_0^\tau \sum_{\ell=n_0}^{n+1}  \bigl( \Psi_n^{(\ell)}-\Psi_{n+1}^{(\ell)} \bigr) (\tau-s) 2^{-\beta\tilde{p}_\infty(s)} \tilde{r}_\ell(s) \de s
=: I_1 + I_2 + I_3 + I_4.
\end{align}
We now estimate separately each term on the right-hand side of \eqref{sprooflinear33b}. For a fixed $\tilde{\theta}\in[\theta,\beta]$ as in the statement, we set $w(\tau):=\sup_{n\geq n_0} 2^{\tilde{\theta} n}|\tilde{y}_n(\tau)-\tilde{y}_{n+1}(\tau)|$. It is also convenient to introduce the constant $\Lambda:=\frac{2^{\beta n_0}}{8}$.

Notice that, by \eqref{sprooflinear21}, we have
\begin{equation*}
|\tilde\lambda(\tau)-\bar{m}_\infty| = |y_{n_0-1}(t)-\bar{m}_\infty| \leq C_{n_0}\|y^0\|_\theta e^{-\frac{\nu}{8} \tau} .
\end{equation*}
Combining this estimate with \eqref{linearfund2} we obtain
\begin{equation} \label{sprooflinear34a}
2^{\tilde{\theta}n} |I_1| \leq C_{n_0}\|y^0\|_\theta 2^{(\tilde\theta-\beta) n}\int_0^\tau e^{-\Lambda(\tau-s)}e^{-\frac{\nu}{8} s}\de s \leq C_{n_0}\|y^0\|_\theta e^{-\frac{\nu}{8}\tau}.
\end{equation}
For the second term $I_2$, using again \eqref{linearfund2} and \eqref{sprooflinear205} we have
\begin{equation} \label{sprooflinear34b}
2^{\tilde{\theta}n} |I_2| \leq C_{n_0}\|y^0\|_\theta 2^{(\tilde\theta-\beta) n} \int_\tau^\infty e^{-\Lambda s}\de s \leq C_{n_0}\|y^0\|_\theta e^{-\Lambda\tau}.
\end{equation}
The following estimate is proved in \cite[(A.44)]{BNVd}:
\begin{equation} \label{sprooflinear34c}
2^{\tilde{\theta}n} |I_3|
\leq C \|y^0\|_\theta 2^{(\tilde\theta-\beta) n} \sum_{\ell=n_0}^{n+1} 2^{(\beta-\theta)\ell}e^{-\frac{2^{\beta\ell}}{4}\tau}
\leq C\|y^0\|_\theta \bigl( 1+\tau^{-\frac{\tilde{\theta}-\theta}{\beta}}\bigr) e^{-\Lambda\tau}.
\end{equation}
To bound the term $I_4$ containing the remainder $\tilde{r}_\ell$, we need some preliminary estimates. In view \eqref{asspn} we have for all $\bar\theta\in(0,\bar{\theta}_1]$, by interpolation,
\begin{equation*}
|\tilde{p}_\ell(\tau)-\tilde{p}_\infty(\tau)|
\leq \sum_{j=\ell}^\infty |\tilde{p}_j(\tau)-\tilde{p}_{j+1}(\tau)| 
\leq \sum_{j=\ell}^\infty \Bigl( \eta_02^{-\bar{\theta}_1 j}\tau^{-\frac{\bar{\theta}_1}{\beta}} \Bigl)^{\frac{\bar{\theta}}{\bar{\theta}_1}} \eta_0^{1-\frac{\bar{\theta}}{\bar{\theta}_1}}
\leq C_{\bar{\theta}}\eta_02^{-\bar{\theta}\ell}\tau^{-\frac{\bar{\theta}}{\beta}}.
\end{equation*}
Then, using also \eqref{kernel5} and \eqref{slinear8}, 
\begin{align*}
2^{-\beta \tilde{p}_\infty(\tau)}|\tilde{r}_\ell^{(1)}(\tau)|
& \leq C \Bigl( |\gamma(2^{\ell+\tilde{p}_\ell(\tau)})-2^{\beta(\ell+\tilde{p}_\ell(\tau))}| + 2^{\beta\ell}|2^{\beta\tilde{p}_\ell(\tau)}-2^{\beta\tilde{p}_\infty(\tau)}| \Bigr) |\tilde{y}_{\ell-1}(\tau)-\tilde{y}_{\ell}(\tau)| \\
& \leq C \Bigl( 2^{\tilde{\beta}\ell} + 2^{\beta\ell}|\tilde{p}_\ell(\tau)-\tilde{p}_\infty(\tau)| \Bigr) |\tilde{y}_{\ell-1}(\tau)-\tilde{y}_{\ell}(\tau)| \\
& \leq C_{\bar{\theta}} \Bigl( 2^{\tilde{\beta}\ell} + \eta_0 2^{(\beta-\bar{\theta})\ell} \tau^{-\frac{\bar{\theta}}{\beta}} \Bigr) |\tilde{y}_{\ell-1}(\tau)-\tilde{y}_{\ell}(\tau)| , \\
2^{-\beta \tilde{p}_\infty(\tau)}|\tilde{r}_\ell^{(2)}(\tau)| &\leq C 2^{\beta\ell}e^{-A_M2^\ell} |\tilde{y}_\ell(\tau)-\tilde{y}_{\ell+1}(\tau)| .
\end{align*}
In turn we obtain, recalling \eqref{linearfund2},
\begin{align*}
2^{\tilde{\theta}n}|I_4|
& \leq C 2^{(\tilde{\theta}-\beta) n} \int_0^\tau \sum_{\ell=n_0}^{n+1} 2^{\beta\ell}e^{-\frac{2^{\beta\ell}}{4}(\tau-s)} \Bigl[ 2^{\tilde{\beta}\ell} + 2^{(\beta-\bar{\theta})\ell}s^{-\frac{\bar{\theta}}{\beta}} \Bigr] |\tilde{y}_{\ell-1}(s)-\tilde{y}_{\ell}(s)| \de s \nonumber \\
& \qquad + C 2^{(\tilde{\theta}-\beta) n} \int_0^\tau \sum_{\ell=n_0}^{n+1} 2^{2\beta\ell}e^{-\frac{2^{\beta\ell}}{4}(\tau-s)} e^{-A_M2^\ell} |\tilde{y}_\ell(s)-\tilde{y}_{\ell+1}(s)| \de s \nonumber \\
& \leq C_{n_0} 2^{(\tilde{\theta}-\beta) n} \int_0^\tau e^{-\Lambda(\tau-s)} \Bigl[ 1 + s^{-\frac{\bar{\theta}}{\beta}} \Bigr] |\tilde{y}_{n_0-1}(s)-\tilde{y}_{n_0}(s)| \de s \nonumber \\
& \qquad + C 2^{(\tilde{\theta}-\beta) n} \int_0^\tau \sum_{\ell=n_0+1}^{n+1} 2^{(\beta-\tilde{\theta})\ell}e^{-\frac{2^{\beta\ell}}{4}(\tau-s)} \Bigl[ 2^{\tilde{\beta}\ell} + 2^{(\beta-\bar{\theta})\ell}s^{-\frac{\bar{\theta}}{\beta}} \Bigr] w(s) \de s \nonumber \\
& \qquad + C 2^{(\tilde{\theta}-\beta) n} \int_0^\tau \biggl( \sum_{\ell=n_0}^{n+1} 2^{(2\beta-\tilde{\theta})\ell} e^{-A_M2^\ell} \biggr) e^{-\Lambda(\tau-s)}w(s) \de s \,. \end{align*}
By using \eqref{sprooflinear21} in the first term, and an estimate similar to \eqref{sprooflinear34c} in the second integral, we end up with
\begin{align} \label{sprooflinear34d}
2^{\tilde{\theta}n}|I_4|
& \leq C_{n_0} \|y^0\|_\theta \int_0^\tau e^{-\Lambda(\tau-s)} \Bigl[ 1 + s^{-\frac{\bar{\theta}}{\beta}}\Bigr] e^{-\frac{\nu}{8} s} \de s \nonumber \\
& \quad + C \int_0^\tau \sum_{\ell=n_0+1}^{n+1} e^{-\frac{2^{\beta\ell}}{4}(\tau-s)} \Bigl[ 2^{\tilde{\beta}\ell} + 2^{(\beta-\bar{\theta})\ell}s^{-\frac{\bar{\theta}}{\beta}} \Bigr] w(s) \de s
+ C \int_0^\tau e^{-\Lambda(\tau-s)}w(s) \de s \nonumber\\
& \leq C_{n_0} \|y^0\|_\theta e^{-\frac{\nu}{8}\tau} + C \int_0^\tau \bigl( 1+ (\tau-s)^{-\frac{\tilde{\beta}}{\beta}} \bigr) e^{-\Lambda(\tau-s)}w(s)\de s \nonumber \\
& \quad + C \int_0^\tau \bigl(1+(\tau-s)^{-\frac{\beta-\bar{\theta}}{\beta}} \bigr) s^{-\frac{\bar{\theta}}{\beta}} e^{-\Lambda(\tau-s)}w(s)\de s
+ C \int_0^\tau e^{-\Lambda(\tau-s)}w(s) \de s .
\end{align}

By inserting \eqref{sprooflinear34a}, \eqref{sprooflinear34b}, \eqref{sprooflinear34c}, \eqref{sprooflinear34d} into \eqref{sprooflinear33b} we find
\begin{multline} \label{sprooflinear35}
w(\tau) \leq C_{n_0}\|y^0\|_\theta \biggl( e^{-\frac{\nu}{8}\tau} + e^{-\Lambda\tau} + \tau^{-\frac{\tilde{\theta}-\theta}{\beta}} e^{-\Lambda\tau} \biggr)
+ C \int_0^\tau \bigl( 1+ (\tau-s)^{-\frac{\tilde{\beta}}{\beta}} \bigr) e^{-\Lambda(\tau-s)}w(s)\de s \\
+ C \int_0^\tau \bigl( 1+ (\tau-s)^{-\frac{\beta-\bar{\theta}}{\beta}} \bigr) s^{-\frac{\bar{\theta}}{\beta}} e^{-\Lambda(\tau-s)} w(s)\de s
+ C \int_0^\tau e^{-\Lambda(\tau-s)}w(s) \de s .
\end{multline}
Thanks to this estimate, the exponential-in-time decay of $w(t)$ can be obtained by means of a Gr\"onwall-type argument. We first consider small times $0<\tau<1$: in this case \eqref{sprooflinear35} becomes
\begin{equation*}
w(\tau) \leq C_{n_0} \|y^0\|_\theta \tau^{-\frac{\tilde{\theta}-\theta}{\beta}} + C \int_0^\tau \Bigl( 1 +  (\tau-s)^{-\frac{\tilde{\beta}}{\beta}} + s^{-\frac{\bar{\theta}}{\beta}} + (\tau-s)^{-\frac{\beta-\bar{\theta}}{\beta}}s^{-\frac{\bar{\theta}}{\beta}} \Bigr)w(s)\de s \,,
\end{equation*}
hence
\begin{equation} \label{sprooflinear36}
w(\tau) \leq C_{n_0}\|y^0\|_\theta \tau^{-\frac{\tilde{\theta}-\theta}{\beta}} \qquad\text{for all $0<\tau<1$.}
\end{equation}
For $\tau\geq1$, we set $\Psi(\tau):=\sup_{1\leq s\leq \tau}e^{\frac{\nu}{8}s}w(s)$. Choosing $\bar{\theta}\in(0,\beta)$ such that $\frac{\tilde{\theta}-\theta+\bar{\theta}}{\beta}<1$, we obtain from \eqref{sprooflinear35}
\begin{equation*}
\begin{split}
\Psi(\tau)
&\leq C_{n_0}\|y^0\|_\theta \\
& \qquad + C_{n_0}\|y^0\|_\theta\sup_{1\leq s\leq\tau} e^{\frac{\nu}{8} s}\int_0^1\Bigl[ 1+(s-r)^{-\frac{\tilde{\beta}}{\beta}} + r^{-\frac{\bar{\theta}}{\beta}}  + (s-r)^{-\frac{\beta-\bar{\theta}}{\beta}} r^{-\frac{\bar{\theta}}{\beta}} \Bigr] e^{-\Lambda(s-r)} r^{-\frac{\tilde{\theta}-\theta}{\beta}}\de r \\
& \qquad + C \Psi(\tau) \sup_{1\leq s \leq\tau} \int_1^s \Bigl[ 1+(s-r)^{-\frac{\tilde{\beta}}{\beta}} + (s-r)^{-\frac{\beta-\bar{\theta}}{\beta}} \Bigr] e^{-(\Lambda-\nu/8)(s-r)}\de r \\
& \leq C_{n_0}\|y^0\|_\theta + C\Psi(\tau) \int_0^\infty \Bigl[ 1+ t^{-\frac{\tilde{\beta}}{\beta}} + t^{-\frac{\beta-\bar{\theta}}{\beta}} \Bigr] e^{-(\Lambda-\nu/8)t}\de t \\
& \leq C_{n_0}\|y^0\|_\theta + C\Bigl( \frac{1}{\Lambda-\nu/8} + \frac{\Gamma(1-\frac{\tilde{\beta}}{\beta})}{(\Lambda-\nu/8)^{1-\tilde{\beta}/\beta}} + \frac{\Gamma(\frac{\bar{\theta}}{\beta})}{(\Lambda-\nu/8)^{\bar{\theta}/\beta}} \Bigr) \Psi(\tau) .
\end{split}
\end{equation*}
Recalling that $\Lambda=\frac{2^{\beta n_0}}{8}$, by choosing $n_0$ sufficiently large we obtain $\Psi(\tau)\leq C_{n_0}\|y^0\|_\theta$ and, in turn, $w(t)\leq C_{n_0}\|y^0\|_\theta e^{-\frac{\nu}{8}\tau}$ for all $\tau\geq1$. Combining this estimate with \eqref{sprooflinear36} we conclude
\begin{equation*}
w(\tau) \leq C_{n_0}\|y^0\|_\theta \bigl( 1+\tau^{-\frac{\tilde{\theta}-\theta}{\beta}} \bigr) e^{-\frac{\nu}{8}\tau} \qquad\text{for all $\tau>0$.}
\end{equation*}
By rescaling the time and going back to the original variables, and bearing in mind the estimate \eqref{sprooflinear21},
\begin{equation}  \label{sprooflinear37}
\| y_n(t)-y_{n+1}(t)\|_{\tilde{\theta}} \leq C_{n_0} \|y^0\|_\theta \bigl( 1+ t^{-\frac{\tilde{\theta}-\theta}{\beta}}\bigr) e^{-\frac{\nu}{8} t}.
\end{equation}

\smallskip\noindent\textit{Step 5.}
The last step of the proof consists in improving the exponent of the exponential in \eqref{sprooflinear37}, in order to get the optimal decay. In order to do this, we go back to Step 2 and we observe that, thanks to \eqref{sprooflinear3} and \eqref{sprooflinear37}, for all $t\geq1$ and $n\geq n_0$
\begin{equation*}
|y_n(t)-\bar{m}(t)| \leq |y_{n_0}(t)-\bar{m}(t)| + \sum_{j=n_0}^{n-1}|y_j(t)-y_{j+1}(t)| \leq C_{n_0}\|y^0\|_\theta e^{-\frac{\nu}{8}t}.
\end{equation*}
Then inserting this estimate into \eqref{sprooflinear207}
\begin{align*}
\bigg| \frac{\de I(t)}{\de t} \bigg|
& \leq -4\nu I(t) + C\eta_0 e^{-\frac{\nu}{2} t} \sum_{n=n_0+1}^\infty 2^{(\beta-\bar\theta_1)n}2^{2n}\bar{m}_n(A_M,p)(y_n-\bar{m})^2 \nonumber\\
& \leq -4\nu I(t) + C_{n_0}\eta_0\|y^0\|_\theta^2 e^{-\frac{\nu}{4}t} e^{-\frac{\nu}{2} t},
\end{align*}
so that Gr\"onwall inequality yields $|I(t)|\leq C\|y^0\|_\theta^2e^{-\frac{\nu}{2}t}$. Hence we have improved the exponent in the estimate \eqref{sprooflinear206}, and repeating the steps 3 and 4 we obtain that \eqref{sprooflinear37} holds with $\frac{\nu}{4}$ in place of $\frac{\nu}{8}$. By iterating this argument, we eventually obtain the desired decay
\begin{equation}  \label{sprooflinear38}
\| y_n(t)-y_{n+1}(t)\|_{\tilde{\theta}} \leq C_{n_0} \|y^0\|_\theta t^{-\frac{\tilde{\theta}-\theta}{\beta}}e^{-\nu t},
\end{equation}
that is \eqref{slinear6b} This also shows the existence of the limit $y_\infty(t):=\lim_{n\to\infty}y_n(t)$ and the estimate \eqref{slinear6}.
\end{proof}

The following discrete Poincar\'e-type inequality is used in the first step of the proof of Theorem~\ref{thm:slinear}.

\begin{lemma}\label{lem:poincare}
With the notation introduced in the proof of Theorem~\ref{thm:slinear}, there exists a constant $c_0>0$ (depending on $M$) such that for every $t>0$
\begin{equation*}
\sum_{n=-\infty}^\infty 2^{2n}\bar{m}_n(A_M,p) (y_n-\bar{m})^2 \leq c_0 \sum_{n=-\infty}^\infty 2^{2n} \gamma(2^{n+1+p_{n+1}})\bar{m}_{n+1}(A_M,p)(D^+_n(y))^2 \,.
\end{equation*}
\end{lemma}

\begin{proof}
The proof can be obtained by adapting the corresponding result in \cite[Lemma~A.2]{BNVd}, with minor changes.
\end{proof}


\bigskip
\bigskip
\noindent
{\bf Acknowledgments.}
The authors acknowledge support through the CRC 1060 \textit{The mathematics of emergent effects} at the University of Bonn that is funded through the German Science Foundation (DFG).

\bibliographystyle{siam}
\bibliography{Bibliography}

\end{document}